\documentclass{amsart}

\usepackage{amsmath}
\usepackage{amsthm}
\usepackage{amsfonts}
\usepackage{amssymb}

\newcommand\R{{\mathbb{R}}}
\newcommand\C{{\mathbb{C}}}

\newcommand\D{{\mathbf{D}}}

\newcommand\I{{\mathbf{I}}}
\renewcommand\P{{\mathbf{P}}}
\newcommand\E{{\mathbf{E}}}

\renewcommand\Im{{\operatorname{Im}}}
\renewcommand\Re{{\operatorname{Re}}}
\newcommand\eps{{\varepsilon}}

\newcommand\tr{\operatorname{trace}}
\newcommand\dist{\operatorname{dist}}

\newcommand\CE{{\mathcal E}}

\subjclass{15A52}

\theoremstyle{plain}
  \newtheorem{theorem}{Theorem}

  \newtheorem{proposition}[theorem]{Proposition}
  
  \newtheorem{lemma}[theorem]{Lemma}
  \newtheorem{corollary}[theorem]{Corollary}

\theoremstyle{definition}
  \newtheorem{definition}[theorem]{Definition}
  \newtheorem{example}[theorem]{Example}
  \newtheorem{remark}[theorem]{Remark}

\include{psfig}

\begin{document}

\title[Random matrices: universality of local eigenvalue statistics]{Random matrices:\\  Universality of local eigenvalue statistics}

\author{Terence Tao}
\address{Department of Mathematics, UCLA, Los Angeles CA 90095-1555}
\email{tao@@math.ucla.edu}
\thanks{T. Tao is supported by a grant from the MacArthur Foundation, by NSF grant DMS-0649473, and by the NSF Waterman award.}

\author{Van Vu}
\address{Department of Mathematics, Rutgers, Piscataway, NJ 08854}
\email{vanvu@math.rutgers.edu-}
\thanks{V. Vu is supported by research grants DMS-0901216 and AFOSAR-FA-9550-09-1-0167.}

\begin{abstract}  In this paper, we consider the universality of  the local eigenvalue statistics of random  matrices. 
Our main result shows that these statistics are determined by the first four moments of the distribution of the entries. 
As a consequence, we derive  the universality of eigenvalue gap distribution and $k$-point correlation
and many other statistics 
(under some mild assumptions) for both Wigner Hermitian matrices and Wigner real symmetric matrices. 
\end{abstract}

\maketitle

\setcounter{tocdepth}{2}
\tableofcontents

\section{Introduction}

\subsection{Wigner matrices and local statistics}

The goal of this paper is to establish a universality property for the local eigenvalue statistics
for random matrices. To simplify the presentation, we are going to focus on \emph{Wigner Hermitian matrices}, which are perhaps the
  most prominent model in the field.  We emphasize however that our main theorem (Theorem \ref{theorem:main})
 is stated in a much more general setting, and can be applied to various other models of random matrices (such
as random real symmetric matrices, for example).

\begin{definition}[Wigner matrices]\label{def:Wignermatrix} 
Let $n$ be a large number. A \emph{Wigner Hermitian matrix} (of size $n$)
 is defined as  a random Hermitian $n \times n$ matrix $M_n$ with upper
triangular complex entries $\zeta_{ij} := \xi_{ij} + \sqrt{-1} \tau_{ij}$ ($1\le i < j
\le n$) and diagonal real entries $\xi_{ii}$
($1\le i \le n$) where

\begin{itemize}
\item For $1 \leq i < j \leq n$, $\xi_{ij}, \tau_{ij}$ are iid copies of a real random variable $\xi$ with  mean zero and variance $1/2$.

\item For $1 \leq i \leq n$, $\xi_{ii} $ are iid copies of a real random variable $\tilde \xi$
with mean zero and  variance $1$.

\item $\xi, \tilde \xi$ have exponential decay, i.e., there are constants $C, C'$ such that 
$\P( |\xi| \ge t^C) \le \exp(- t),   \P( |\tilde \xi| \ge t^C) \le \exp(- t)$, for all $t \ge C'$.
\end{itemize}
We refer to $\xi, \tilde \xi$ as the \emph{atom distributions} of $M_n$, and $\xi_{ij}, \tau_{ij}$ as the \emph{atom variables}.  We refer to the matrix $W_n := \frac{1}{\sqrt{n}} M_n$ as the \emph{coarse-scale normalized Wigner Hermitian matrix}, and $A_n := \sqrt{n} M_n$ as the \emph{fine-scale normalized Wigner Hermitian matrix}.
\end{definition}

\begin{example}\label{gue}  An important special case of a Wigner Hermitian matrix is the \emph{gaussian unitary ensemble} (GUE), in which $\xi, \tilde \xi$ are gaussian random variables with mean zero and variance $1/2$, $1$ respectively.  The coarse-scale normalization $W_n$ is convenient for placing all the eigenvalues in a bounded interval, while the fine-scale normalization $A_n$ is convenient for keeping the spacing between adjacent eigenvalues to be roughly of unit size.
\end{example}


Given an $n \times n$ Hermitian matrix $A$, we denote its $n$ eigenvalues as
$$ \lambda_1(A) \leq \ldots \leq \lambda_n(A),$$
and write $\lambda(A) := (\lambda_1(A),\ldots,\lambda_n(A))$.  We also let $u_1(A),\ldots,u_n(A) \in \C^n$ be an orthonormal basis of eigenvectors of $A$ with $A u_i(A) = \lambda_i(A) u_i(A)$; these eigenvectors $u_i(A)$ are only determined up to a complex phase even when the eigenvalues are simple, but this ambiguity will not cause a difficulty in our results as we will only be interested in the \emph{magnitude} $|u_i(A)^* X|$ of various inner products $u_i(A)^* X$ of $u_i(A)$ with other vectors $X$.

The study of the eigenvalues $\lambda_i( W_n )$ of (normalized) Wigner Hermitian matrices has been one of the major topics of study in random matrix theory.  The properties of these eigenvalues are not only interesting in their own right, but also have been playing
essential roles in many other areas of mathematics, such as mathematical
physics, probability, combinatorics, and the theory of computing.

It will be convenient to introduce the following notation for frequent events depending on $n$, in increasing order of likelihood:

\begin{definition}[Frequent events]\label{freq-def}  Let $E$ be an event depending on $n$.
\begin{itemize}
\item $E$ holds \emph{asymptotically almost surely} if\footnote{See Section \ref{notation-sec} for our conventions on asymptotic notation.} $\P(E) = 1-o(1)$.
\item $E$ holds \emph{with high probability} if $\P(E) \geq 1-O(n^{-c})$ for some constant $c>0$.
\item $E$ holds \emph{with overwhelming probability} if $\P(E) \geq 1-O_C(n^{-C})$ for \emph{every} constant $C>0$ (or equivalently, that $\P(E) \geq 1 - \exp(-\omega(\log n))$).
\item $E$ holds \emph{almost surely} if $\P(E)=1$.  
\end{itemize}
\end{definition}

\begin{remark} Note from the union bound that the intersection of $O(n^{O(1)})$ many events with uniformly overwhelming probability, still has overwhelming probability.  Unfortunately, the same is not true for events which are merely of high probability, which will cause some technical difficulties in our arguments.
\end{remark}

A cornerstone of this theory is the \emph{Wigner semicircular law}.  Denote by $\rho_{sc}$ the semi-circle density function with
support on $[-2,2]$,
\begin{equation}\label{semi}
 \rho_{sc} (x):= \begin{cases} \frac{1}{2\pi} \sqrt {4-x^2}, &|x| \le 2 \\ 0,
&|x| > 2. \end{cases} 
\end{equation}

\begin{theorem}[Semi-circular law]\label{theorem:Wigner} Let $M_n$ be a Wigner Hermitian matrix.  Then for any real number $x$,
$$\lim_{n \rightarrow \infty} \frac{1}{n} |\{1 \leq i \leq n: \lambda_i(W_n) \le  x \}|
 = \int_{-2}^x \rho_{sc}(y)\ dy$$
in the sense of probability (and also in the almost sure sense, if the $M_n$ are all minors of the same infinite Wigner Hermitian matrix), where we use $|I|$ to denote the cardinality of a finite set $I$.  
\end{theorem}

\begin{remark} Wigner\cite{wig} proved this theorem for special ensembles. The general
 version  above is due to  Pastur \cite{Pas} (see \cite{BS, AGZ} for a detailed discussion).  The semi-circular law in fact holds under substantially more general hypotheses than those given in Definition \ref{def:Wignermatrix}, but we will not discuss this matter further here. One consequence of Theorem \ref{theorem:Wigner} is that we expect most of the eigenvalues of $W_n$ to lie in the interval $(-2+\eps,2+\eps)$ for $\eps > 0$ small; we shall thus informally refer to this region as the \emph{bulk} of the spectrum.
\end{remark} 

Several stronger versions of Theorem \ref{theorem:Wigner} are known.  For instance, it is known (see e.g. \cite{baiyin}, \cite{baiyin2}) that asymptotically almost surely, one has
\begin{equation}\label{lambdaj-t}
\lambda_j(W_n) = t\left(\frac{j}{n}\right) + O(n^{-\delta})
\end{equation}
for all $1 \leq j \leq n$ and some absolute constant $\delta > 0$, where $-2 \leq t(a) \leq 2$ is defined by the formula
\begin{equation}\label{lt}
a =: \int_{-2}^{t(a)} \rho_{sc}(x)\ dx.
\end{equation}
In particular we have
\begin{equation}\label{bai}
\sup_{1 \leq i \leq n} |\lambda_i(M_n)| = 2\sqrt{n} (1 + o(1))
\end{equation}
asymptotically almost surely (see \cite{baiyin2} for further discussion).

Theorem \ref{theorem:Wigner} addressed the global behavior of the
eigenvalues. The local properties are  much harder and their studies
require much more sophisticated tools. Most of the precise theorems
have been obtained for the GUE, defined in Example \ref{gue}. In the next few paragraphs, we
mention some of the most famous results concerning this model.

\subsection{Distribution of the spacings (gaps) of the eigenvalues of GUE}

In this section $M_n$ is understood to have the GUE distribution.

For a vector $x=(x_1, \dots, x_n)$ where $x_1 < x_2 \dots < x_n $,
define the normalized gap distribution $S_n(s;x)$ by the formula
$$S_n(s; x):=\frac{1}{n} |\{1 \leq i \leq n: x_{i+1}-x_i \le s \} |. $$
For the GUE ensemble it is known\cite{Meh} that
\begin{equation} \label{eqn:gaplimit}
\lim_{n \rightarrow \infty} \E S_n(s, \lambda(A_n)) =
\int_0^s p(\sigma) d \sigma,
\end{equation}
\noindent where $A_n := \sqrt{n} M_n$ is the fine-scale normalization of $M_n$, and $p(\sigma)$ is the \emph{Gaudin distribution}, given by the formula

$$p(s) := \frac{d^2}{ds^2} \det (I- K)_{L^2(0,s)}, $$

\noindent  where $K$ is the integral operator on $L^2 ((0,s))$ with the \emph{Dyson sine kernel} 
\begin{equation}\label{dyson}
K(x,y) := \frac{\sin \pi(x-y) }{\pi (x-y)}.
\end{equation}

\vskip2mm

In fact a stronger result is known in the bulk of the spectrum.  Let $l_n$ be any sequence of numbers tending to infinity such that
$l_n/n$ tends to zero. Define
\begin{equation}\label{ssn}
\tilde S_n (s; x, u) := \frac{1}{l_n} |\{1 \leq i \leq n: x_{i+1}- x_i \le \frac{s}{
\rho_{sc} (u) }, |x_i-nu| \le \frac{l_n}{\rho_{sc}(u) } \}|. 
\end{equation}

It is proved in \cite{DKMVZ} that for any fixed $-2 < u < 2$, we have

\begin{equation} \label{eqn:gaplimit2}
\lim_{n \rightarrow \infty} \E \tilde S_n (s; \lambda (A_n), u) =
\int_0^s p(\sigma) d \sigma.
\end{equation}

The eigenvalue  gap distribution has received much attention in the mathematics
community, partially thanks to  the fascinating (numerical)
coincidence with the gap distribution of the zeros of the zeta
functions. For more discussions, we refer to  \cite{Deibook, KS,
Deisur} and the references therein.

\subsection{$k$-point correlation for GUE}

Given a fine-scale normalized Wigner Hermitian matrix $A_n$, we can define the symmetrized distribution function $\rho_n^{(n)}: \R^n \to \R^+$ to be the symmetric function on $n$ variables such that the distribution of the eigenvalues $\lambda(A_n)$ is given by the restriction of $n! \rho_n^{(n)}(x_1,\ldots,x_n)\ dx_1 \ldots d x_n$ to the region $\{ x_1 \leq \ldots \leq x_n\}$. For any $1 \leq k \leq n$, the \emph{$k$-point correlation function} $\rho_n^{(k)}: \R^k \to \R^+$ is defined as the marginal integral of $\rho_n^{(n)}$:
$$\rho_n^{(k)}(x_1,\dots, x_k) := \frac{n!}{(n-k)!} \int_{\R^{n-k} } \rho_n^{(n)}(x) dx_{k+1} \dots d x_{n}. $$

In  the GUE case, one has an explicit formula for $\rho_n^{(n)}$, obtained by Ginibre \cite{Gin}:

\begin{equation} \label{eqn:ginibrecomplex}
\rho_n^{(n)}(x)  := \frac{n^{-n/2}}{Z_n^{(2)}} \prod_{1\le i < j \le n} |x_i -x_j |^2
\exp (-\frac{1}{2n}(x_1^2 + \dots + x_n^2)), \end{equation}
where $Z_n^{(2)} > 0$ is a normalizing constant, known as the \emph{partition function}. 
From this formula, one can compute $\rho_n^{(k)}$ explicitly.  Indeed, it was established by Gaudin and Mehta \cite{gaudin} that
\begin{equation} \label{eqn:correlation}
\rho_n^{(k)}(x_1, \dots, x_k)  = \det (K_n (x_i, x_j))_{1 \leq i,j \leq k}
\end{equation}
\noindent where the kernel $K_n (x,y)$ is given by the formula
$$ K_n(x,y) := \frac{1}{\sqrt{2n}} e^{-\frac{1}{4n}(x^2+y^2)} \sum_{j=0}^{n-1} h_j(\frac{x}{\sqrt{2n}}) h_j(\frac{y}{\sqrt{2n}})$$
and $h_0,\ldots,h_{n-1}$ are the first $n$ Hermite polynomials, normalized to be orthonormal with respect to $e^{-x^2}\ dx$.
From this and the asymptotics of Hermite polynomials, it was shown by Dyson \cite{Dys} that
\begin{equation} \label{eqn:correlationlimit}
\lim_{n \rightarrow \infty} \frac{1}{\rho_{sc} (u)^k} \rho_n^{(k)}(nu+
\frac{t_1}{\rho_{sc} (u) }, \dots, nu+ \frac{t_k}{\rho_{sc}
(u)}) = \det ( K( t_i, t_j ) )_{1\le
i, j \le k},
\end{equation}
\noindent for any fixed $-2 < u < 2$ and real numbers $t_1, \dots,
t_k$, where the Dyson sine kernel $K$ was defined in \eqref{dyson}. 

\vskip2mm

\subsection{The universality conjecture and previous results}

It has been conjectured, since the
1960s, by Wigner, Dyson, Mehta, and many others, that the local statistics
(such as the above limiting distributions) are \emph{universal}, in the sense that they hold not only for the GUE, but for any other Wigner random matrix also. This conjecture was motivated by similar phenomena in physics, such as the same laws of
 thermodynamics, which should emerge no matter what the details of atomic
 interaction.

The universality conjecture is  one of the  central questions in the
theory of random matrices. In many cases, it is stated for a
specific local statistics (such as the gap distribution or the
$k$-point correlation, see \cite[page 9]{Meh} for example). These
problems have been discussed in numerous books and surveys (see \cite{Meh, Deibook, Deisur}).

Despite the
conjecture's  long and distinguished history and the overwhelming
supporting numerical evidence, rigorous results on this problem for general Wigner random matrices have only begun to emerge recently.
At the edge of the spectrum, Soshnikov\cite{Sos1} proved the universality of the joint distribution of the
largest $k$ eigenvalues (for any fixed $k$), under the extra
assumption that  the atom distribution is symmetric:

\begin{theorem} \label{theorem:Sos} \cite{Sos1}
Let $k$ be a fixed integer and $M_n$ be a Wigner Hermitian matrix,
whose atom distribution is symmetric.  Set $W_n := \frac{1}{\sqrt{n}} M_n$. Then the joint distribution of the
$k$ dimensional random vector $$\Big( (\lambda_n(W_n) -2)n^{2/3}, \dots,
(\lambda_{n-k}(W_n)-2) n^{2/3} \Big) $$ has a weak limit as $n
\rightarrow \infty$, which coincides with that in the GUE case. The
result also holds for the smallest eigenvalues $\lambda_1, \dots,
\lambda_k $.
\end{theorem}

Note that this significantly strengthens \eqref{bai} in the symmetric case.  (For the non-symmetric case, see \cite{peches}, \cite{peches2} for some recent results).

Returning to the bulk of the spectrum, Johansson \cite{Joh} proved \eqref{eqn:correlationlimit} and
\eqref{eqn:gaplimit2}  for random Hermitian matrices whose entries
are \emph{gauss divisible}. (See also the paper \cite{BenP}  of Ben Arous and P\'ech\'e where they discussed the removal of a technical
condition in \cite{Joh}.) More precisely, Johansson considered the model $M_n= (1-t)^{1/2} M^1_n +
t^{1/2} M^2_n$, where $0 < t \leq 1$ is fixed (i.e. independent of $n$), $M^1_n$ is a Wigner Hermitian matrix and $M^2_n$
is a GUE matrix independent of $M^1_n$. We will refer to such matrices as \emph{Johansson matrices}.

\begin{theorem}\label{theorem:johansson}\cite{Joh}

\eqref{eqn:correlationlimit} (in the weak sense) and \eqref{eqn:gaplimit2} (and hence \eqref{eqn:gaplimit}) hold for
Johansson matrices, as $n \rightarrow \infty$.  By ``weak sense'', we mean that
\begin{equation} \label{eqn:correlationlimit2}
\begin{split}
&\lim_{n \rightarrow \infty} 
\frac{1}{\rho_{sc} (u)^k} \int_{\R^k} f(t_1,\ldots,t_k) \rho_n^{(k)}(nu+
\frac{t_1}{\rho_{sc} (u) }, \dots, nu+ \frac{t_k}{\rho_{sc}
(u)})\ dt_1 \ldots dt_k \\
&\quad = \int_{\R^k} f(t_1,\ldots,t_k) \det ( K( t_i, t_j ) )_{1\le
i, j \le k}\ dt_1 \ldots dt_k
\end{split}
\end{equation}
for any test function $f \in C_c(\R^k)$.
\end{theorem}

The property of being gauss divisible can be viewed as a strong regularity assumption on the atom distribution.
Very recently, Erd\H{o}s, Peche, Ramirez, Schlein and Yau \cite{ERSY}, \cite{ERSY2} have relaxed this regularity assumption significantly.
 In particular in \cite{ERSY2} an analogue of Theorem \ref{theorem:johansson}  (with $k=2$ for the correlation and $l_{n}$ polynomial in $n$ 
 for the gap distribution)
 is proven assuming that  the atom distribution is of the form 
 
 $$\nu dx = e^{-V(x) } e^{-x^{2}} dx  $$ where $V(x) \in C^{6} $ and $\sum_{i=1}^{6} |V^{j} (x)| \le C(1+x^{2} )^{k} $ and 
 $\nu (x) \le C' \exp(-\delta x^{2})$, for some fixed $k, \delta, C, C'$. It was remarked in \cite{ERSY2} that the last (exponential decay) 
  assumption can be weakened somewhat. 
  
  We note that both Erd\H os et al and our approach make
important, but different use of the ideas and results by Erdos,Schlein and Yau ([17,18,19]) concerning the rate of convergence to the semi-circle
law. (See page 24 and Section 5 for a detailed description.)

Finally, let us mention that in a different direction, universality was established by Deift, Kriecherbauer, McLaughlin,  Venakides
and  Zhou, \cite{DKMVZ}, Pastur and Shcherbina \cite{PS}, Bleher and Its  \cite{BI} for a different model of random
matrices, where the joint distribution of the eigenvalues is given
explicitly by the formula

\begin{equation}\label{eqn:ginibregeneral}
 \rho_n^{(n)} (x_1,\ldots,x_n)  := c_{n} \prod_{1\le i < j
\le n} |x_i -x_j |^2 \exp (-V(x) ),
\end{equation}

\noindent where $V$ is a general function and $c_n > 0$ is a normalization factor. The case $V=x^2$ corresponds to
\eqref{eqn:ginibrecomplex}. For a general $V$, the
 entries of the matrix are correlated, and so this model differs from the Wigner model.  (See \cite{ll} for some recent developments concerning these models, which are studied using the machinery of orthogonal polynomials.)

\vskip2mm

One of the main difficulties in establishing universality for general matrix ensembles lies in the fact that most of the results obtained in the GUE case (and the case in Johansson's theorem and those in  \cite{DKMVZ, PS, BI}) came from
heavy use of the explicit joint distribution of the eigenvalues such
as \eqref{eqn:ginibrecomplex} and \eqref{eqn:ginibregeneral}. The
desired limiting distributions  were proved using estimates on
integrals with respect to these measures. Very powerful tools have
been developed  to handle this task (see \cite{Meh, Deibook} for
example), but  they cannot be applied for general Wigner
matrices  where an explicit measure is not available.

Nevertheless, some methods have been developed which do not require the explicit joint distribution.  
For instance, Soshnikov's result\cite{Sos1} was obtained using the (combinatorial) trace method rather than from 
an explicit formula from the distribution, although it is well understood that this method, while efficient for the studying of the edge, is of 
much less use in the study of the spacing distribution in the bulk of the spectrum.  The recent argument in \cite{ERSY} also avoid explicit formulae, relying instead on an analysis of the \emph{Dyson Brownian motion}, which describes the stochastic dynamics of the spectrum of Johansson matrices $M_n= (1-t)^{1/2} M^1_n +
t^{1/2} M^2_n$ in the $t$ variable.  (On the other hand, the argument in \cite{ERSY2} uses explicit formulae for the joint distribution.) However, it appears that their method still requires a high degree of regularity on the atom distribution, whereas here we shall be interested in methods that do not require any regularity hypotheses at all (and in particular will be applicable to \emph{discrete} atom distributions\footnote{Subsequently to the release of this paper, we have realized that the two methods can in fact be combined to 
address the gap distribution problem and the $k$-point correlation problem even for  discrete distributions without requiring moment conditions; see \cite{ERSTVY} for details.}.


\subsection{Universality theorems}

In this paper, we introduce a new method to study the local statistics. This method is 
 based on the \emph{Lindeberg strategy} \cite{lindeberg} of replacing non-gaussian random variables with gaussian ones. 
 (For more modern discussions about Lindeberg's method, see \cite{Chat, PR}.)   Using this method, we are able to  prove  universality for general Wigner matrices under very mild assumptions.  For instance, we have

\begin{theorem}[Universality of gap]\label{theorem:main1}
The limiting gap distribution \eqref{eqn:gaplimit} holds for Wigner Hermitian matrices whose atom distribution $\xi$ has support on at least 3 points.
The stronger version   \eqref{eqn:gaplimit2} holds for Wigner Hermitian matrices whose atom distribution $\xi$ has support on at least 3 points
and  the third moment $\E \xi^3$ vanishes. \end{theorem}

\begin{remark} \label{remark:gap1} 
Our method also enables us to prove the universality of the variance and higher moments.
Thus, the whole distribution of $S_{n}(s, \lambda)$ is universal, not only its expectation. See Remark
\ref{remark:gap2}. 
\end{remark} 

\begin{theorem}[Universality of correlation]\label{theorem:main2}
The $k$-point correlation \eqref{eqn:correlationlimit} (in the weak sense) 
 holds for Wigner Hermitian matrices whose
 atom distribution $\xi$ has support on at least 3 points
 and the third moment $\E \xi^3$ vanishes.
\end{theorem}

These theorems (and several others, see Section \ref{applic}) are consequences
 of our more general main theorem below (Theorem \ref{theorem:main}). Roughly speaking,
 Theorem \ref{theorem:main}
 states that the local statistics of the eigenvalues of a random matrix
 is determined by  the first four moments of the atom distributions.

Theorem \ref{theorem:main} applies in a very general setting.
 We will consider  random Hermitian matrix
$M_n$ with entries $\xi_{ij}$ obeying the following condition.

\begin{definition}[Condition {\bf C0}]\label{co-def}  A random Hermitian matrix $A_n = (\zeta_{ij})_{1 \leq i, j \leq n}$ is said to obey \emph{condition {\bf C0}} if
\begin{itemize}
\item The $\zeta_{ij}$ are independent (but not necessarily identically distributed) for $1 \leq i \leq j \leq n$, and have mean zero and variance $1$.

\item  (Uniform exponential decay) There exist constants $C, C' > 0$ such that
\begin{equation}\label{ued}
\P( |\zeta_{ij}| \ge t^C) \le \exp(- t)
\end{equation}
for all $t \ge C'$ and $1 \leq i,j \leq n$.
\end{itemize}
\end{definition}

Clearly, all Wigner Hermitian matrices obey condition {\bf C0}.  However, the class of matrices obeying condition {\bf C0} is much richer.  For instance the gaussian orthogonal ensemble (GOE), in which $\zeta_{ij} \equiv N(0,1)$ independently for all $i < j$ and $\zeta_{ii} \equiv N(0,2)$, is also essentially of this form\footnote{Note that for  GOE an diagonal entry has variance $2$ rather than $1$.  We thank Sean O'Rourke for pointing out this issue. On the other hand, Theorem \ref{theorem:main} still holds if we change the variances of the diagonal entries, see Remark \ref{remark:main0}}, and so are all Wigner real symmetric matrices (the definition of which is given at the end of this section). 

\begin{definition}[Moment matching]\label{def:match}
We say that two complex random variables $\zeta$ and $\zeta'$ \emph{match to order} $k$ if
$$ \E \Re(\zeta)^m \Im(\zeta)^l = \E \Re(\zeta')^m \Im(\zeta')^l$$
for all $m, l \ge 0$ such that $m+l  \le k$. 
\end{definition}
 
\begin{example} Given two random matrices $A_n = (\zeta_{ij})_{1 \leq i,j \leq n}$ and $A'_n = (\zeta'_{ij})_{1 \leq i,j \leq n}$ obeying condition {\bf C0}, $\zeta_{ij}$ and $\zeta'_{ij}$ automatically match to order $1$.  If they are both Wigner Hermitian matrices, then they automatically match to order $2$.  If furthermore they are also symmetric (i.e. $A_n$ has the same distribution as $-A_n$, and similarly for $A'_n$ and $-A'_n$) then $\zeta_{ij}$ and $\zeta'_{ij}$ automatically match to order $3$.
\end{example}

Our main result is

\begin{theorem}[Four Moment Theorem]\label{theorem:main} There is a small positive constant $c_0$ such that for every $0 < \eps < 1$ and $k \geq 1$ the following holds.
 Let $M_n = (\zeta_{ij})_{1 \leq i,j \leq n}$ and $M'_n = (\zeta'_{ij})_{1 \leq i,j \leq n}$ be
 two random matrices satisfying {\bf C0}. Assume furthermore that for any $1 \le  i<j \le n$, $\zeta_{ij}$ and
 $\zeta'_{ij}$  match to order $4$
  and for any $1 \le i \le n$, $\zeta_{ii}$ and $\zeta'_{ii}$ match  to order $2$.  Set $A_n := \sqrt{n} M_n$ and $A'_n := \sqrt{n} M'_n$,
  and let $G: \R^k \to \R$ be a smooth function obeying the derivative bounds
\begin{equation}\label{G-deriv}
|\nabla^j G(x)| \leq n^{c_0}
\end{equation}
for all $0 \leq j \leq 5$ and $x \in \R^k$.
 Then for any $\eps n \le i_1 < i_2 \dots < i_k \le (1-\eps)n$, and for $n$ sufficiently large depending on $\eps,k$ (and the constants $C, C'$ in Definition \ref{co-def}) we have
\begin{equation} \label{eqn:approximation}
 |\E ( G(\lambda_{i_1}(A_n), \dots, \lambda_{i_k}(A_n))) -
 \E ( G(\lambda_{i_1}(A'_n), \dots, \lambda_{i_k}(A'_n)))| \le n^{-c_0}.
\end{equation}
If $\zeta_{ij}$ and $\zeta'_{ij}$ only match to order $3$ rather than $4$, 
then there is a positive  constant $C$ independent of $c_{0}$ such that 
the conclusion \eqref{eqn:approximation} still holds provided that one strengthens \eqref{G-deriv} to
$$
|\nabla^j G(x)| \leq n^{-C j c_0}
$$
for all $0 \leq j \leq 5$ and $x \in \R^k$.
\end{theorem}

The proof of this theorem begins at Section \ref{hlt-sec}.  As mentioned earlier, Theorem \ref{theorem:main} asserts (roughly speaking) that the fine spacing statistics of a random Hermitian matrix in the bulk of the spectrum are only sensitive to the first four moments of the entries.  It may be possible to reduce the number of matching moments in this theorem, but this seems to require a refinement to the method; see Section \ref{stratsec} for further discussion.

\begin{remark} \label{remark:main0} 
Theorem \ref{theorem:main} still holds if we assume that the diagonal entries $\zeta_{ii}$ and $\zeta_{ii}'$ have the same 
mean and variance, for all $1\le i \le n$, but these means and variances can be different at different $i$. The proof is  essentially the same. In our analysis, we consider the random vector formed by non-diagonal entries of a row, and it is important that these entries have mean zero and the same variance, but the mean and variance of the diagonal entry 
never plays a role. Details will appear elsewhere. 
\end{remark} 

\begin{remark} \label{remark:main1} 
In a subsequent paper \cite{TVedge}, we show that the condition  $\eps n \le i_1 < i_2 \dots < i_k \le (1-\eps)n$ can be omitted. In other words, 
Theorem \ref{theorem:main} also holds for eigenvalues at the edge of the spectrum. 
\end{remark} 

Applying Theorem \ref{theorem:main} for the special case when $M'_n$ is GUE, we obtain

\begin{corollary} \label{cor:matching1}
Let $M_n$ be a Wigner Hermitian matrix whose atom distribution $\xi$
satisfies $\E \xi^3 =0$ and $E\xi^4 = \frac{3}{4}$ and $M_n'$ be a random matrix
sampled from GUE. Then with $G$, $A_n$, $A'_n$ as in the previous theorem, and $n$ sufficiently large, one has
\begin{equation} \label{eqn:approximation2}
 |\E ( G(\lambda_{i_1}(A_n), \dots, \lambda_{i_k}(A_n))) -
 \E ( G(\lambda_{i_1}(A'_n), \dots, \lambda_{i_k}(A'_n)))| \le 
  n^{-c_0}.
  \end{equation}
\end{corollary}

In the proof of  Theorem \ref{theorem:main}, the following  lower tail estimate on the consecutive spacings plays  an important role. 
This theorem is of independent interest, and will also help in applications of Theorem \ref{theorem:main}. 

\begin{theorem}[Lower tail estimates]\label{ltail} Let $0 < \eps < 1$ be a constant, and let $M_n$ be a random matrix obeying Condition {\bf C0}.  Set $A_n := \sqrt{n} M_n$.  Then for every $c_0 > 0$, and for $n$ sufficiently large depending on $\eps$, $c_0$ and the constants $C, C'$ in Definition \ref{co-def}, and for each $\eps n \leq i \leq (1-\eps) n$, one has $\lambda_{i+1}(A_n) - \lambda_i(A_n) \ge n^{-c_0}$
 with high probability. In fact, one has
$$ \P( \lambda_{i+1}(A_n) - \lambda_i(A_n) \leq n^{-c_0} ) \leq n^{-c_1}$$
for some $c_1 > 0$ depending on $c_0$ (and independent of $\eps$).
\end{theorem}

The proof of this theorem begins at Section \ref{hlt2-sec}.




\subsection{Applications}\label{applic}

By using Theorem \ref{theorem:main} and Theorem \ref{ltail} in combination with existing results in the literature for GUE (or other special random matrix ensembles) one can establish universal asymptotic statistics for a wide range of random matrices.  For instance, consider the $i^{th}$ eigenvalue $\lambda_i(M_n)$ of a Wigner Hermitian matrices.  In the GUE case,  Gustavsson\cite{Gus}, based on \cite{Sos2, CLe}, proved that $\lambda_{i} $ has gaussian fluctuation:

\begin{theorem}[Gaussian fluctuation for GUE]\label{gusnorm}\cite{Gus} Let $i = i(n)$ be such that $i/n \to c$ as $n \to \infty$ for some $0 < c < 1$.  Let $M_n$ be drawn from the GUE.  Set $A_n := \sqrt{n} M_n$.  Then
$$ \sqrt{\frac{4-t(i/n)^2}{2}} \frac{\lambda_i(A_n) - t(i/n) n}{\sqrt{\log n}} \to N(0,1)$$
in the sense of distributions, where $t()$ is defined in \eqref{lt}.  (More informally, we have $\lambda_i(M_n) \approx t(i/n)\sqrt{n} + N( 0, \frac{2\log n}{(4-t(i/n)^2)n} )$.)
\end{theorem}

As an application of our main results, we have

\begin{corollary}[Universality of gaussian fluctuation]\label{skug} 
The conclusion of  Theorem \ref{gusnorm} also holds for any other Wigner Hermitian matrix $M_n$ whose atom distribution $\xi$ satisfies 
 $\E \xi^3 = 0$ and $\E \xi^4 = \frac{3}{4}$.
\end{corollary}

\begin{proof} Let $M_n$ be a Wigner Hermitian matrix, and let $M'_n$ be drawn from GUE.  Let $i, c, t$ be as in Theorem \ref{gusnorm}, and let $c_0$ be as in Theorem \ref{ltail}.  In view of Theorem \ref{gusnorm}, it suffices to show that
\begin{equation}\label{lama}
\P( \lambda_i( A'_n ) \in I_- ) - n^{-c_0} \leq
\P( \lambda_i( A_n ) \in I ) \leq \P( \lambda_i( A'_n ) \in I_+ ) + n^{-c_0}
\end{equation}
for all intervals $I = [a,b]$, and $n$ sufficiently large depending on $i$ and the constants $C,C'$ in Definition \ref{def:Wignermatrix}, where $I_+ := [a-n^{-c_0/10}, b+n^{-c_0/10}]$ and $I_- := [a+n^{-c_0/10},b-n^{-c_0/10}]$.

We will just prove the second inequality in \eqref{lama}, as the first is very similar.
We define a smooth bump function $G:\R \to \R^+$ equal to one on $I_-$ and vanishing outside of $I_+$.  Then we have
$$ \P( \lambda_i( A_n ) \in I ) \leq \E G( \lambda_i(A_n) ) $$
and
$$ \E G( \lambda_i(A'_n) ) \leq \P( \lambda_i( A'_n ) \in I )$$
On the other hand, one can choose $G$ to obey \eqref{G-deriv}.  Thus by Corollary \ref{cor:matching1} we have
$$ |\E G( \lambda_i(A_n) ) - \E G( \lambda_i(A'_n) )| \leq n^{-c_0}$$
and the second inequality in \eqref{lama} follows from the triangle inequality.  The first inequality is similarly proven using a smooth function that equals $1$ on $I_-$ and vanishes outside of $I$.
\end{proof}

\begin{remark} The same argument lets one establish the universality of the asymptotic joint distribution law for any $k$ eigenvalues $\lambda_{i_1}(M_n),\ldots,\lambda_{i_k}(M_n)$ in the bulk of the spectrum of a Wigner Hermitian matrix for any fixed $k$ (the GUE case 
is treated in  \cite{Gus}).  In particular, we have the generalization
\begin{equation}\label{lama-2}
\begin{split}
\P( \lambda_{i_j}( A'_n ) \in I_{j,-} \hbox{ for all } 1 \leq j \leq k) &+ O_k( n^{-c_0} )\\
&\leq \P( \lambda_{i_j}( A_n ) \in I_j \hbox{ for all } 1 \leq j \leq k ) \\
&\leq \P( \lambda_{i_j}( A'_n ) \in I_{j,+} \hbox{ for all } 1 \leq j \leq k) + O_k( n^{-c_0} )
\end{split}
\end{equation}
for all $i_1,\ldots,i_k$ between $\eps n$ and $(1-\eps) n$ for some fixed $\eps > 0$, and all intervals $I_1,\ldots,I_k$, assuming $n$ is sufficiently large depending on $\eps$ and $k$, and $I_{j,-} \subset I_j \subset I_{j,+}$ are defined as in the proof of Corollary \ref{skug}. The details are left as an exercise to the interested reader. 
\end{remark}

Another quantity of interest is the least singular value
$$ \sigma_n(M_n) := \inf_{1 \leq i \leq n} |\lambda_i(M_n)|$$
of a Wigner Hermitian matrix.  In the GUE case, we have the following asymptotic distribution:

\begin{theorem}[Distribution of least singular value of GUE]\label{lsv-gue}\cite[Theorem 3.1.2]{AGZ}, \cite{JMM} For any fixed $t > 0$, and $M_n$ drawn from GUE, one has
$$ \P( \sigma_n(M_n) \leq \frac{t}{2\sqrt{n}} ) \to \exp( \int_0^t  \frac{f(x)}{x} dx )$$
as $n \to \infty$, where $f:\R \to \R$ is the solution of the differential equation
$$(tf'')^2 + 4 (tf'-f) (tf' -f + (f')^2) =0 $$
with the asymptotics  $f(t)  = \frac{-t}{\pi} -\frac{t^2}{\pi^2} -\frac{t^3}{\pi^3} + O(t^4)$ as $t \to 0$.
\end{theorem}

Using our theorems, we can extend this result to more general ensembles:

\begin{corollary}[Universality of the distribution of the  least singular value]\label{lsv-wigner}  The conclusions of Theorem \ref{lsv-gue} also hold for 
any other Wigner Hermitian matrix $M_n$ whose atom distribution $\xi$ has $\E \xi^3 = 0$ and $\E \xi^4 = \frac{3}{4}$.
\end{corollary}

\begin{proof} Let $M_n$ be a Wigner Hermitian matrix, and let $M'_n$ be drawn from GUE.
Let $N_I$ be the number of eigenvalues of $W'_n$ in an interval $I$. It is well known (see \cite[Chapter 4]{AGZ}) that 
\begin{equation}\label{noga}
N_I = \int_I \rho_{sc} (x) dx + O(\log n)
\end{equation}
asymptotically almost surely (cf. \eqref{lambdaj-t} and Theorem \ref{gusnorm}).  Applying this fact to the two intervals $I = [-\infty, \pm \frac{t}{2\sqrt{n}}]$, we conclude that
$$\P( \lambda_{n/2 \pm \log^2 n}(M'_n) \in (-\frac{t}{2\sqrt{n}}, \frac{t}{2\sqrt{n}} ) ) = o(1)$$
for either choice of sign $\pm$.   Using \eqref{lama} (or \eqref{lama-2}) (and modifying $t$ slightly), we conclude that the same statement is true for $M_n$.  In particular, we have
$$
\P( \sigma_n(M_n) > \frac{t}{2\sqrt{n}} ) = \sum_{n/2 - \log^2 n \leq i \leq n/2 + \log^2 n} \P( \lambda_i(M_n) < -\frac{t}{2\sqrt{n}} \wedge
\lambda_{i+1}(M_n) > \frac{t}{2\sqrt{n}} ) + o(1)$$
and similarly for $M'_n$.  Using \eqref{lama-2}, we see that

\begin{align*} &  \P \Big( \lambda_i(M_n) < -\frac{t}{2\sqrt{n}} \wedge
\lambda_i(M_n) > -\frac{t}{2\sqrt{n}} \Big) \\  &\leq \P \Big( \lambda_i(M'_n) < -\frac{t}{2\sqrt{n}}+n^{-c_0/10} \wedge
\lambda_{i+1}(M'_n) > \frac{t}{2\sqrt{n}}-n^{-c_0/10}  \Big) + O( n^{-c_0} ) \end{align*}
and

\begin{align*}  &\P \Big( \lambda_i(M_n) < -\frac{t}{2\sqrt{n}} \wedge
\lambda_i(M_n) > -\frac{t}{2\sqrt{n}} \Big)  \\ &\geq \P \Big( \lambda_i(M'_n) < -\frac{t}{2\sqrt{n}}-n^{-c_0/10} \wedge
\lambda_{i+1}(M'_n) > \frac{t}{2\sqrt{n}}+n^{-c_0/10} \Big) - O( n^{-c_0} ) \end{align*}

for some $c>0$.  Putting this together, we conclude that

\begin{align*}  \P( \sigma_n(M'_n) > \frac{t}{2\sqrt{n}}-n^{-c_0/10} ) + o(1)  &\leq
\P( \sigma_n(M_n) > \frac{t}{2\sqrt{n}} ) \\&\leq \P( \sigma_n(M'_n) > \frac{t}{2\sqrt{n}}+n^{-c_0/10} ) + o(1) \end{align*}
and the claim follows.
\end{proof}

\begin{remark} A similar universality result for the least singular value of non-Hermitian 
 matrices  was recently established by the authors in \cite{TVhard}.  Our arguments in \cite{TVhard} also used the Lindeberg strategy, but were rather different in many other respects (in particular, they proceeded by analyzing random submatrices of the inverse matrix $M_n^{-1}$).  One consequence of Corollary \ref{lsv-wigner} is that $M_n$ is asymptotically almost surely invertible.  For discrete random matrices, 
 this is already a non-trivial fact, first proven in \cite{cost}.  If Theorem \ref{lsv-gue} can be extended to the Johansson matrices considered in \cite{Joh}, then the arguments below would allow one to remove the fourth moment hypothesis in Corollary \ref{lsv-wigner} (assuming that $\xi$ is supported on at least three points).
\end{remark}

\begin{remark} The above corollary still holds under  a weaker assumption that the first three moments of $\xi$ match those of the gaussian variable; in other words, we can omit the last 
assumption that $\E \xi^4 =3/4$.   Details will appear else where. \end{remark}

\begin{remark} By combining this result with \eqref{bai} one also obtains a universal distribution for the condition number $\sigma_1(M_n)/\sigma_n(M_n)$ of Wigner Hermitian matrices (note that the non-independent nature of $\sigma_1(M_n)$ and $\sigma_n(M_n)$ is not relevant, because \eqref{bai} gives enough concentration of $\sigma_1(M_n)$ that it can effectively be replaced with $2\sqrt{n}$). We omit the details. 
\end{remark}

Now we are going to prove the first part of Theorem \ref{theorem:main1}.  Note that in contrast to previous applications, we are making no assumptions on the third and fourth moments of the atom distribution $\xi$.  The extra observation here is that we do not always need to compare $M_{n} $ with GUE. It is sufficient  to compare it with any model where the desired statistics have been computed.  In this case, we are going to compare 
$M_{n} $ with  a Johansson matrix.  The definition of Johansson matrices provides more degrees of freedom via the parameters $t$ and $M_n^1$, and we can 
use this to remove the condition of the third and fourth moments. 

\begin{lemma}[Truncated moment matching problem]\label{tmmp}   Let $\xi$ be a real random variable with mean zero, variance $1$, third moment $\E \xi^3 = \alpha_3$, and fourth moment $\E \xi^4 = \alpha_4 < \infty$.  Then $\alpha_4 - \alpha_3^2 - 1 \geq 0$, with equality if and only if $\xi$ is supported on exactly two points.  Conversely, if $\alpha_4 - \alpha_3^2 - 1 \geq 0$, then there exists a real random variable with the specified moments.
\end{lemma}

\begin{proof}  For any real numbers $a, b$, we have
$$ 0 \leq \E (\xi^2 + a\xi + b)^2 = \alpha_4 + 2 a \alpha_3 + a^2 + 2b + b^2.$$
setting $b := -1$ and $a := - \alpha_3$ we obtain the inequality $\alpha_4 - \alpha_3^2 - 1 \geq 0$.  Equality only occurs when $\E( \xi^2 - \alpha_3 \xi - 1)^2 = 0$, which by the quadratic formula implies that $\xi$ is supported on at most two points.

Now we show that every pair $(\alpha_3,\alpha_4)$ with $\alpha_4 - \alpha_3^2 - 1 \geq 0$ arises as the moments of a random variable with mean zero and variance $1$.  The set of all such moments is clearly convex, so it suffices to check the case when $\alpha_4 - \alpha_3^2 - 1 = 0$.  But if one considers the random variable $\xi$ which equals $\tan \theta$ with probability $\cos^2 \theta$ and $-\cot \theta$ with probability $\sin^2 \theta$ for some $-\pi/2 < \theta < \pi/2$, one easily computes that $\xi$ has mean zero, variance $1$, third moment $-2 \cot(2\theta)$, and fourth moment $4 \operatorname{cosec}(2\theta) - 3$, and the claim follows from the trigonometric identity $\operatorname{cosec}(2 \theta)^2 = \cot(2\theta)^2 + 1$.
\end{proof}

\begin{remark}  The more general truncated moment problem
 (i.e., the truncated version of the classical  Hamburger moment sequence problem, see \cite{CF,Lan}) was solved by Curto and Fialkow \cite{CF}.
\end{remark}

\begin{corollary}[Matching lemma]\label{matchc}  Let $\xi$ be a real random variable with mean zero and variance $1$, which is supported on at least three points.  Then $\xi$ matches to order $4$ with $(1-t)^{1/2} \xi' + t^{1/2} \xi_G$ for some $0 < t < 1$ and some independent $\xi', \xi_G$ of mean zero and variance $1$, where $\xi_G \equiv N(0,1)$ is Gaussian.
\end{corollary}

\begin{proof}  The formal characteristic function $\E e^{s \xi} := \sum_{j=0}^\infty \frac{s^j}{j!} \E \xi^j$ has the expansion $1 + \frac{1}{2} s^2 + \frac{1}{6} \alpha_3 s^3 + \frac{1}{24} \alpha_4 s^4 + O(s^5)$; by Lemma \ref{tmmp}, we have $\alpha_4 - \alpha_3^2 - 1 > 0$.  Observe that $\xi$ will match to order $4$ with $(1-t)^{1/2} \xi' + t^{1/2} \xi_G$ if and only if one has the identity
$$ 1 + \frac{1}{2} s^2 + \frac{1}{6} \alpha_3 s^3 + \frac{1}{24} \alpha_4 s^4 = 
(1 + \frac{1-t}{2} s^2 + \frac{(1-t)^{3/2}}{6} \alpha'_3 s^3 + \frac{(1-t)^2}{24} \alpha'_4 s^4) (1 + \frac{t}{2} s^2 + \frac{t}{8} s^4) + O(s^5)$$
where $\alpha'_3, \alpha'_4$ are the moments of $\xi'$.  Formally dividing out by $1 + \frac{t}{2} s^2 + \frac{t}{8} s^4$, one can thus solve for $\alpha'_3, \alpha'_4$ in terms of $\alpha_3, \alpha_4$.  Observe that as $t \to 0$, $\alpha'_3, \alpha'_4$ must converge to $\alpha_3, \alpha_4$ respectively.  Thus, for $t$ sufficiently small, we will have $\alpha'_4 - (\alpha'_3)^2 - 1 > 0$.  The claim now follows from Lemma \ref{tmmp}.
\end{proof}

\begin{proof} (Proof of the first part of 
Theorem \ref{theorem:main1})  Let $M_n$ be as in this theorem and consider  \eqref{eqn:gaplimit}.  By Corollary \ref{matchc}, we can find a Johansson matrix $M'_n$ which matches $M_n$ to order $4$.  By Theorem \ref{theorem:johansson}, \eqref{eqn:gaplimit} already holds for $M'_n$.  Thus it will suffice to show that
$$ \E \tilde S_n (s; \lambda (A_n) ) = \E \tilde S_n (s; \lambda (A'_n) ) + o(1).$$
By \eqref{ssn} and linearity of expectation, it suffices to show that
$$ \P( \lambda_{i+1}(A_n) - \lambda_i(A_n) \le s ) = \P( \lambda_{i+1}(A'_n) - \lambda_i(A'_n) \le s ) + o(1)$$
uniformly for all $\eps n \leq i \leq (1-\eps) n$, for each fixed $\eps > 0$.  But this follows by a modification of the argument used to prove \eqref{lama} (or \eqref{lama-2}), using a function $G(x,y)$ of two variables which is a smooth approximant to the indicator function of the half-space $\{ y-x \leq s \}$ (and using Theorem \ref{ltail} to errors caused by shifting $s$); we omit the details. The second part of the theorem will be treated together with Theorem \ref{theorem:main2}. 
\end{proof}

\begin{remark} \label{remark:gap2}
By considering  $ \P( \{ \lambda_{i+1}(A_n) - \lambda_i(A_n) \le s\} \wedge \{ \lambda_{j+1}(A_n) - \lambda_j(A_n) \le s\}   ) $ we can prove the universality of the variance of $S_{n}(s, \lambda)$. The same applies for higher moments. 
\end{remark} 

The proof of Theorem \ref{theorem:main2} is a little more complicated.  We first need a strengthening of \eqref{lambdaj-t} which may be of independent interest.

\begin{theorem}[Convergence to the semicircular law]\label{conv}  Let $M_n$ be a Wigner Hermitian matrix whose atom distribution $\xi$ has vanishing third moment.  Then for any fixed $c > 0$ and $\eps > 0$, and any $\eps n \leq j \leq (1-\eps) n$, one has
$$\lambda_j(W_n) = t\left(\frac{j}{n}\right) + O(n^{-1+c})$$
asymptotically almost surely, where $t()$ was defined in \eqref{lt}.
\end{theorem}

\begin{proof}  It suffices to show that
$$ \lambda_j(W_n) \in \left[t\left(\frac{j}{n}\right) - n^{-1+c}, t\left(\frac{j}{n}\right) + n^{-1+c}\right]$$
asymptotically almost surely.
Let $M'_n$ be drawn from GUE, thus the off-diagonal entries of $M_n$ and $M'_n$ match to third order.   From \eqref{noga} we have
$$ \lambda_j(W'_n) \in \left[t\left(\frac{j}{n}\right) - n^{-1+c}/2, t\left(\frac{j}{n}\right) + n^{-1+c}/2\right]$$
asymptotically almost surely.
The claim now follows from the last part of Theorem \ref{theorem:main}, letting $G = G(\lambda_j)$ be a smooth cutoff equal to $1$ on 
$\left[t\left(\frac{j}{n}\right) n - n^{+c}/2, t\left(\frac{j}{n}\right) n + n^{+c}/2\right]$ and vanishing outside of $\left[t\left(\frac{j}{n}\right) n - n^{+c}, t\left(\frac{j}{n}\right) n + n^{c}\right]$.
\end{proof}

\begin{proof}[Proof of Theorem \ref{theorem:main2}]
Fix $k, u$, and let $M_n$ be as in Theorem \ref{theorem:main2}.  By Corollary \ref{matchc}, we can find a Johansson matrix $M'_n$ whose entries match $M_n$ to fourth order.  By Theorem \ref{theorem:johansson} (and a slight rescaling), it suffices to show that the quantity
\begin{equation}\label{erk}
 \int_{\R^k} f(t_1,\ldots,t_k) \rho_n^{(k)}(nu+t_1, \ldots, nu + t_k)\ dt_1 \ldots dt_k 
\end{equation}
only changes by $o(1)$ when the matrix $M_n$ is replaced with $M'_n$, for any fixed test function $f$.  By an approximation argument we can take $f$ to be smooth.

We can rewrite the expression \eqref{erk} as
\begin{equation}\label{Sum}
 \sum_{1 \leq i_1,\ldots,i_k \leq n} \E f( \lambda_{i_1}(A_n) - nu, \ldots, \lambda_{i_k}(A_n) - nu ).
\end{equation}
Applying Theorem \ref{theorem:main}, we already have
$$ \E f( \lambda_{i_1}(A_n) - nu, \ldots, \lambda_{i_k}(A_n) - nu ) = 
\E f( \lambda_{i_1}(A'_n) - nu, \ldots, \lambda_{i_k}(A'_n) - nu ) + O( n^{-c_0} )
$$
for each individual $i_1,\ldots,i_k$ and some absolute constant $c_0 > 0$.  Meanwhile, by Theorem \ref{conv}, we see that asymptotically almost surely, the only $i_1,\ldots,i_k$ which contribute to \eqref{Sum} lie within $O(n^c)$ of $t^{-1}(u) n$, where $c > 0$ can be made arbitrarily small.  The claim then follows from the triangle inequality (choosing $c$ small enough compared to $c_0$).
\end{proof}

\begin{proof}[Proof of the second part of Theorem \ref{theorem:main1}] This proof is similar to the one above. We already know that 
$\P(\lambda_{i+1} -\lambda_{i} \le s)$ is basically the same in the two models ($M_{n} $ and $M'_{n}$). Theorem \ref{conv} now shows that after fixing a 
small neighborhood  of $u$, the interval  of indices $i$ that involve fluctuates by at most $n^{c}$, where $c$ can be made arbitrarily small. 
\end{proof} 

\begin{remark} \label{remark:dropthird} In fact, in the above applications, we only need Theorem \ref{conv} to hold for Johansson matrices.
Thus, in order to remove the third moment assumption, it suffices to have this theorem for Johansson matrices
 (without the third moment assumption). We believe this is within the power of
the determinant process method, but do not pursue this direction here. 
\end{remark}

As another application, we can prove the following 
 asymptotic for the determinant (and more generally, the characteristic polynomial) of a Wigner Hermitian matrix. The detailed proof 
 is deferred to Appendix \ref{section:determinant}. 

\begin{theorem}[Asymptotic for determinant]  \label{theorem:determinant} Let $M_n$ be a Wigner Hermitian matrix whose atom distribution $\xi$ has vanishing third moment and is supported on at least three points.  Then there is a constant $c > 0$ such that
$$ \P( |\log |\det M_n| - \log \sqrt{n!}| \geq n^{1-c} ) = o(1).$$
More generally, for fixed any complex number $z$, one has
$$ \P( |\log |\det (M_n - \sqrt{n} z I)| - \frac{1}{2} n \log n - n \int_{-2}^2 \log |y-z| \rho_{sc}(y)\ dy| \geq n^{1-c} ) = o(1)$$
where the decay rate of $o(1)$ is allowed to depend on $z$.
\end{theorem}

\begin{remark} A similar result was established for iid random matrices in \cite{TVdet} (see also \cite{costvu} for a refinement), based on controlling the distance from a random vector to a subspace.  That method relied heavily on the joint independence of all entries and does not seem to extend easily to the Hermitian case.  We also remark that a universality result for correlations of the characteristic polynomial has recently been established in \cite{got}.
\end{remark}

\vskip2mm


Let us now go beyond the model of Wigner Hermitian matrices. 
As already mentioned, our main theorem also applies for real symmetric matrices. In the next  paragraphs, we formulate a few results one can obtain in this direction.

\begin{definition}[Wigner symmetric matrices]\label{def:Wignermatrix-real}
Let $n$ be a large number. A \emph{Wigner symmetric  matrix} (of size $n$)
 is a random symmetric matrix $M_n = (\xi_{ij})_{1 \leq i,j \leq n}$ where for $1 \leq i < j \leq n$, $\xi_{ij}$ are iid copies of a real random variable $\xi$ with mean zero, variance $1$, and exponential decay (as in Definition \ref{def:Wignermatrix}), while for $1 \leq i=j \leq n$, $\xi_{ii}$ are iid copies of a real random variable $\xi'$ with mean zero, variance $2$, and exponential decay.  We set $W_n := \frac{1}{\sqrt{n}} M_n$ and $A_n := \sqrt{n} M_n$ as before.
\end{definition}

\begin{example}  The \emph{Gaussian orthogonal ensemble} (GOE) is the Wigner symmetric matrix in which the off-diagonal atom distribution $\xi$ is the Gaussian $N(0,1)$, and the diagonal atom distribution $\xi'$ is $N(0,2)$.
\end{example}

As remarked earlier, while the Wigner symmetric matrices do not, strictly speaking, obey Condition C0 due to the diagonal variance being $2$ instead of $1$, it is not hard to verify that all the results in this paper continue to hold after changing the diagonal variance to $2$.  As a consequence, we can easily deduce the following analogue of Theorems \ref{theorem:main1} and \ref{theorem:main2}. 

\begin{theorem}[Universality for Random Symmetric Matrices]\label{theorem:main3} 
The limiting gap distribution and $k$-correlation function 
of Wigner symmetric real matrices with atom variable $\sigma$ satisfying $\E \sigma^3 =0$ and $\E\sigma^4 = 3$ are the same as those for GOE.
(The explicit formulae for the  limiting gap distribution and $k$-correlation function for GOE can be found in \cite{Meh, AGZ}.  The limit of the $k$-correlation function is again in the weak sense.)
\end{theorem} 

The proof of Theorem \ref{theorem:main3} is similar to that of Theorems \ref{theorem:main1}, \ref{theorem:main2} and is omitted.
The reason that we need to match the moments to order $4$ here (compared to lower orders in Theorems \ref{theorem:main1} and \ref{theorem:main2}) is that there is currently no analogue of Theorem \ref{theorem:johansson} for the GOE. Once such a result becomes available, the order automatically reduces to those in Theorems \ref{theorem:main1} and \ref{theorem:main2}, respectively.

Finally let us mention that 
our results can be refined and extended in several directions. 
For instance, we can handle 
 Hermitian matrices whose upper triangular entries are still independent, but having  a non-trivial 
  covariance matrix (the real and imaginary parts need not be independent).  The diagonal entries can have mean different from zero (which, in the case the off-diagonal entries are gaussian,  corresponds to gaussian matrices with external field and has been studied in \cite{BK}) and we can obtain  universality
  results  in this case as well.  
  We can also refine our argument to prove universality near the edge of the spectrum. These extensions and many others will be discussed  in a subsequent paper.

\subsection{Notation}\label{notation-sec}

We consider $n$ as an asymptotic parameter tending
to infinity.  We use $X \ll Y$, $Y \gg X$, $Y = \Omega(X)$, or $X =
O(Y)$ to denote the bound $X \leq CY$ for all sufficiently large $n$
and for some constant  $C$. Notations such as
 $X \ll_k Y, X= O_k(Y)$ mean that the hidden constant $C$ depend on
 another constant $k$. $X=o(Y)$ or $Y= \omega(X)$ means that
 $X/Y \rightarrow 0$ as $n \rightarrow \infty$; the rate of decay here will be allowed to depend on other parameters.
 The eigenvalues are always ordered increasingly.

We view vectors $x \in \C^n$ as column vectors. The Euclidean norm of a vector $x \in \C^n$ is defined as $\|x\| := (x^* x)^{1/2}$. 
The \emph{Frobenius norm} $\|A\|_F$ of a matrix is defined
 as $\|A\|_F = \tr( A A^*)^{1/2}$.  Note that this bounds
 the operator norm $\|A\|_{op} := \sup \{ \|Ax\|: \|x\|=1\}$ of the same matrix.  We will also use  the following simple inequalities without further
comment:
$$ \|AB\|_F \leq \|A\|_F \|B\|_{op}$$
and
$$ \| B\|_{op} \leq \|B\|_F$$
and hence
$$ \|AB\|_F \leq \|A\|_F\|B\|_F.$$

\section{Preliminaries: Tools from linear algebra and probability }

\subsection{Tools from Linear Algebra.} 
It is useful to keep in mind  the (Courant-Fisher) minimax characterization of the
eigenvalues
$$\lambda_i(A) = \min_V \max_{u \in V} u^* A u$$ 
of a Hermitian $n \times n$ matrix $A$, where $V$ ranges over $i$-dimensional subspaces
 of $\C^n$ and $u$ ranges over unit vectors in $V$.

From this, one easily obtain \emph{Weyl's inequality}

\begin{equation}\label{weyl}
 \lambda_i(A) - \|B\|_{op} \leq \lambda_i(A+B) \leq \lambda_i(A) + \|B\|_{op}.
\end{equation}

Another consequence of the minimax formula is the \emph{Cauchy interlacing inequality}
\begin{equation}\label{cauchy-interlace}
\lambda_i(A_{n-1}) \leq \lambda_i(A_n) \leq \lambda_{i+1}(A_{n-1})
\end{equation}
for all $1 \leq i < n$, whenever $A_n$ is an $n \times n$ Hermitian matrix and $A_{n-1}$ is the top $n-1 \times n-1$ minor.  In a similar spirit, one has
$$
\lambda_i(A) \leq \lambda_{i}(A + B) \leq \lambda_{i+1}(A)
$$
for all $1 \leq i < n$, whenever $A, B$ are $n \times n$ Hermitian matrices with $B$ being positive semi-definite and rank $1$.  If $B$ is instead negative semi-definite, one has
$$
\lambda_i(A) \leq \lambda_{i+1}(A + B) \leq \lambda_{i+1}(A).$$
In either event, we conclude

\begin{lemma} \label{lemma:addingrankone}
Let $A,B$ be Hermitian matrices of the same size where
$B$ has rank one. Then for any interval $I$,
$$|N_I(A+B)- N_I(A)| \le 1 , $$ 
where $N_I(M)$ is the number of
eigenvalues of $M$ in $I$.
\end{lemma}

One also has the following more precise version of the Cauchy interlacing inequality:

\begin{lemma}[Interlacing identity]\label{jn-lem}  Let $A_n$ be an $n \times n$ Hermitian matrix, let $A_{n-1}$ be the top $n-1 \times n-1$ minor, let $a_{nn}$ be the bottom right component, and let $X \in \C^{n-1}$ be the rightmost column with the bottom entry $a_{nn}$ removed.  Suppose that $X$ is not orthogonal to any of the unit eigenvectors $u_j(A_{n-1})$ of $A_{n-1}$.  Then we have 
\begin{equation}\label{jn}
 \sum_{j=1}^{n-1} \frac{|u_j(A_{n-1})^* X|^2}{\lambda_j(A_{n-1}) - \lambda_i(A_n)} = a_{nn} - \lambda_i(A_n)
\end{equation}
for every $1 \leq i \leq n$.
\end{lemma}

\begin{proof} By diagonalising $A_{n-1}$ (noting that this does not affect either side of \eqref{jn}), we may assume that $A_{n-1} = \operatorname{diag}(\lambda_1(A_{n-1}),\ldots,\lambda_{n-1}(A_{n-1}))$ and $u_j(A_{n-1}) = e_j$ for $j=1,\ldots,n-1$.  One then easily verifies that the characteristic polynomial $\det(A_n - \lambda I)$ of $A_n$ is equal to
$$
\prod_{j=1}^{n-1} (\lambda_j(A_{n-1}) - \lambda) [ (a_{nn} - \lambda) - \sum_{j=1}^{n-1} \frac{|u_j(A_{n-1})^* X|^2}{\lambda_j(A_{n-1}) - \lambda} ]
$$
when $\lambda$ is distinct from $\lambda_1(A_{n-1}),\ldots,\lambda_{n-1}(A_{n-1})$.  Since $u_j(A_{n-1})^* X$ is non-zero by hypothesis, we see that this polynomial does not vanish at any of the $\lambda_j(A_{n-1})$.  Substituting $\lambda_i(A_n)$ for $\lambda$, we obtain \eqref{jn}.  
\end{proof}

%

The following lemma will be useful to control the coordinates of eigenvectors.

\begin{lemma} \label{lemma:firstcoordinate}\cite{ESY1} Let 
$$ A_n = \begin{pmatrix} a & X^* \\ X & A_{n-1} \end{pmatrix}$$
be a $n \times n$ Hermitian matrix for some $a \in \R$ and $X \in \C^{n-1}$, and let $\begin{pmatrix} x \\ v \end{pmatrix}$ be a unit eigenvector of $A_n$ with eigenvalue $\lambda_i(A_n)$, where $x \in \C$ and $v \in \C^{n-1}$.  Suppose that none of the eigenvalues of $A_{n-1}$ are equal to $\lambda_i(A_n)$.  Then
$$|x|^2 =  \frac{1}{1 + \sum_{j=1}^{n-1}  (\lambda_j(A_{n-1})-\lambda_i(A_n))^{-2} |u_j(A_{n-1})^* X|^2}, $$ 
where $u_j(A_{n-1})$ is a unit eigenvector corresponding to the eigenvalue $\lambda_j(A_{n-1})$.
\end{lemma}

\begin{proof} By subtracting $\lambda_i(A) I$ from $A$ we may assume $\lambda_i(A) =0$.  The eigenvector equation then gives
$$ x X + A_{n-1} v = 0,$$
thus
$$ v = - x A_{n-1}^{-1} X.$$
Since $\|v'\|^2 + |x|^2 = 1$, we conclude
$$ |x|^2 (1 + \|A_{n-1}^{-1} X\|^2 ) = 1.$$
Since $\|A_{n-1}^{-1} X\|^2 = \sum_{j=1}^{n-1}  (\lambda_j(A_{n-1}))^{-2} |u_j(A_{n-1})^* X|^2$, the claim follows.
\end{proof}

The \emph{Stieltjes transform} $s_n(z)$ of a Hermitian matrix $W$ is defined for complex $z$ by the formula
$$ s_n(z) := \frac{1}{n} \sum_{i=1}^n \frac{1}{\lambda_i(W)-z}.$$
It has the following alternate representation (see e.g. \cite[Chapter 11]{BS}):

\begin{lemma}\label{stielt}  Let $W = (\zeta_{ij})_{1 \leq i,j \leq n}$ be a Hermitian matrix, and let $z$ be a complex number not in the spectrum of $W$.  Then we have
$$ s_n(z) = \frac{1}{n} \sum_{k=1}^n \frac{1}{\zeta_{kk} - z - a_k^* (W_k - zI)^{-1} a_k}$$
where $W_k$ is the $n-1 \times n-1$ matrix with the $k^{th}$ row and column removed, and $a_k \in \C^{n-1}$ is the $k^{th}$ column of $W$ with the $k^{th}$ entry removed.
\end{lemma}

\begin{proof} By Schur's complement, $\frac{1}{\zeta_{kk} - z - a_k^* (W_k - zI)^{-1} a_k}$ is the $k^{th}$ diagonal entry of $(W-zI)^{-1}$.  Taking traces, one obtains the claim.
\end{proof}

\subsection{Tools from Probability}

We will make frequent use of the following lemma, whose proof is presented in Appendix \ref{section:projection}. This lemma is a
generalization of a result in \cite{TVdet}. 

\begin{lemma}[Distance between a random vector and a subspace]\label{lemma:projection}   Let $X=(\xi_1, \dots, \xi_n) \in \C^n$ be a
random vector whose entries are independent with mean zero, variance $1$, 
and are bounded in magnitude by $K$ almost surely for some $K $, where $K \ge 10( \E |\xi|^{4} +1)$. Let $H$ be a subspace of dimension $d$ and
$\pi_H$ the orthogonal projection onto $H$. Then
$$\P (|\|\pi_H (X)\| - \sqrt d| \ge t) \le 10 \exp(-
\frac{t^{2}}{10K^2}). $$
In particular, one has
$$ \| \pi_H(X)\| = \sqrt{d} + O( K \log n )$$
with overwhelming probability.
\end{lemma}

Another useful tool is the following theorem, which is a corollary of a more general theorem proved in Appendix \ref{section:BE}. 

\begin{theorem}[Tail bounds for complex random walks]\label{bes}  Let $1 \leq N \leq n$ be integers, and let $A = (a_{i,j})_{1 \leq i \leq N; 1 \leq j \leq n}$ be an $N \times n$ complex matrix whose $N$ rows are orthonormal in $\C^n$, and obeying the incompressibility condition
\begin{equation}\label{summer}
 \sup_{1 \leq i \leq N; 1 \leq j \leq n} |a_{i,j}| \leq \sigma
\end{equation}
for some $\sigma > 0$.  Let $\zeta_1,\ldots,\zeta_n$ be independent complex random variables with mean zero, variance $\E |\zeta_j|^2$ equal to $1$, and obeying $\E |\zeta_{i} |^{3} \le C$ for some $C \geq 1$.  For each $1 \leq i \leq N$, let $S_i$ be the complex random variable
$$ S_i := \sum_{j=1}^n a_{i,j} \zeta_j$$
and let $\vec S$ be the $\C^N$-valued random variable with coefficients $S_1,\ldots,S_N$.
\begin{itemize}
\item (Upper tail bound on $S_i$)  For $t \geq 1$, we have $\P( |S_i| \geq t ) \ll \exp(-ct^2) + C \sigma$ for some absolute constant $c>0$.
\item (Lower tail bound on $\vec S$)  For any $t \leq \sqrt{N}$, one has $\P( |\vec S| \leq t) \ll  O( t/\sqrt{N} )^{\lfloor N/4\rfloor} + C N^4 t^{-3} \sigma$.
\end{itemize}
\end{theorem}

\section{Overview of argument}

We now give a high-level proof of our main results, Theorem \ref{theorem:main} and Theorem \ref{ltail}, contingent on several technical propositions that we prove in later sections.

\subsection{Preliminary truncation}

In the hypotheses of Theorem \ref{theorem:main} and Theorem \ref{ltail}, it is assumed that one has the uniform exponential decay property \eqref{ued} on the coefficients $\zeta_{ij}$ on the random matrix $M_n$.  From this and the union bound, we thus see that
$$ \sup_{1 \leq i,j \leq n} |\zeta_{ij}| \leq \log^{C+1} n$$ 
with overwhelming probability.  Since events of probability less than, say, $O(n^{-100})$ are negligible for the conclusion of either Theorem \ref{theorem:main} or Theorem \ref{ltail}, we may thus apply a standard truncation argument (see e.g. \cite{BS}) and \emph{redefine} the atom variables $\zeta_{ij}$ on the events where their magnitude exceeds $\log^{C+1} n$, so that one in fact has
\begin{equation}\label{supi}
 \sup_{1 \leq i,j \leq n} |\zeta_{ij}| \leq \log^{C+1} n
\end{equation}
almost surely.  (This modification may affect the first, second, third, and fourth moments on the real and imaginary parts of the $\zeta_{ij}$ by a very small factor (e.g. $O(n^{-10})$), but one can easily compensate for this by further adjustment of the $\zeta_{ij}$, using the Weyl inequalities \eqref{weyl} if necessary; we omit the details.)  Thus we will henceforth assume that \eqref{supi} holds for proving both Theorem \ref{theorem:main} and Theorem \ref{ltail}.

\begin{remark} If one only assumed some finite number of moment conditions on $\zeta_{ij}$, rather than the exponential condition \eqref{ued}, then one could only truncate the $|\zeta_{ij}|$ to be of size $n^{1/C_0}$ for some constant $C_0$ rather than polylogarithmic in $n$.  While several of our arguments extend to this setting, there is a key induction on $n$ argument in Section \ref{hlt2-sec} that seems to require $|\zeta_{ij}|$ to be of size $n^{o(1)}$ or better, which is the main reason why our results are restricted to random variables of exponential decay.  However, this appears to be a largely technical restriction, and it seems very plausible that the results of this paper can be extended to atom distributions that are only assumed to have a finite number of moments bounded.
\end{remark}

For technical reasons, it is also convenient to make the qualitative assumption that the $\zeta_{ij}$ have an (absolutely) continuous distribution in the complex plane, rather than a discrete one.  This is so that pathological events such as eigenvalue collision will only occur with probability zero and can thus be ignored (though one of course still must deal with the event that two eigenvalues have an extremely small but non-zero separation).  None of our bounds will depend on any quantitative measure of how continuous the $\zeta_{ij}$ are, so one can recover the discrete case from the continuous one by a standard limiting argument (approximating a discrete distribution by a smooth one while holding $n$ fixed, and using the Weyl inequalities \eqref{weyl} to justify the limiting process); we omit the details.  

\subsection{Proof strategy for Theorem \ref{theorem:main}}\label{stratsec}

  For sake of exposition let us restrict attention to the case $k=1$, thus we wish to show that the expectation $\E G(\lambda_i(A_n))$ of the random variable $G(\lambda_i(A_n))$ only changes by $O( n^{-c_0})$ if one replaces $A_n$ with another random matrix $A'_n$ with moments matching up to fourth order off the diagonal (and up to second order on the diagonal).  To further simplify the exposition, let us suppose that the coefficients $\zeta_{pq}$ of $A_n$ (or $A'_n$) are real-valued rather than complex-valued.

At present, $A'_n$ differs from $A_n$ in all $n^2$ components.  But suppose we make a much milder change to $A_n$, namely replacing a single entry $\sqrt{n} \zeta_{pq}$ of $A_n$ with its counterpart $\sqrt{n} \zeta'_{pq}$ for some $1 \leq p \leq q \leq n$.  If $p \neq q$, one also needs to  replace the companion entry $\sqrt{n} \zeta_{qp} = \sqrt{n} \overline{\zeta}_{pq}$ with $\sqrt{n} \zeta'_{qp} = \sqrt{n} \overline{\zeta}'_{pq}$, to maintain the Hermitian property.  This creates another random matrix $\tilde A_n$ which differs from $A_n$ in at most two entries.  Note that $\tilde A_n$ continues to obey Condition {\bf C0}, and has matching moments with either $A_n$ or $A'_n$ up to fourth order off the diagonal, and up to second order on the diagonal.

Suppose that one could show that $\E G( \lambda_i(A_n) )$ differed from $\E G( \lambda_i(\tilde A_n) )$ by at most $n^{-2-c_0}$ when $p \neq q$ and by at most $n^{-1-c_0}$ when $p=q$.  Then, by applying this swapping procedure once for each pair $1 \leq p \leq q \leq n$ and using the triangle inequality, one would obtain the desired bound $|\E G(\lambda_i(A_n) ) - \E G( \lambda_i(A'_n) ) |= O( n^{-c_0} )$.

Now let us see why we would expect $\E G( \lambda_i(A_n) )$ to differ from $\E G( \lambda_i(\tilde A_n) )$ by such a small amount.  For sake of concreteness let us restrict attention to the off-diagonal case $p \neq q$, where we have four matching moments; the diagonal case $p=q$ is similar but one only assumes two matching moments, which is ultimately responsible for the $n^{-1-c_0}$ error rather than $n^{-2-c_0}$.

Let us freeze (or condition on) all the entries of $A_n$ except for the $pq$ and $qp$ entries.  For any complex number $z$, let $A(z)$ denote the matrix which equals $A_n$ except at the $pq$, $qp$, entries, where it equals $z$ and $\overline{z}$ respectively.  (Actually, with our hypotheses, we only need to consider real-valued $z$.)  Thus it would suffice to show that
\begin{equation}\label{eff}
\E F( \sqrt{n} \zeta_{pq} ) = \E F( \sqrt{n} \zeta'_{pq} ) + O( n^{-2-c_0} )
\end{equation}
for all (or at least most) choices of the frozen entries of $A_n$, where $F(z) := G( \lambda_i( A(z) ) )$.  Note from \eqref{supi} that we only care about values of $z$ of size $O( n^{1/2+o(1)} )$.

Suppose we could show the derivative estimates
\begin{equation}\label{fzk}
\frac{d^l}{dz^l} F(z) = O( n^{-l+O(c_0)+o(1)} )
\end{equation}
for $l=1,2,3,4,5$.   (If $z$ were complex-valued rather than real, we would need to differentiate in the real and imaginary parts of $z$ separately, as $F$ is not holomorphic, but let us ignore this technicality for now.)  Then by Taylor's theorem with remainder, we would have
$$ F(z) = F(0) + F'(0) z + \ldots + \frac{1}{4!} F^{(4)}(0) z^4 + O( n^{-5+O(c_0)+o(1)} |z|^5 )$$
and so in particular (using \eqref{supi})
$$ F( \sqrt{n} \zeta_{pq} ) = F(0) + F'(0) \sqrt{n} \zeta_{pq} + \ldots + \frac{1}{4!} F^{(4)}(0) \sqrt{n}^4 \zeta_{pq}^4 + O( n^{-5/2+O(c_0)+o(1)} )$$
and similarly for $F( \sqrt{n} \zeta'_{pq} )$.  Since $n^{-5/2+O(c_0)+o(1)} = O( n^{-2-c_0} )$ for $n$ large enough and $c_0$ small enough, we thus obtain the claim \eqref{eff} thanks to the hypothesis that the first four moments of $\zeta_{pq}$ and $\zeta'_{pq}$ match.  (Note how this argument barely fails if only three moments are assumed to match, though it is possible that some refinement of this argument might still succeed by exploiting further cancellations in the fourth order term $\frac{1}{4!} F^{(4)}(0) \sqrt{n}^4 \zeta_{pq}^4$.)

Now we discuss why one would expect an estimate such as \eqref{fzk} to be plausible.  For simplicity we first focus attention on the easiest case $l=1$, thus we now wish to show that $F'(z) = O( n^{-1+O(c_0)+o(1)} )$.  By \eqref{G-deriv} and the chain rule, it suffices to show that
$$ \frac{d}{dz} \lambda_i(A(z)) = O( n^{-1 + O(c_{0})+ o(1)} ).$$
A crude application of the Weyl bound \eqref{weyl} gives $\frac{d}{dz} \lambda_i(A(z)) = O(1)$, which is not good enough for what we want (although in the actual proof, we will take advantage of a variant of this crude bound to round $z$ off to the nearest multiple of $n^{-100}$, which is useful for technical reasons relating to the union bound).  But we can do better by recalling the \emph{Hadamard first variation formula}
$$ \frac{d}{dz} \lambda_i(A(z)) = u_i(A(z))^* A'(z) u_i(A(z))$$
where we recall that $u_i(A(z))$ is the $i^{th}$ eigenvector of $A(z)$, normalized to be of unit magnitude.  By construction, $A'(z) = e_p e_q^* + e_q e_p^*$, where $e_1,\ldots,e_n$ are the basis vectors of $\C^n$.  
So to obtain the claim, one needs to  show that the coefficients of $u_i(A(z))$ have size $O( n^{-1/2+o(1)} )$.  
This type of delocalization result for eigenvalues has recently been established (with overwhelming probability) by Erd\H{o}s, Schlein, and Yau in \cite{ESY1, ESY2, ESY3}
for Wigner Hermitian matrices, assuming some quantitative control on the continuous distribution of the $\zeta_{pq}$.
(A similar, but weaker, argument was used in \cite{TVhard} with respect to non-Hermitian random matrices; see \cite[Section 4]{TVhard} and \cite[Appendix F]{TVhard}.)  With some extra care and  a new tool (Lemma \ref{lemma:projection}), we are able  extend their arguments to cover  the current 
more general setting (see Proposition \ref{deloc} and Corollary \ref{sdb2}), with a slightly simpler proof. 
Also, $z$ ranges over uncountably many possibilities, so one cannot apply the the union bound to each instance of $z$ separately; instead, one must perform the rounding trick mentioned earlier.

Now suppose we wish to establish the $l=2$ version of \eqref{fzk}.  Again applying the chain rule, we would now seek to establish the bound
\begin{equation}\label{dzz}
 \frac{d^2}{dz^2} \lambda_i(A(z)) = O( n^{-2 + O(c_0) + o(1)} ).
\end{equation}
For this, we apply the Hadamard second variation formula
$$ \frac{d^2}{dz^2} \lambda_i(A(z)) = - 2 u_i(A(z))^* A'(z) \left(A(z)-\lambda_i(A(z)) I\right)^{-1} \pi_{u_i(A(z))^\perp} A'(z) u_i(A(z)),$$
where $\pi_{u_i(A(z))^\perp}$ is the orthogonal projection to the orthogonal complement $u_i(A(z))^\perp$ of $u_i(A(z))$, and $(A(z)-\lambda_i(A(z)) I)^{-1}$ is the inverse of $A(z)-\lambda_i(A(z))$ on that orthogonal complement.  (This formula is valid as long as the eigenvalues $\lambda_j(A(z))$ are simple, which is almost surely the case due to the hypothesis of continuous distribution.)  One can expand out the right-hand side in terms of the other (unit-normalized) eigenvectors $u_j(A(z))$, $j \neq i$ as
$$ \frac{d^2}{dz^2} \lambda_i(A(z)) = - 2 \sum_{j \neq i} \frac{|u_j(A(z))^* A'(z) u_i(A(z))|^2}{\lambda_j(A(z))-\lambda_i(A(z))}.$$
By using Erd\H{o}s-Schlein-Yau type estimates one expects $|u_j(A(z))^* A'(z) u_i(A(z))|$ to be of size about $O(n^{-1+o(1)})$, while from Theorem \ref{ltail} we expect $|\lambda_j(z)-\lambda_i(z)|$ to be bounded below by $n^{-c_0}$ with high probability, and so the claim \eqref{dzz} is plausible (one still needs to sum over $j$, of course, but one expects $\lambda_j(z)-\lambda_i(z)$ to grow roughly linearly in $j$ and so this should only contribute a logarithmic factor $O(\log n) = O(n^{o(1)})$ at worst).  So we see for the first time how Theorem \ref{ltail} is going to be an essential component in the proof of Theorem \ref{theorem:main}.  Similar considerations also apply to the third, fourth, and fifth derivatives of $\lambda_i(A(z))$, though as one might imagine the formulae become more complicated.

There is however a technical difficulty that arises, namely that the lower bound $$|\lambda_j(A(z))-\lambda_i(A(z))| \geq n^{-c_0}$$ holds with \emph{high} probability, but not  with \emph{overwhelming} probability (see Definition \ref{freq-def} for definitions).  Indeed, given that eigenvalue collision is a codimension two event for real symmetric matrices and codimension three for Hermitian ones, one expects the failure probability to be about $n^{-2c_0}$ in the real case and $n^{-3c_0}$ in the complex case (this heuristic is also supported by the gap statistics for GOE and GUE).  As one needs to take the union bound over many values of $z$ (about $n^{100}$ or so), this presents a significant problem.  However, this difficulty can be avoided by going back to the start of the argument and replacing the quantity $G(\lambda_i(z))$ with a ``regularized'' variant which vanishes whenever $\lambda_i(z)$ gets too close to another eigenvalue.  To do this, it is convenient to introduce the quantity
$$ Q_i(A(z)) := \sum_{j \neq i} \frac{1}{|\lambda_j(A(z)) - \lambda_i(A(z))|^2} = \| (A(z)-\lambda_i(z) I)^{-1} \|_F^2;$$
this quantity is normally of size $O(1)$, but becomes large precisely when the gap between $\lambda_i(A(z))$ and other eigenvalues becomes small.  The strategy is then to replace $G(\lambda_i(A(z)))$ by a truncated variant $G(\lambda_i(A(z)),Q_i(A(z)))$ which is supported on the region where $Q_i$ is not too large (e.g. of size at most $n^{c_0}$), and apply the swapping strategy to the latter quantity instead.  (For this, one needs control on derivatives of $Q_i(A(z))$ as well as on $\lambda_i(A(z))$, but it turns out that such bounds are available; this smoothness of $Q_i$ is one reason why we work with $Q_i$ in the first place, rather than more obvious alternatives such as $\inf_{j \neq i} |\lambda_{j}(A(z)) - \lambda_i(A(z))|$.)  Finally, to remove the truncation at the beginning and end of the iterated swapping process, one appeals to Theorem \ref{ltail}. Notice that this result is  now only used twice, rather than $O(n^2)$ or $O(n^{100})$ times, and so the total error probability remains acceptably bounded.

One way to interpret this truncation trick is that while the ``bad event'' that $Q_i$ is large  has reasonably large probability (of order about $n^{-c_0}$), which makes the union bound ineffective, the $Q_i$ does not change too violently when swapping one or more of the entries of the random matrix, and so one is essentially faced with the \emph{same} bad event throughout the $O(n^2)$ different swaps (or throughout the $O(n^{100})$ or so different values of $z$).  So the union bound is actually far from the truth in this case.

\subsection{High-level proof of Theorem \ref{theorem:main}}\label{hlt-sec}

We now begin the rigorous proof of Theorem \ref{theorem:main}, breaking it down into simpler propositions which will be proven in subsequent sections.

The heart of the argument consists of  two key propositions.  The first proposition asserts that one can swap a single coefficient (or more precisely, two coefficients) of a (deterministic) matrix $A$ as long as $A$ obeys a certain ``good configuration condition'':

\begin{proposition}[Replacement given a good configuration]\label{swap}  There exists a positive constant $C_1$ such that the following holds.  Let $k \geq 1$ and $\eps_1 > 0$, and assume $n$ sufficiently large depending on these parameters.  Let $1 \leq i_1 < \ldots < i_k \leq n$.
For a complex parameter $z$, let $A(z)$ be a (deterministic) family of $n \times n$ Hermitian matrices of the form
$$ A(z) = A(0) + z e_p e_q^* + \overline{z} e_q e_p^*$$
where $e_p, e_q$ are unit vectors.  We assume that for every $1 \leq j \leq k$ and every $|z| \leq n^{1/2+\eps_1}$ whose real and imaginary parts are multiples of $n^{-C_1}$, we have
\begin{itemize}
\item (Eigenvalue separation)  For any $1 \leq i \leq n$ with $|i-i_j| \geq n^{\eps_1}$, we have
\begin{equation}\label{noon}
 |\lambda_i(A(z)) - \lambda_{i_j}(A(z))| \geq n^{-\eps_1} |i-i_j|.
\end{equation}
\item (Delocalization at $i_j$)  If $P_{i_j}(A(z))$ is the orthogonal projection to the eigenspace associated to $\lambda_{i_j}(A(z))$, then
\begin{equation}\label{pz1}
 \| P_{i_j}(A(z)) e_p \|, \| P_{i_j}(A(z)) e_q \| \leq n^{-1/2+\eps_1}.
\end{equation}
\item For every $\alpha \geq 0$
\begin{equation}\label{pz2}
\| P_{i_j,\alpha}(A(z)) e_p \|, \| P_{i_j,\alpha}(A(z)) e_q \| \leq 2^{\alpha/2} n^{-1/2+\eps_1},
\end{equation}
whenever $P_{i_j,\alpha}$ is the orthogonal projection to the eigenspaces corresponding to eigenvalues $\lambda_i(A(z))$ with $2^\alpha \leq |i-i_j| < 2^{\alpha+1}$.
\end{itemize}
We say that $A(0), e_p, e_q$ are a \emph{good configuration} for $i_1,\ldots,i_k$ if the above properties hold.  Assuming this good configuration, then we have
\begin{equation}\label{barg}
\E (F(\zeta)) = \E F(\zeta') + O( n^{-(r+1)/2 + O(\eps_1)} ), 
\end{equation}
whenever
$$
F(z) := G( \lambda_{i_1}(A(z)),\ldots,\lambda_{i_k}(A(z)),Q_{i_1}(A(z)),\ldots,Q_{i_k}(A(z))),$$
and
$$ G = G( \lambda_{i_1},\ldots,\lambda_{i_k}, Q_{i_1},\ldots,Q_{i_k} )$$
is a smooth function from $\R^k \times \R_+^k \to \R$ that is supported on the region
$$ Q_{i_1},\ldots,Q_{i_k} \leq n^{\eps_1}$$
and obeys the derivative bounds
$$ |\nabla^j G| \leq n^{\eps_1}$$
for all $0 \leq j \leq 5$, and $\zeta, \zeta'$ are random variables with $|\zeta|, |\zeta'| \leq n^{1/2+\eps_1}$ almost surely, which 
match to order $r$ for some $r=2,3,4$.  

If $G$ obeys the improved derivative bounds
$$ |\nabla^j G| \leq n^{-C j \eps_1}$$
for $0 \leq j \leq 5$ and some sufficiently large absolute constant $C$, then we can strengthen $n^{-(r+1)/2 + O(\eps_1)}$ in \eqref{barg} to $ n^{-(r+1)/2 - \eps_1}$.
\end{proposition}

\begin{remark} The need to restrict $z$ to multiples of $n^{-C_1}$, as opposed to all complex $z$ in the disk of radius $n^{1/2+\varepsilon_1}$, is so that we can verify the hypotheses in the next proposition using the union bound (so long as the events involved hold with overwhelming probability).  For $C_1$ large enough, we will be able to use rounding methods to pass from the discrete setting of multiples of $n^{-C_1}$ to the continuous setting of arbitrary complex numbers in the disk without difficulty.
\end{remark}

We prove this proposition in Section \ref{swap-sec}.  To use this proposition, we of course need to have the good configuration property hold often.  This leads to the second key proposition:

\begin{proposition}[Good configurations occur very frequently]\label{lemma:GCC}
Let $\eps, \eps_1 > 0$ and $C, C_1, k \geq 1$.  Let $\eps n \leq i_1 < \ldots < i_k \leq (1-\eps) n$, let $1 \leq p,q \leq n$, let $e_1,\ldots,e_n$ be the standard basis of $\C^n$, and let $A(0) = (\zeta_{ij})_{1 \leq i,j \leq n}$ be a random Hermitian matrix with independent upper-triangular entries and $|\zeta_{ij}| \leq n^{1/2} \log^C n$ for all $1 \leq i,j \leq n$, with $\zeta_{pq}=\zeta_{qp}=0$, but with $\zeta_{ij}$ having mean zero and variance $1$ for all other $ij$, and also being distributed continuously in the complex plane.  Then $A(0),e_p,e_q$ obey the Good Configuration Condition in Theorem \ref{swap} for $i_1,\ldots,i_k$ and with the indicated value of $\eps_1, C_1$ with overwhelming probability.
\end{proposition}

We will prove this proposition in Section \ref{gcc-sec}.

Given these two propositions (and Theorem \ref{ltail}) we can now prove Theorem \ref{theorem:main}.  As discussed at the beginning of the section, we may assume that the $\zeta_{ij}$ are continuously distributed in the complex plane and obey the bound \eqref{supi}.  

Let $0 < \eps < 1$ and $k \geq 1$, and assume $c_0$ is sufficiently small and $C_1$ sufficiently large.  Let $M_n, M'_n, \zeta_{ij}, \zeta'_{ij}, A_n, A'_n, G$, $i_1,\ldots,i_k$ be as in Theorem \ref{theorem:main}.

We first need

\begin{lemma}\label{jordan} For each $1 \leq j \leq k$, one has $Q_{i_j}(A_n) \leq n^{c_0}$ with high probability.
\end{lemma}

\begin{proof}  For brevity we omit the $A_n$ variable.  Fix $j$.  Suppose that $Q_{i_j} > n^{c_0}$, then
$$ \sum_{i \neq i_j} \frac{1}{|\lambda_i - \lambda_{i_j}|^2} > n^{c_0}$$
and so by the pigeonhole principle there exists an integer $0 \leq m \ll \log n$ such that 
$$ \sum_{2^m \leq |i-i_j| < 2^{m+1}} \frac{1}{|\lambda_i - \lambda_{i_j}|^2} \gg 2^{-m/2} n^{c_0}$$
which implies that 
$$ |\lambda_{i_j + 2^m} - \lambda_{i_j}| \ll 2^{\frac{3}{4} m} n^{-c_0/2}$$
or
$$ |\lambda_{i_j - 2^m} - \lambda_{i_j}| \ll 2^{\frac{3}{4} m} n^{-c_0/2}.$$
It thus suffices to show that
$$ \P( |\lambda_{i_j + 2^m} - \lambda_{i_j}| \ll 2^{\frac{3}{4} m} n^{-c_0/2} ) \leq n^{-c_1}$$
uniformly in $m$ (and similarly for $\lambda_{i_j-2^m}$), since the $\log n$ loss caused by the number of $m$'s can easily be absorbed into the right-hand side.

Fix $m$.  Suppose that $|\lambda_{i_j + 2^m} - \lambda_{i_j}| \ll 2^{\frac{3}{4} m} n^{-c_0/2}$; then expressing the left-hand side as $\sum_{k=0}^{2^m-1} \lambda_{i_j+k+1}-\lambda_{i_j+k}$ and using Markov's inequality we see that
$$ \lambda_{i_j+k+1}-\lambda_{i_j+k} \ll n^{-c_0/2}$$
for $\gg 2^m$ values of $k$, and thus
$$ \P( |\lambda_{i_j + 2^m} - \lambda_{i_j}| \ll 2^{\frac{3}{4} m} n^{-c_0/2} ) \ll \E \frac{1}{2^m} \sum_{k=0}^{2^m-1} \I( \lambda_{i_j+k+1}-\lambda_{i_j+k} \ll n^{-c_0/2} )$$
and hence by linearity of expectation
$$ \P( |\lambda_{i_j + 2^m} - \lambda_{i_j}| \ll 2^{\frac{3}{4} m} n^{-c_0/2} ) \ll \frac{1}{2^m} \sum_{k=0}^{2^m-1} \P( \lambda_{i_j+k+1}-\lambda_{i_j+k} \ll n^{-c_0/2} ).$$
The claim now follows from Theorem \ref{ltail}.  (There is a slight issue when $2^m \sim n$, so that the index $i_j+k$ may leave the bulk; but then one works with, say, $\lambda_{i_j + 2^{m-1}} - \lambda_{i_j}$ instead of $\lambda_{i_j + 2^{m}} - \lambda_{i_j}$).
\end{proof}

\begin{remark} One can also use Theorem \ref{sdb} below to control all terms in the sum with $|i-i_j| \gg \log^{C'} n$ for some $C'$, leading to a simpler proof of Lemma \ref{jordan}.
\end{remark}

Of course, Lemma \ref{jordan} also applies with $A_n$ replaced by $A'_n$.

Let $\tilde G: \R^k \times \R_+^k \to \R$ be the function
$$ \tilde G( \lambda_{i_1},\ldots,\lambda_{i_k}, Q_{i_1},\ldots,Q_{i_k}) := G(\lambda_{i_1},\ldots,\lambda_{i_k}) \prod_{j=1}^k \eta( Q_{i_j} )$$
where $\eta(x)$ is a smooth cutoff to the region $x \leq n^{c_0}$ which equals $1$ on $x \leq n^{c_0}/2$.  From \eqref{G-deriv} and the chain rule we see that
$$ |\nabla^j \tilde G| \ll n^{c_0}$$
for $j=0,1,2,3,4,5$.  Also, from Lemma \ref{jordan} we have
$$
 |\E ( G(\lambda_{i_1}(A_n), \dots, \lambda_{i_k}(A_n))) -
 \E ( \tilde G(\lambda_{i_1}(A_n), \dots, \lambda_{i_k}(A_n), Q_{i_1}(A_n),\ldots,Q_{i_k}(A_n)))| \ll n^{-c}$$
for some $c>0$, and similarly with $A_n$ replaced by $A'_n$.  Thus (by choosing $c_0$ small enough) to prove \eqref{eqn:approximation} it will suffice to show that the quantity
\begin{equation}\label{gaa}
 \E ( \tilde G(\lambda_{i_1}(A_n), \dots, \lambda_{i_k}(A_n), Q_{i_1}(A_n),\ldots,Q_{i_k}(A_n)))
\end{equation}
only changes by at most $n^{-c_0}/2$ (say) when one replaces $A_n$ by $A'_n$.

As discussed in Section \ref{stratsec}, it will suffice to show that the quantity \eqref{gaa} changes by at most $n^{-2-c_0}/4$ when one swaps the $\zeta_{pq}$ entry with $1 \leq p < q \leq n$ to $\zeta'_{pq}$ (and $\zeta_{qp}$ with $\zeta'_{qp}$), and changes by at most $n^{-1-c_0}/4$ when one swaps a diagonal entry $\zeta_{pp}$ with $\zeta'_{pp}$.  But these claims follow from Proposition \ref{lemma:GCC} and Proposition \ref{swap}. (The last part of Proposition \ref{swap} is used in the case when one only has three moments matching rather than four.) 

The proof of Theorem \ref{theorem:main} is now complete (contingent on Theorem \ref{ltail}, Proposition \ref{swap}, and Proposition \ref{lemma:GCC}).

\subsection{Proof strategy for Theorem \ref{ltail}}

We now informally discuss the proof of Theorem \ref{ltail}.  

The machinery of Erd\H{o}s, Schlein, and Yau\cite{ESY1,ESY2,ESY3}, which is useful in particular for controlling the Stieltjes transform of Wigner matrices, will allow us to obtain good lower bounds on the spectral gap $\lambda_{i}(A_n)-\lambda_{i-1}(A_n)$ in the bulk as soon as $k \gg \log^{C'} n$ for a sufficiently large $C$'; see Theorem \ref{sdb} for a precise statement.  The difficulty here is that $k$ is exactly  $1$.  To overcome this difficulty, we will try to  amplify the value of $k$ by looking at the top left $n-1 \times n-1$ minor $A_{n-1}$ of $A_n$, and observing the following ``backwards gap propagation'' phenomenon:

\emph{If $\lambda_{i}(A_n)-\lambda_{i-k}(A_n)$ is very small, then $\lambda_{i}(A_{n-1})-\lambda_{i-k-1}(A_{n-1})$ will also be small with reasonably high probability.}

If one accepts this phenomenon, then by iterating it about $\log^{C'} n$ times one can enlarge the spacing $k$ to be of the size 
large enough so that a  Erd\H{o}s-Schlein-Yau  type bound can be invoked to obtain a contradiction.  (There will be a technical difficulty caused by the fact that the failure probability of this phenomenon, when suitably quantified, can be as large as $1/\log^{O(1)} n$, thus apparently precluding the ability to get a polynomially strong bound on the failure rate, but we will address this issue later.)

Note that the converse of this statement follows from the Cauchy interlacing property \eqref{cauchy-interlace}.  To explain why this phenomenon is plausible, observe from \eqref{cauchy-interlace} that if $\lambda_{i}(A_n)-\lambda_{i-k}(A_n)$ is small, then $\lambda_i(A_{n}) - \lambda_{i-k-1}(A_{n-1})$ is also small. On the other hand, from Lemma \ref{jn-lem} one has the identity
\begin{equation}\label{jn2}
 \sum_{j=1}^{n-1} \frac{|u_j(A_{n-1})^* X|^2}{\lambda_j(A_{n-1}) - \lambda_i(A_n)} = \sqrt{n} \zeta_{nn} - \lambda_i(A_n),
\end{equation}
where $X$ is the rightmost column of $A_n$ (with the bottom entry $\sqrt{n} \zeta_{nn}$ removed).

One expects $|u_j(A_{n-1})^* X|^2$ to have size about $n$ on the average (cf. Lemma \ref{lemma:projection}).  In particular, if $\lambda_i(A_{n}) - \lambda_{i-k}(A_{n-1})$ is small (e.g. of size $O(n^{-c})$), then the $j=i-1$ term is expected give a large negative contribution (of size $\gg n^{1+c}$) to the left-hand side of \eqref{jn2}. Meanwhile, the right-hand side is much smaller, of size $O(n)$ or so on the average; so we expect to have the large negative contribution mentioned earlier to be counterbalanced by a large positive contribution from some other index.  The index which is most likely to supply such a large positive contribution is $j=i$, and so one expects $\lambda_i(A_{n-1}) - \lambda_{i}(A_{n})$ to be small (also of size $O(n^{-c})$, in fact).  A similar argument also leads one to expect $\lambda_{i-k}(A_{n}) - \lambda_{i-k-1}(A_n)$ to be small, and the claimed phenomenon then follows from the triangle inequality.

In order to make the above strategy rigorous, there are a number of technical difficulties.  The first is that the counterbalancing term mentioned above need not come from $j=i$, but could instead come from another value of $j$, or perhaps a ``block'' of several $j$ put together, and so one may have to replace the gap $\lambda_i(A_{n-1}) - \lambda_{i-k-1}(A_{n-1})$ by a more general type of gap.  A second problem is that the gap $\lambda_{i}(A_{n-1})-\lambda_{i-k-1}(A_{n-1})$ is going to be somewhat larger than the gap $\lambda_{i}(A_{n})-\lambda_{i-k}(A_{n})$, and one is going to be iterating this gap growth about $\log^{O(1)} n$ times.  In order to be able to contradict Theorem \ref{sdb} at the end of the argument, the net gap growth should only be $O(n^c)$ at most for some small $c > 0$.  So one needs a reasonable control on the ratio between the gap for $A_{n-1}$ and the gap for $A_n$; in particular, if one can keep the former gap to be at most $(1 + \frac{1}{k})^{O(\log^{0.9} n)}$ times the latter gap, then the net growth in the gap telescopes to $(\log^{O(1)} n)^{O(\log^{0.9} n)}$, which is indeed less than $O(n^c)$ and thus acceptable.  To address these issues, we fix a base value $n_0$ of $n$, and for any $1 \leq i-l < i \leq n \leq n_0$, we define the \emph{regularized gap}
\begin{equation}\label{giln}
g_{i,l,n} := \inf_{1 \leq i_- \leq i-l < i \leq i_+ \leq n} \frac{\lambda_{i_+}(A_n)-\lambda_{i_-}(A_n)}{\min( i_+-i_-, \log^{C_1} n_0 )^{\log^{0.9} n_0}},
\end{equation}
where $C_1 > 1$ is a large constant (depending on $C$) to be chosen later.  (We need to cap $i_+-i_-$ off at $\log^{C_1} n_0$ to prevent the large values of $i_+-i_-$ from overwhelming the infimum, which is not what we want.)

We will shortly establish a rigorous result that asserts, roughly speaking, that if the gap $g_{i,l,n+1}$ is small, then the gap $g_{i,l+1,n}$ is also likely to be small, thus giving a precise version of the phenomenon mentioned earlier.

There is one final obstacle, which has to do with the failure probability when $g_{i,l,n+1}$ is small but $g_{i,l+1,n}$ is large.  If this event could be avoided with overwhelming probability (or even a high probability), then one would be done by the union bound (note we only need to take the union over $O(\log^{O(1)} n)$ different events).  While many of the events that could lead to failure can indeed be avoided with high probability, there is one type of event which does cause a serious problem, namely that the inner products $u_j(A_{n-1})^* X$ for $i_- \leq j \leq i_+$ could be unexpectedly small.  Talagrand's inequality (Lemma \ref{lemma:projection}) can be used to control this event effectively when $i_+-i_-$ is large, but when $i_+-i_-$ is small the probability of failure can be as high as $1/\log^c n$ for some $c>0$.  However, one can observe that such high failure rates only occur when $g_{i,l+1,n}$ is only \emph{slightly} larger than $g_{i,l,n+1}$.  Indeed, one can show that the probability that $g_{i,l,n+1}$ is much higher than $g_{i,l+1,n}$, say of size $2^m g_{i,l,n+1}$ or more, is only $O( 2^{-m/2} / \log^c n )$ (for reasonable values of $m$), and in fact (thanks to Talagrand's inequality) the constant $c$ can be increased to be much larger when $l$ is large.  This is still not quite enough for a union bound to give a total failure probability of $O(n^{-c})$, but one can exploit the martingale-type structure of the problem (or more precisely, the fact that the column $X$ remains random, with independent entries, even after conditioning out all of the block $A_{n-1}$) to multiply the various bad failure probabilities together to end up with the final bound of $O(n^{-c})$.

\subsection{High-level proof of Theorem \ref{ltail}}\label{hlt2-sec}

We now prove Theorem \ref{ltail}.  Fix $\eps, c_0$.  We write $i_0$ and $n_0$ for $i, n$, thus
$$ \eps n_0 \leq i_0 \leq (1-\eps) n_0$$
and the task is to show that $|\lambda_{i_0}(A_{n_0}) - \lambda_{i_0}(A_{n_0-1})| \geq n_0^{-c_0}$ with high probability.  We can of course assume that $n_0$ is large compared to all other parameters.  We can also assume the bound \eqref{supi}, and that the distribution of the $A_n$ is continuous, so that events such as repeated eigenvalues occur with probability zero and can thus be neglected.

We let $C_1$ be a large constant to be chosen later.
For any $l, n$ with $1 \leq i-l < i \leq n \leq n_0$, we define the normalized gap $g_{i,l,n}$ by \eqref{giln}.  It will suffice to show that
\begin{equation}\label{goin}
g_{i_0,1,n_0} \leq n^{-c_0}
\end{equation}
with high probability.  As before, we let $u_1(A_n),\ldots,u_n(A_n)$ be an orthonormal eigenbasis of $A_n$ associated to the eigenvectors $\lambda_1(A_n),\ldots,\lambda_n(A_n)$.  We also let $X_n \in \C^n$ be the rightmost column of $A_{n+1}$ with the bottom coordinate $\sqrt{n} \zeta_{n+1,n+1}$ removed.

The first main tool for this will be the following (deterministic) lemma, proven in Section \ref{backprop-sec}.

\begin{lemma}[Backwards propagation of gap]\label{backprop}  Suppose that $n_0/2 \leq n < n_0$ and $l \leq \eps n/10$ is such that
\begin{equation}\label{gdel}
g_{i_0,l,n+1} \leq \delta
\end{equation}
for some $0 < \delta \leq 1$ (which can depend on $n$), and that
\begin{equation}\label{gilp}
 g_{i_0,l+1,n} \geq 2^m g_{i_0,l,n+1}
\end{equation}
for some $m \geq 0$ with
\begin{equation}\label{mcivil}
2^m \leq \delta^{-1/2}.
\end{equation}
Then one of the following statements hold:
\begin{itemize}
\item[(i)]  (Macroscopic spectral concentration) There exists $1 \leq i_- < i_+ \leq n+1$ with $i_+-i_- \geq \log^{C_1/2} n$ such that $|\lambda_{i_+}(A_{n+1}) - \lambda_{i_-}(A_{n+1})| \leq \delta^{1/4} \exp( \log^{0.95} n ) (i_+-i_-)$.
\item[(ii)]  (Small inner products)  There exists $\eps n/2 \leq i_- \leq i_0-l < i_0 \leq i_+ \leq (1-\eps/2) n$ with $i_+-i_- \leq \log^{C_1/2} n$ such that
\begin{equation}\label{smallin}
\sum_{i_- \leq j < i_+} |u_j(A_n)^* X_n|^2 \leq \frac{n (i_+-i_-)}{2^{m/2} \log^{0.01} n}
\end{equation}
\item[(iii)]  (Large coefficient)  We have
$$ |\zeta_{n+1,n+1}| \geq n^{0.4}.$$
\item[(iv)] (Large eigenvalue)  For some $1 \leq i \leq n+1$ one has
$$ |\lambda_i(A_{n+1})| \geq \frac{n \exp( -\log^{0.95} n )}{\delta^{1/2}}.$$
\item[(v)] (Large inner product in bulk)  There exists $\eps n/10 \leq i \leq (1-\eps/10) n$ such that
$$ |u_i(A_n)^* X_n|^2 \geq \frac{n \exp( - \log^{0.96} n )}{\delta^{1/2}}.$$
\item[(vi)] (Large row)  We have
$$ \|X_n\|^2 \geq \frac{n^2 \exp( - \log^{0.96} n )}{\delta^{1/2}}.$$
\item[(vii)] (Large inner product near $i_0$)  There exists $\eps n/10 \leq i \leq (1-\eps/10) n$ with $|i-i_0| \leq \log^{C_1} n$ such that
$$ |u_i(A_n)^* X_n|^2 \geq 2^{m/2} n \log^{0.8} n.$$
\end{itemize}
\end{lemma}

\begin{remark} In applications $\delta$ will be a small negative power of $n$.  The main bad event here is (ii) (and to a lesser extent, (vii)); the other events will have a polynomially small probability of occurrence in practice (as a function of $n$) and so can be easily discarded.  The events (ii), (vii) are more difficult to discard, 
 since their probability is not polynomially small in $n$, if $m$ is small. 
 On the other hand, these probabilities decay exponentially in $m$, and furthermore are independent in a martingale sense, and this will be enough for us to obtain a proper control.  The exact numerical values of the exponents such as $0.9$, $0.95$, $0.8$, etc. are not particularly important, though of course they need to lie between $0$ and $1$.
\end{remark}

The second key proposition bounds the probability that each of the bad events (i)-(vii) occur, proven in Section \ref{bad-event-sec}.

\begin{proposition}[Bad events are rare]\label{bad-event}  Suppose that $n_0/2 \leq n < n_0$ and $l \leq \eps n/10$, and set $\delta := n_0^{-\kappa}$ for some sufficiently small fixed $\kappa > 0$.  Then:
\begin{itemize}
\item[(a)] The events (i), (iii), (iv), (v), (vi) in Lemma \ref{backprop} all fail with high probability.  
\item[(b)] There is a constant $C'$ such that all the coefficients of the eigenvectors $u_j(A_n)$ for $\eps n/2 \leq j \leq (1-\eps/2) n$ are of magnitude at most $n^{-1/2} \log^{C'} n$ with overwhelming probability.  Conditioning $A_n$ to be a matrix with this property, the events (ii) and (vii) occur with a conditional probability of at most $2^{-\kappa m} + n^{-\kappa}$.  
\item[(c)] Furthermore, there is a constant $C_2$ (depending on $C',\kappa,C_1$) such that if $l \geq C_2$ and $A_n$ is conditioned as in (b), then (ii) and (vii) in fact occur with a conditional probability of at most $2^{-\kappa m} \log^{-2C_1} n + n^{-\kappa}$.
\end{itemize}
\end{proposition}

Let us assume these two propositions for now and conclude the proof of Theorem \ref{ltail}.

We may assume $c_0$ is small.  Set $\kappa := c_0/10$.
For each $n_0-\log^{2C_1} n_0 \leq n \leq n_0$, let $E_n$ be the event that one of the eigenvectors $u_j(A_n)$ for $\eps n/2 \leq j \leq (1-\eps/2) n$ has a coefficient of magnitude more than $n^{-1/2} \log^{C'}_n$, and let $E_0$ be the event that at least one of the exceptional events (i), (iii)-(vi), or $E_n$ hold for some $n-\log^{2C_1} n_0 \leq n \leq n_0$; then by Proposition \ref{bad-event} and the union bound we have
\begin{equation}\label{peo}
\P( E_0 ) \leq n_0^{-\kappa/2}.
\end{equation}
(say).  It thus suffices to show that the event
$$ g_{i_0,1,n_0} \leq n^{-10\kappa} \wedge E_0^c $$
is avoided with high probability.

To bound this, the first step is to increase the $l$ parameter from $1$ to $C_2$, in order to use Proposition \ref{bad-event}(c).  Set $2^m := n_0^{\kappa/C_2}$.  From Proposition \ref{bad-event}(b), we see that the event that $E_n$ fails, but (ii) or (vii) holds for $n=n_0-1$, $l=1$, and some $i_-,i_+$ occurs with probability $O( 2^{-\kappa m} + n_0^{-\kappa} )$.  Applying Lemma \ref{backprop} (noting that $n_0^{-10\kappa} \leq \delta$) we conclude that
$$ \P( g_{i_0,1,n_0} \leq n_0^{-10\kappa} \wedge E_0^c ) \leq \P( g_{i_0,2,n_0-1} \leq 2^m n_0^{-10\kappa} \wedge E_0^c ) + O( 2^{-\kappa m} + n_0^{-\kappa} ).$$
We can iterate this process $C_2$ times and conclude
$$ \P( g_{i_0,1,n_0} \leq n_0^{-10\kappa} \wedge E_0^c ) \leq \P( g_{i_0,C_2+1,n_0-C_2} \leq 2^{C_2 m} n^{-10\kappa} \wedge E_0^c ) + O( C_2 2^{-\kappa m} + C_2 n_0^{-\kappa} );$$
substituting in the definition of $m$, we conclude
$$ \P( g_{i_0,1,n_0} \leq n_0^{-10\kappa} \wedge E_0^c ) \leq \P( g_{i_0,C_2+1,n_0-C_2} \leq n_0^{-9\kappa} \wedge E_0^c ) + n_0^{-\kappa^2 / 2C_2}$$
(say).  So it will suffice to show that
$$ \P( g_{i_0,C_2+1,n_0-C_2} \leq n_0^{-9\kappa} \wedge E_0^c ) \leq n_0^{-\kappa^2/2C_2}.$$

By Markov's inequality, it suffices to show that
\begin{equation}\label{kappa1}
 \E Z_{n_0-C_2}^{-\kappa/2} \I(E_0^c) \leq n_0^{\kappa^2}
\end{equation}
(say), where for each $n_0-\log^{C_1} n_0 \leq n \leq n_0-C_2$, $Z_n$ is the random variable
$$ Z_n := \max(\min( g_{i_0,n_0-n+1,n}, \delta ), n_0^{-9\kappa}).$$
Indeed, we have
\begin{align*}
\P( g_{i_0,C_2+1,n_0-C_2} \leq n_0^{-9\kappa} \wedge E_0^c ) &\leq \P( Z_{n_0-C_2} = n_0^{-9\kappa} \wedge E_0^c )\\
&\leq n_0^{-9\kappa^2/2} \E Z_{n_0-C_2}^{-\kappa/2} \I(E_0^c) 
\end{align*}
whence the claim.

We now establish a recursive inequality for $\E Z_n^{-\kappa/2} \I(E_0^c)$.  Let $n_0-\log^{C_1} n_0 \leq n < n_0-C_2$.  Suppose we condition $A_n$ so that $E_n$ fails.  Then for any $m \geq 0$, we see from Proposition \ref{bad-event}(c) that (ii) or (vii) holds for $l=n_0-n$ and some $i_-,i_+$ with (conditional) probability at most $O(2^{-\kappa m} \log^{-2C_1} n + n^{-\kappa})$.  Applying Lemma \ref{backprop}, we conclude that
$$ \P( g_{i_0,n_0-n,n+1} \leq \delta \wedge  g_{i_0,n_0-n+1,n} \geq 2^m g_{i_0,n_0-n,n+1} \wedge E_0^c | A_n )
\ll 2^{-\kappa m} \log^{-2C_1} n + n^{-\kappa}.$$
Note that this inequality is also vacuously true of $A_n$ is such that $E_n$ holds, since the event $E_0^c$ is then empty.

Observe that if $Z_n > 2^m Z_{n+1}$ for some $m \geq 0$ then $g_{i_0,n_0-n,n+1} \leq \delta \wedge  g_{i_0,n_0-n+1,n} \geq 2^m g_{i_0,n_0-n,n+1}$.  Thus
$$ \P( Z_n > 2^m Z_{n+1} \wedge E_0^c | A_n )
\ll 2^{-\kappa m} \log^{-2C_1} n + n^{-\kappa}$$
or equivalently
$$ \P( Z_{n+1}^{-\kappa/2} > 2^{m\kappa/2} Z_n^{-\kappa/2} \wedge E_0^c | A_n )
\ll 2^{-\kappa m} \log^{-2C_1} n + n^{-\kappa}.$$
Since we are conditioning on $A_n$, $Z_n$ is deterministic.   Also, from the definition of $Z_n$, this event is vacuous for $2^m > n^{8\kappa}$, thus we can simplify the above bound as
$$ \P( Z_{n+1}^{-\kappa/2} > 2^{m\kappa/2} Z_n^{-\kappa/2} \wedge E_0^c | A_n )
\leq 3 \times 2^{-\kappa m} \log^{-2C_1} n$$
Now we multiply this by $2^{m \kappa/2}$ and sum over $m \geq 0$ to obtain
$$ \E( Z_{n+1}^{-\kappa/2} \I( E_0^c ) | A_n ) \leq Z_n^{-\kappa/2} ( 1 + \log^{-(2C_1-1)} n )$$
(say).  Undoing the conditioning on $A_n$, we conclude
$$ \E( Z_{n+1}^{-\kappa/2} \I( E_0^c ) ) \leq
( 1 + \log^{-(2C_1-1)} n ) \E( Z_n^{-\kappa/2} ).$$
Applying \eqref{peo} (and the trivial bound $Z_n^{-\kappa/2} \leq n^{9\kappa^2/2}$) we have
$$ \E( Z_{n+1}^{-\kappa/2} \I( E_0^c ) ) \leq
( 1 + \log^{-(2C_1-1)} n ) \E( Z_n^{-\kappa/2} \I(E_0^c) ) + n^{-\kappa/4}$$
(say).  Iterating this we conclude that
\begin{equation}\label{kappa2}
\E( Z_{n_0-C_2}^{-\kappa/2} \I( E_0^c ) ) \leq
2 \E( Z_{n_0-\lfloor \log^{C_1} n_0 \rfloor}^{-\kappa/2} \I(E_0^c) ) + n_0^{-\kappa/8}
\end{equation}
(say).

On the other hand, if $E_0^c$ holds, then by (i) we have
$$ \left|\lambda^{(n_0-\lfloor \log^{C_1} n_0 \rfloor)}_{i_+} - \lambda^{(n_0-\lfloor \log^{C_1} n_0 \rfloor)}_{i_-}\right| < n_0^{-\kappa} \exp( \log^{0.95} n ) (i_+-i_-)$$
whenever $1 \leq i_- \leq i-\log^{C_1/2} n_0 < i \leq i_+ \leq n$.  From this we have
$$ g_{i_0,\lfloor \log^{C_1} n_0 \rfloor+1,n_0-\lfloor \log^{C_1} n_0 \rfloor} \leq n_0^{-\kappa/2}$$
(say) and hence
$$ Z_{n_0-\lfloor \log^{C_1} n_0 \rfloor} = n_0^{-\kappa}.$$
Inserting this into \eqref{kappa2} we obtain \eqref{kappa1} as required.

\section{Good configurations have stable spectra}\label{swap-sec}

The purpose of this section is to prove Proposition \ref{swap}.  The first stage is to obtain some equations for the derivatives of eigenvalues of Hermitian matrices with respect to perturbations.

\subsection{Spectral dynamics of Hermitian matrices}

Suppose that $\lambda_i(A)$ is a \emph{simple} eigenvalue, which means that
 $\lambda_i(A) \neq \lambda_j(A)$ for all $j \neq i$; note that almost all Hermitian matrices have simple eigenvalues.  We then define $P_i(A)$
 to be the orthogonal projection to the one-dimensional eigenspace corresponding
  to $\lambda_i(A)$; thus, if $u_i(A)$ is a unit eigenvector for the eigenvalue
  $\lambda_i(A)$, then $P_i(A) = u_i(A) u_i(A)^*$.  We also define the \emph{resolvent}
   $R_i(A)$ to be the unique Hermitian matrix inverting $A - \lambda_i(A) I$ on the range of
    $I - P_i(A)$, and vanishing on the range of $P_i(A)$. If $u_1(A),\ldots,u_n(A)$ form an orthonormal eigenbasis
    associated to the $\lambda_1(A),\ldots,\lambda_n(A)$, we can write $R_i(A)$
    explicitly as
$$ R_i(A) = \sum_{j \neq i} \frac{1}{\lambda_j-\lambda_i} u_j(A) u_j(A)^*.$$

It is clear that 

\begin{equation} \label{RA} R_i (A-\lambda_i I) = I -P_i . \end{equation} 

We also need the quantity
$$ Q_i(A) := \|R_i(A)\|_F^2 = \sum_{j \neq i} \frac{1}{|\lambda_j-\lambda_i|^2}.$$

By \eqref{weyl},  each eigenvalue function $A \mapsto \lambda_i(A)$ for $1 \leq i \leq n$
is continuous.  However, we will  need a  quantitative
control on the derivatives of this function.  The first observation
 is that $\lambda_i$ (and $P_i, R_i, Q_i$) depend smoothly on $A$
 whenever that eigenvalue is simple (even if other eigenvalues have multiplicity):

\begin{lemma} Let $1 \leq i \leq n$, and let $A_0$ be a Hermitian matrix
which has a simple eigenvalue at $\lambda_i(A_0)$.
Then $\lambda_i$, $P_i$, $R_i$, and $Q_i$ are smooth for $A$
in a neighborhood of $A_0$.
\end{lemma}

\begin{proof}  By the Weyl inequality \eqref{weyl}, $\lambda_i(A_0)$ stays
away from the other $\lambda_j(A_0)$ by a bounded distance for
all $A_0$ in a neighborhood of $A_0$.  In particular,
 the characteristic polynomial $\det( A - \lambda I )$ has a
 simple zero at $\lambda_i(A)$ for all such $A$.  Since this
 polynomial also depends smoothly on $A$, the smoothness of $\lambda_i$ now follows.
 As $A - \lambda_i(A) I$ depends smoothly on $A$, has a single zero eigenvalue, and has all other eigenvalues bounded away from zero, we see that the one-dimensional kernel $\ker(A - \lambda_i(A) I)$ also depends smoothly on $A$ near $A_0$.  Since $P_i$ is the orthogonal projection to this kernel, the smoothness of $P_i$ now follows.

Since $A-\lambda_i(A) I$ depends smoothly on $A$, and has eigenvalues bounded away from zero on the range of $1-P_i$ (which also smoothly dependent on $A$), we see that $R_i$ (and hence $Q_i$) also depend smoothly on $A$.
\end{proof}

Now we turn to more quantitative estimates on the smoothness of $\lambda_i$, $P_i$, $R_i$, $Q_i$ for fixed $1 \leq i \leq n$.  For our applications to Proposition \ref{swap}, we consider matrices $A = A(z)$ which are parameterized smoothly (though not holomorphically, of course) by some complex parameter $z$ in a domain $\Omega \subset \C$.  We assume that $\lambda_i(A(z))$ is simple for all $z \in \Omega$, which by the above lemma implies that $\lambda_i := \lambda_i(A(z))$, $P_i := P_i(A(z))$, $R_i := R_i(A(z))$, and $Q_i := Q_i(A(z))$ all depend smoothly on $z$ in $\Omega$.  

It will be convenient to introduce some more notation, to deal with the technical fact that $z$ is complex-valued rather than real.
For any smooth function $f(z)$ (which may be scalar, vector, or matrix-valued), we use 
$$\nabla^m f(z) := (\frac{\partial^m f}{\partial \Re(z)^l \partial \Im(z)^{m-l}})_{l=0}^m$$
to denote the $m^{th}$ gradient with respect to the real and imaginary parts of $z$ (thus $\nabla^m f$ is an $(m+1)$-tuple, each of whose components is of the same type as $f$; for instance, if $f$ is matrix valued, so are all the components of $\nabla^m f$).  If $f$ is matrix-valued, we define $\|\nabla^m f\|_F$ to be the $\ell^2$ norm of the Frobenius norms of the various components of
$\nabla^m f$, and similarly for other norms.

We observe the Leibniz rule
\begin{equation}\label{nabfag}
\nabla^k ( f g ) = \sum_{m=0}^k (\nabla^m f) * (\nabla^{k-m} g) = f (\nabla^k g) + (\nabla^k f) g + \sum_{m=1}^{k-1} (\nabla^m f) * (\nabla^{k-m} g)
\end{equation}
where the $(k+1)$-tuple $(\nabla^m f) * (\nabla^{k-m} g)$ is defined as
$$ \Big( \sum_{\max(0,l+m-k) \leq l' \leq \min(l,m)} \binom{l}{l'} \binom{k-l}{m-l'} \frac{\partial^{k-m} f}{\partial \Re(z)^{l'} \partial \Im(z)^{m-l'}}
\frac{\partial^m g}{\partial \Re(z)^{l-l'} \partial \Im(z)^{k-m-l+l'}} \Big)_{l=0}^k.$$
The exact coefficients here are not important, and one can view $(\nabla^m f) * (\nabla^{k-m} g)$ simply as a bilinear combination of $\nabla^m f$ and $\nabla^{k-m} g$.  Note that \eqref{nabfag} is valid for matrix-valued $f,g$ as well as scalar $f,g$. 
For a tuple $(A_1, \dots, A_{l})$ of matrices, we define 

$$\tr (A_1, \dots, A_{l}) := (\tr (A_1) , \dots, \tr (A_l) ). $$

We can now give the higher order Hadamard variation formulae:

\begin{proposition}[Recursive formula for derivatives of $\lambda_i, P_i, R_i$]\label{recur}  For any integer $k \geq 1$, we have
\begin{equation}\label{alp-0}
\nabla^k \lambda_i = \sum_{m=1}^k \tr( (\nabla^m A) * (\nabla^{k-m} P_i) P_i ) - \sum_{m=1}^{k-1} (\nabla^m \lambda_i) * \tr( (\nabla^{k-m} P_i) P_i ).
\end{equation}
and
\begin{equation}\label{ppi-0}
\begin{split}
\nabla^k P_i &= - R_i (\nabla^k A) P_i - P_i (\nabla^k A) R_i \\
&\quad - \sum_{m=1}^{k-1} [ R_i ((\nabla^m A) - (\nabla^m \lambda_i) I) * (\nabla^{k-m} P_i) P_i +P_i (\nabla^{k-m} P_i) * ((\nabla^m A) - (\nabla^m \lambda_i) I) R_i] \\
&\quad + \sum_{m=1}^{k-1} (\nabla^{m} P_i) * (\nabla^{k-m} P_i) (I - 2P_i). 
\end{split}
\end{equation}
\begin{equation}\label{rop-cont}
 (\nabla^{k} R_i) P_i = - \sum_{m=0}^{k-1} (\nabla^{m} R_i) * (\nabla^{k-m} P_i).
\end{equation}
and
\begin{equation}\label{rop0}
 (\nabla^{k} R_i) (I-P_i) = - (\nabla^{k} P_i) R_i - \sum_{m=0}^{k-1} (\nabla^{m} R_i) * ((\nabla^{k-m} A) - (\nabla^{k-m} \lambda_i) I) R_i
\end{equation}
and thus
\begin{equation}\label{rop}
\begin{split}
(\nabla^{k} R_i) &= - (\nabla^{k} P_i) R_i - \sum_{m=0}^{k-1} (\nabla^{m} R_i) * ((\nabla^{k-m} A) - (\nabla^{k-m} \lambda_i) I) R_i \\
&\quad - \sum_{m=0}^{k-1} (\nabla^{m} R_i) * (\nabla^{k-m} P_i).
\end{split}\end{equation}
\end{proposition}

\begin{proof}  Our starting point is the identities
\begin{equation}\label{alp}
\lambda_i P_i = A P_i
\end{equation}
and
\begin{equation}\label{ppi}
P_i P_i = P_i.
\end{equation}
We differentiate these identities $k$ times using the Leibniz rule \eqref{nabfag} to obtain
\begin{equation}\label{alp-k}
(\nabla^{k} \lambda_i) P_i + \sum_{m=0}^{k-1} (\nabla^m \lambda_i) * (\nabla^{k-m} P_i) = \sum_{m=0}^k (\nabla^m A) * (\nabla^{k-m} P_i).
\end{equation}
and
\begin{equation}\label{ppi-k}
(\nabla^{k} P_i) P_i + P_i (\nabla^{k} P_i) + \sum_{m=1}^{k-1} (\nabla^{m} P_i) * (\nabla^{k-m} P_i) = (\nabla^{k} P_i).
\end{equation}
Multiplying \eqref{alp-k} by $P_i$ and taking traces one obtains \eqref{alp-0} (the $m=0$ terms cancel because of \eqref{alp}, which implies that $\tr( A (\nabla^{m} P_i) P_i ) = \lambda_i \tr( (\nabla^{m} P_i) P_i )$).

We next compute $\nabla^{k} P_i$ using the decomposition

\begin{equation} \label{eqn:decomposition} \nabla^{k} P_i= P_i (\nabla^{k} P_i) P_i + (I-P_i) (\nabla^{k} P_i) (I-P_i) +
(I-P_i) (\nabla^{k} P_i) P_i + P_i (\nabla^{k} P_i) (I- P_i). \end{equation}

Multiplying both sides of \eqref{ppi-k}  by $P_i$ (on the right) and using the identity $P_iP_i=P_i$, we get a cancelation which implies

\begin{equation}\label{ppi-1}
P_i (\nabla^{k} P_i) P_i = - \sum_{m=1}^{k-1} (\nabla^{m} P_i) * (\nabla^{k-m} P_i) P_i.
\end{equation}

Repeating the same trick with  $I-P_i$  instead of $P_i$, we have

\begin{equation}\label{ppi-2}
(I-P_i) (\nabla^{k} P_i) (I-P_i) = \sum_{m=1}^{k-1} (\nabla^{m} P_i) * (\nabla^{k-m} P_i) (I-P_i).
\end{equation}

This gives two of the four components of $\nabla^{k} P_i$.  To obtain the other components, multiply \eqref{alp-k} on the left by $I-P_i$ and notice that 
the $(I-P_i) (\nabla^{k} \lambda_i) P_i$ term vanishes because of \eqref{ppi}. Rearranging the terms, we obtain 

$$
(I-P_i) (A - \lambda_i) (\nabla^{k} P_i) = - \sum_{m=1}^{k-1} (I-P_i) ((\nabla^m A) - (\nabla^m \lambda_i) I) * (\nabla^{k-m} P_i) - (I-P_i) (\nabla^k A) P_i
$$
  Applying $R_i$ on the left and $P_i$ on the right and using \eqref{RA}, we get 
$$
(I-P_i) (\nabla^{k} P_i) P_i = - \sum_{m=1}^{k-1} R_i ((\nabla^m A) - (\nabla^m \lambda_i) I) * (\nabla^{k-m} P_i) P_i - R_i (\nabla^k A) P_i.
$$
By taking adjoints, we obtain 
$$
P_i (\nabla^{k} P_i) (I-P_i) = - \sum_{m=1}^{k-1} P_i (\nabla^{k-m} P_i) * ((\nabla^m A) - (\nabla^m \lambda_i) I) R_i - P_i (\nabla^k A) R_i.
$$
These, together with \eqref{ppi-1}, \eqref{ppi-2} and \eqref{eqn:decomposition} imply \eqref{ppi-0}.

Now we turn to $R_i$.  Here, we use the identities \eqref{RA}
$$ R_i (A-\lambda_i I) = I-P_i$$
and
$$ R_i P_i = 0.$$
Differentiating the second identity $k$ times gives \eqref{rop-cont}.  Differentiating the first identity $k$ times, meanwhile, gives
$$ (\nabla^{k} R_i) (A-\lambda_i I) = - (\nabla^{k} P_i) - \sum_{m=0}^{k-1} (\nabla^{m} R_i) * ((\nabla^{k-m} A) - (\nabla^{k-m} \lambda_i) I);$$
multiplying on the right by $R_i$, we obtain \eqref{rop0}, and then \eqref{rop} follows.
\end{proof}

We isolate the $k=1$ case of Proposition \ref{recur}, obtaining the \emph{Hadamard variation formulae}
\begin{equation}\label{alp-first}
\nabla \lambda_i = \tr( \nabla A P_i )
\end{equation}
and
\begin{equation}\label{ppi-first}
\nabla P_i = - R_i (\nabla A) P_i - P_i (\nabla A) R_i
\end{equation}

\subsection{Bounding the derivatives}

 We now use the recursive inequalities obtained in the previous section to bound the derivatives of $\lambda_i$ and $P_i$, assuming some quantitative control on the spectral gap between $\lambda_i$ and other eigenvalues, and on the matrix $A$ and its derivatives.  Let us
   begin with a crude bound.

\begin{lemma}[Crude bound]\label{crude}  Let $A = A(z)$ be an $n \times n$ matrix varying (real)-linearly in $z$ (thus $\nabla^k A = 0$ for $k \geq 2$), with
$$ \| \nabla A \|_{op} \leq V$$
for some $V > 0$.  Let $1 \leq i \leq n$.
At some fixed value of $z$, suppose we have the spectral gap condition
\begin{equation}\label{laj}
|\lambda_j(A(z)) - \lambda_i(A(z))| \geq r
\end{equation}
for all $j \neq i$ and some $r > 0$ (in particular, $\lambda_i(A(z))$ is a simple eigenvalue).
Then for all $k \geq 1$ we have (at this fixed choice of $z$)
\begin{equation}\label{li-crude}
 |\nabla^{k} \lambda_i| \ll_k V^k r^{1-k}
\end{equation}
and
\begin{equation}\label{pi-crude}
 \| \nabla^{k} P_i \|_{op} \ll_k V^k r^{-k}
\end{equation}
and
\begin{equation}\label{ri-crude}
 \| \nabla^{k} R_i \|_{op} \ll_k V^k r^{-k-1}.
\end{equation}
and
\begin{equation}\label{xi-crude}
 |\nabla^{k} Q_i| \ll_k n V^k r^{-k-2}.
\end{equation}

\end{lemma}

\begin{proof}  Observe that the spectral gap condition \eqref{laj} ensures that
\begin{equation}\label{laj-inv}
\|R_i\|_{op} \leq \frac{1}{r}.
\end{equation}
We also observe the easy inequality
\begin{equation}\label{rank1}
|\tr( B P_i )| = |\tr(P_i B)| = |\tr( P_i B P_i)| \leq \|B\|_{op}
\end{equation}
for any Hermitian matrix $B$, which follows as $P_i$ is a rank one orthogonal projection.

To prove \eqref{li-crude}, \eqref{pi-crude}, we induct on $k$.  The case $k=1$ follows from \eqref{alp-first}, \eqref{ppi-first}, \eqref{rank1}, and \eqref{laj-inv}; and then, for $k > 1$, the claim follows from the induction hypotheses and \eqref{alp-0}, \eqref{ppi-0}, \eqref{laj-inv}, \eqref{rank1}.

To prove \eqref{ri-crude}, we also induct on $k$.  The case $k=0$ follows from \eqref{laj-inv}.  For $k \geq 1$, the claim then follows from the induction hypotheses and \eqref{rop}.

To prove \eqref{xi-crude}, we use the product rule to bound
\begin{align*}
 |\nabla^{k} Q_i| &\ll_k \sum_{m=0}^k |\tr( (\nabla^m R_i) * (\nabla^{k-m} R_i) )|\\
&\ll_k n \sum_{m=0}^k \| \nabla^m R_i \|_{op} \| \nabla^{k-m} R_i \|_{op}
\end{align*}
and the claim follows from \eqref{ri-crude}.
\end{proof}

This crude bound is insufficient for our applications, and we will need to supplement it with one that strengthens the spectral condition, and also assumes an ``delocalization'' property for the projections $P_\alpha, P_\beta$ relative to the perturbation $\dot A$.

\begin{lemma}[Better bound]\label{better} Let $A = A(z)$ be an $n \times n$ matrix varying real-linearly in $z$.  Let $1 \leq i \leq n$.  At some fixed value of $z$, suppose that $\lambda_i = \lambda_i(A(z))$ is a simple eigenvalue, and that we have a partition
$$ I = P_i + \sum_{\alpha \in J} P_\alpha$$
where $J$ is a finite index set (not containing $i$), and $P_\alpha$ are orthogonal projections to invariant spaces on $A$ (i.e. to spans of eigenvectors not corresponding to $\lambda_i$).  Suppose that on the range of each $P_\alpha$, the eigenvalues of $A-\lambda_i$ have magnitude at least $r_\alpha$ for some $r_\alpha > 0$; equivalently, we have
\begin{equation}\label{riF}
\| R_i P_\alpha \|_{op} \leq \frac{1}{r_\alpha}.
\end{equation}
Suppose also that we have the delocalization bounds
\begin{equation}\label{vu}
\| P_\alpha \dot A P_\beta\|_{F} \leq v c_\alpha c_\beta
\end{equation}
for all $\alpha,\beta \in J$ and some $v>0$ and $c_\alpha \geq 1$ with $c_i=1$, and the strong spectral gap condition
\begin{equation}\label{jil}
 \sum_{\alpha \in J} \frac{c_\alpha^2}{r_\alpha} \leq L
\end{equation}
for some $L>0$.  Then at this fixed choice of $z$, and for all $\alpha,\beta \in J$, we have
\begin{align}
 |\nabla^{k} \lambda_i| &\ll_k L^{k-1} v^k \label{luck}\\
 \| P_i (\nabla^{k} P_i) P_i \|_{F} &\ll_k L^k v^k\label{uki}\\
 \|P_\alpha (\nabla^{k} P_i) P_i\|_{F} = \|P_i (\nabla^{k} P_i) P_\alpha\|_{F} &\ll_k \frac{c_\alpha}{r_\alpha} L^{k-1} v^k\label{ukl}\\
 \|P_\alpha (\nabla^{k} P_i) P_\beta\|_{F} &\ll_k \frac{c_\alpha c_\beta}{r_\alpha r_\beta} L^{k-2} v^k\label{ukl2}
\end{align}
 for all $k \geq 1$, and
\begin{align}
\| P_i (\nabla^{k} R_i) P_i \|_{F} &\ll_k L^{k+1} v^k \label{rop4}\\
\|P_\alpha (\nabla^{k} R_i) P_i\|_{F} = \|P_i (\nabla^{k} R_i) P_\alpha\|_{F} &\ll_k \frac{c_\alpha}{r_\alpha} L^{k} v^k\label{rop5}\\
\| P_\alpha (\nabla^{k} R_i) P_\beta \|_{F} &\ll_k \frac{c_\alpha c_\beta}{r_\alpha r_\beta} L^{k-1} v^k \label{rop2}
\end{align}
for all $k \geq 0$.
\end{lemma}

We remark that we can unify the bounds \eqref{uki}-\eqref{ukl2} and \eqref{rop4}-\eqref{rop2} by allowing $\alpha,\beta$ to vary in $J \cup \{i\}$ rather than $J$, and adopting the convention that $r_i := 1/L$.

\begin{proof}  Note that the projections $P_i$ and the $P_\alpha$ are idempotent and all annihilate each other, and commute with $A$ and $R_i$.  We will use these facts throughout this proof without further comment.

To prove \eqref{luck}-\eqref{ukl2}, we again induct on $k$.  When $k=1$, the claim \eqref{luck} follows from \eqref{alp-first}, \eqref{vu}, and \eqref{rank1}, while \eqref{uki}-\eqref{ukl2} follow from \eqref{ppi-first} and \eqref{vu}.  (For \eqref{ukl2} we in fact obtain that the left-hand side is zero.)

Now suppose inductively that $k > 1$, and that the claims have already been proven for all smaller values of $k$.

We first prove \eqref{luck}.  From \eqref{alp-0} and the linear nature of $A$, we have
$$
|\nabla^{k} \lambda_i| \ll_k |\tr( (\nabla A) * (\nabla^{k-1} P_i) P_i )| + \sum_{m=1}^{k-1} |\nabla^m \lambda_i| |\tr( (\nabla^{k-m} P_i) P_i )|.
$$
From \eqref{rank1} and the inductive hypothesis \eqref{uki} we have
$$ |\tr( (\nabla^{k-m} P_i) P_i )| \ll_k L^{k-m} v^{k-m}$$
for any $1 \leq m \leq k-1$, and thus by the inductive hypothesis \eqref{luck} we see that
$$ \sum_{m=1}^{k-1} |\nabla^m \lambda_i| |\tr( (\nabla^{k-m} P_i) P_i )| \ll_k L^{k-1} v^k.$$
Next, by splitting $(\nabla A) * (\nabla^{k-1} P_i) P_i$ as $\sum_{\alpha \in J \cup \{i\}} (\nabla A) * P_\alpha (\nabla^{k-1} P_i) P_i$ and using \eqref{rank1}, we have
$$ |\tr( (\nabla A) * (\nabla^{k-1} P_i) P_i )| \ll_k \sum_{\alpha \in J \cup \{i\}} \| P_i (\nabla A) P_\alpha \|_{F} \| P_\alpha (\nabla^{k-1} P_i) P_i \|_{F}.$$
Using \eqref{vu} and the inductive hypotheses \eqref{uki}, \eqref{ukl}, we thus have
$$ |\tr( (\nabla A) * (\nabla^{k-1} P_i) P_i )| \ll_k v L^{k-1} v^{k-1} + \sum_{\alpha \in J} v c_\alpha \frac{c_\alpha}{r_\alpha} L^{k-2} v^{k-1}.$$
Applying \eqref{jil} we conclude \eqref{luck} as desired.

Now we prove \eqref{uki}.  From \eqref{ppi-0} (or \eqref{ppi-1}) we have
$$
\| P_i (\nabla^{k} P_i) P_i \|_{F} \ll_k \sum_{m=1}^{k-1} \| P_i (\nabla^{m} P_i) * (\nabla^{k-m} P_i) P_i \|_{F}.$$
We can split
$$ \| P_i (\nabla^{m} P_i) * (\nabla^{k-m} P_i) P_i \|_{F} \ll_k \sum_{\alpha \in J \cup \{i\}} \|P_i (\nabla^{m} P_i) P_{\alpha}\|_{F} \|P_{\alpha} (\nabla^{k-m} P_i) P_i\|_{F}.$$
Applying the inductive hypotheses \eqref{uki}, \eqref{ukl}, we conclude
$$
\| P_i (\nabla^{k} P_i) P_i \|_{F} \ll_k \sum_{m=1}^{k-1} L^m v^m L^{k-m} v^{k-m} + \sum_{\alpha \in J} \frac{c_\alpha}{r_\alpha} L^{m-1} v^m \frac{c_\alpha}{r_\alpha} L^{k-m-1} v^{k-m}.$$
bounding one of the $\frac{1}{r_\alpha}$ factors crudely by $L$ and then summing using \eqref{jil} we obtain \eqref{uki} as desired.

Now we prove \eqref{ukl2}.  From \eqref{ppi-0} (or \eqref{ppi-2}) we have
\begin{align*}
\| P_\alpha (\nabla^{k} P_i) P_\beta \|_{F} &\ll_k \sum_{m=1}^{k-1} \| P_\alpha (\nabla^{m} P_i) * (\nabla^{k-m} P_i) P_\beta \|_{F} \\
&\ll_k \sum_{m=1}^{k-1} \sum_{\gamma \in J \cup \{i\}} \| P_\alpha (\nabla^{m} P_i) P_\gamma \|_{F} \| P_\gamma (\nabla^{k-m} P_i) P_\beta \|_{F}
\end{align*}
Using the inductive hypotheses \eqref{ukl}, \eqref{ukl2}, we obtain
$$
\| P_\alpha (\nabla^{k} P_i) P_\beta \|_{F} \ll_k
\sum_{m=1}^{k-1} \frac{1}{r_\alpha} L^{m-1} v^m \frac{1}{r_\beta} L^{k-m-1} v^{k-m} + \sum_{\gamma \in J} \frac{c_\alpha c_\gamma}{r_\alpha r_\gamma} L^{m-2} v^m \frac{c_\gamma c_\beta}{r_\gamma r_\beta} L^{k-m-2} v^{k-m}.
$$
Bounding one of the $\frac{1}{r_\gamma}$ factors crudely by $L$ and applying \eqref{jil} we obtain \eqref{ukl2} as desired.

Finally, we prove \eqref{ukl}.  Since $P_\alpha (\nabla^{k} P_i) P_i$ has the same Frobenius norm as its adjoint $P_i (\nabla^{k} P_i) P_\alpha$, it suffices to bound
$\|P_\alpha (\nabla^{k} P_i) P_i\|_{F}$.  From \eqref{ppi-0} we have
$$
\|P_\alpha (\nabla^{k} P_i) P_i\|_{F} \ll_k \| R_i P_\alpha (\nabla A) * (\nabla^{k-1} P_i) P_i \|_{F} + \sum_{m=1}^{k-1} |\nabla^m \lambda_i|
\| R_i P_\alpha (\nabla^{k-m} P_i) P_i \|_{F} $$
and hence by \eqref{riF}
$$
\|P_\alpha (\nabla^{k} P_i) P_i\|_{F} \ll_k
\frac{1}{r_\alpha} \| P_\alpha (\nabla A) (\nabla^{k-1} P_i) P_i \|_{F} + \frac{1}{r_\alpha} \sum_{m=1}^{k-1} |\nabla^m \lambda_i|
\| P_\alpha (\nabla^{k-m} P_i) P_i \|_{F}.$$
From the inductive hypotheses \eqref{luck}, \eqref{ukl} and crudely bounding $\frac{1}{r_\alpha}$ by $L$, we have
$$ \frac{1}{r_\alpha} \sum_{m=1}^{k-1} |\nabla^m \lambda_i|
\| P_\alpha (\nabla^{k-m} P_i) P_i \|_{F} \ll_k \frac{c_\alpha}{r_\alpha} L^{k-1} v^k.$$
Meanwhile, by splitting
$$ \| P_\alpha (\nabla A) * (\nabla^{k-1} P_i) P_i \|_{F} \leq \sum_{\beta \in J \cup \{i\}} \| P_\alpha (\nabla A) P_\beta \|_{F} \| P_\beta (\nabla^{k-1} P_i) P_i \|_{F}$$
and using \eqref{vu} and the inductive hypothesis \eqref{ukl} we have
$$
\frac{1}{r_\alpha} \| P_\alpha (\nabla A) * (\nabla^{k-1} P_i) P_i \|_{F}  \ll_k \frac{1}{r_\alpha} (v c_\alpha L^{k-1} v^{k-1} + \sum_{\beta \in J} v c_\alpha c_\beta \frac{c_\beta}{r_\beta} L^{k-2} v^{k-1}).
$$
Applying \eqref{jil} we obtain \eqref{ukl} as required.

Having proven \eqref{luck}-\eqref{ukl2} for all $k \geq 1$, we now prove \eqref{rop4}-\eqref{rop2} by induction.  The claim is easily verified for $k=0$ (note that the left-hand side of \eqref{rop4} and \eqref{rop5} in fact vanishes, as does the left-hand side of \eqref{rop2} unless $\alpha=\beta$), so suppose $k \geq 1$ and that \eqref{rop4}-\eqref{rop2} has been proven for smaller values of $k$.

We first prove \eqref{rop4}.  From \eqref{rop-cont} we have
\begin{align*}
 \| P_i (\nabla^{k} R_i) P_i \|_{F} &\ll_k \sum_{m=0}^{k-1} \| P_i (\nabla^{m} R_i) * (\nabla^{k-m} P_i) P_i \|_{F} \\
&\ll_k \sum_{m=0}^{k-1} \sum_{\alpha \in J \cup \{i\}} \| P_i (\nabla^{m} R_i) P_\alpha \|_{F} \| P_\alpha (\nabla^{k-m} P_i) P_i \|_{F}.
\end{align*}
Applying \eqref{uki}, \eqref{ukl} and the induction hypotheses \eqref{rop4}, \eqref{rop5} we conclude
$$
 \| P_i (\nabla^{k} R_i) P_i \|_{F}  \ll_k \sum_{m=0}^{k-1} L^{m+1} v^m L^{k-m} v^{k-m} + \sum_{\alpha \in J}
 \frac{c_\alpha}{r_\alpha} L^{m} v^m \frac{c_\alpha}{r_\alpha} L^{k-m-1} v^{k-m};$$
crudely bounding one of the $\frac{1}{r_\alpha}$ factors by $L$ and using \eqref{jil} we obtain the claim.

Similarly, to prove \eqref{rop5}, we apply \eqref{rop-cont} as before to obtain
\begin{align*}
 \| P_\alpha (\nabla^{k} R_i) P_i \|_{F} &\ll_k \sum_{m=0}^{k-1} \| P_\alpha (\nabla^{m} R_i) * (\nabla^{k-m} P_i) P_i \|_{F} \\
&\ll_k \sum_{m=0}^{k-1} \sum_{\beta \in J \cup \{i\}} \| P_\alpha (\nabla^{m} R_i) P_\beta \|_{F} \| P_\beta (\nabla^{k-m} P_i) P_i \|_{F};
\end{align*}
applying \eqref{uki}, \eqref{ukl} and the induction hypotheses \eqref{rop5}, \eqref{rop2} we conclude
$$
 \| P_\alpha (\nabla^{k} R_i) P_i \|_{F}  \ll_k \sum_{m=0}^{k-1} \frac{c_\alpha}{r_\alpha} L^{m} v^m L^{k-m} v^{k-m} + \sum_{\beta \in J}
 \frac{c_\alpha c_\beta}{r_\alpha r_\beta} L^{m-1} v^m \frac{c_\beta}{r_\beta} L^{k-m-1} v^{k-m}.$$
Again, bounding one of the $\frac{1}{r_\beta}$ factors by $L$ and using \eqref{jil} we obtain the claim.

Finally, we prove \eqref{rop2}.  From \eqref{rop0} we have
$$
\| P_\alpha (\nabla^{k} R_i) P_\beta \|_{F} \ll_k \| P_\alpha (\nabla^{k} P_i) R_i P_\beta \|_{F} + \| P_\alpha (\nabla^{k-1} R_i) * (\nabla A) R_i P_\beta \|_{F} + \sum_{m=0}^{k-1}  |\nabla^{k-m} \lambda_i| \| P_\alpha (\nabla^{m} R_i) R_i P_\beta \|_{F}.$$
As $R_i P_\beta = P_\beta R_i P_\beta$ has a norm of at most $\frac{1}{r_\beta}$, we conclude
$$
\| P_\alpha (\nabla^{k} R_i) P_\beta \|_{F} \ll_k \frac{1}{r_\beta} \| P_\alpha (\nabla^{k} P_i) P_\beta \|_{F} + \frac{1}{r_\beta} \| P_\alpha (\nabla^{k-1} R_i) * (\nabla A) P_\beta \|_{F} + \frac{1}{r_\beta} \sum_{m=0}^{k-1} |\nabla^{k-m} \lambda_i| \| P_\alpha (\nabla^{m} R_i) P_\beta \|_{F}.$$
From \eqref{luck}, the crude bound $\frac{1}{r_\beta} \leq L$, and the induction hypothesis \eqref{rop2} we have
$$ \frac{1}{r_\beta} \sum_{m=0}^{k-1} |\nabla^{k-m} \lambda_i| \| P_\alpha (\nabla^{m} R_i) P_\beta \|_{F} \ll_k \frac{c_\alpha c_\beta}{r_\alpha r_\beta} L^{k-1} v^k.$$
From \eqref{ukl2} and $\frac{1}{r_\beta} \leq L$ we similarly have
$$ \frac{1}{r_\beta} \| P_\alpha (\nabla^{k} P_i) P_\beta \|_{F} \ll_k \frac{c_\alpha c_\beta}{r_\alpha r_\beta} L^{k-1} v^k.$$
Finally, splitting
$$ \| P_\alpha (\nabla^{k-1} R_i) * (\nabla A) P_\beta \|_{F} \leq \sum_{\gamma \in J \cup \{i\}} \| P_\alpha (\nabla^{k-1} R_i) P_\gamma \|_{F} \| P_\gamma (\nabla A) P_\beta \|_{F}$$
and using the induction hypothesis \eqref{rop5}, \eqref{rop2} and \eqref{vu}, we obtain
$$ \| P_\alpha (\nabla^{k-1} R_i) * (\nabla A) P_\beta \|_{F} \ll_k \frac{c_\alpha}{r_\alpha} L^{k-1} v^{k-1} v c_\beta + \sum_{\gamma \in J}
\frac{c_\alpha c_\gamma}{r_\alpha r_\gamma} L^{k-2} v^{k-1} v c_\gamma c_\beta$$
and thus by \eqref{jil}
$$ \| P_\alpha (\nabla^{k-1} R_i) * (\nabla A) P_\beta \|_{F} \ll_k \frac{c_\alpha c_\beta}{r_\alpha} L^{k-1} v^{k}.$$
Putting all this together we obtain \eqref{rop2} as claimed.
\end{proof}

We extract a special case of the above lemma, in which the perturbation only affects a single entry (and its transpose):

\begin{corollary}[Better bound, special case]\label{better-cor} Let $A = A(z)$ be an $n \times n$ matrix depending on a complex parameter $z$ of the form
$$ A(z) = A(0) + z e_p^* e_q + \overline{z} e_q^* e_p$$
for some vectors $e_p,e_q$.  Let $1 \leq i \leq n$.  At some fixed value of $z$, suppose that $\lambda_i = \lambda_i(A(z))$ is a simple eigenvalue, and that we have a partition
$$ I = P_i + \sum_{\alpha \in J} P_\alpha$$
where $J$ is a finite index set, and $P_\alpha$ are orthogonal projections to invariant spaces on $A$ (i.e. to spans of eigenvectors not corresponding to $\lambda_i$).  Suppose that on the range of each $P_\alpha$, the eigenvalues of $A-\lambda_i$ have magnitude at least $r_\alpha$ for some $r_\alpha > 0$.  Suppose also that we have the incompressibility bounds
$$ \| P_\alpha e_p \|, \| P_\alpha e_q \| \leq w d_\alpha^{1/2} $$
for all $\alpha \in J \cup \{i\}$ and some $w > 0$ and $d_\alpha \geq 1$.  Then at this value of $z$, and for all $k \geq 1$, we have
\begin{equation}\label{luck2}
 |\nabla^{k} \lambda_i| \ll_k (\sum_{\alpha \in J} \frac{d_\alpha}{r_\alpha})^{k-1} w^{2k}.
\end{equation}
and
\begin{equation}\label{furioso}
|\nabla^k Q_i| \ll_k (\sum_{\alpha \in J} \frac{d_\alpha}{r_\alpha})^{k+2} w^{2k}
\end{equation}
for all $k \geq 0$ at this value of $z$.
\end{corollary}

\begin{proof} A short computation shows that the hypotheses of Lemma \ref{better} are obeyed with $v$ replaced by $O(w^2)$, $c_\alpha$ set equal to $d_\alpha^{1/2}$, and $L$ equal to $O(\sum_{\alpha \in J} \frac{d_\alpha}{r_\alpha})$.  From \eqref{luck} we then conclude \eqref{luck2}.  As for \eqref{furioso}, we see from the product rule that
$$ |\nabla^k Q_i| \ll_k \sum_{m=0}^k |\tr( (\nabla^m R_i) (\nabla^{k-m} R_i^*) )|$$
which we can split further using Cauchy-Schwarz as
$$ |\nabla^k Q_i| \ll_k \sum_{m=0}^k \sum_{\alpha,\beta \in J \cup \{i\}} \| P_\alpha (\nabla^m R_i) P_\beta \|_F \| P_\alpha (\nabla^{k-m} R_i) P_\beta \|_F.$$
Applying \eqref{rop4}, \eqref{rop5}, \eqref{rop2}, we conclude
$$
|\nabla^k Q_i| \ll_k L^{k+2} v^{k} + \sum_{\alpha \in J} \frac{c_\alpha^2}{r_\alpha^2} L^{k} v^{k} + \sum_{\alpha,\beta \in J} \frac{c_\alpha^2 c_\beta^2}{r_\alpha^2 r_\beta^2} L^{k-2} v^{k}.$$
Bounding one of the factors of $\frac{1}{r_\alpha}$ in the first sum by $L$, and $\frac{1}{r_\alpha r_\beta}$ in the second sum by $L^2$, and using \eqref{jil} and the choices for $v$ and $L$, we obtain the claim.
\end{proof}

\subsection{Conclusion of the argument}

We can now prove Proposition \ref{swap}.  
Fix $k \geq 1$, $r=2,3,4$ and $\eps_1 > 0$, and suppose that $C_1$ is sufficiently large.  We assume $A(0), e_p, e_q, i_1,\ldots,i_k, G, F, \zeta, \zeta'$ are as in the proposition.

We may of course assume that $F(z_0) \neq 0$ for at least one $z_0$ with $|z_0| \leq n^{1/2 +\eps_1}$, since the claim is vacuous otherwise.

Suppose we can show that
\begin{equation}\label{naba}
\nabla^m F(z) = O( n^{-m+O(\eps_1)} )
\end{equation}
for all $|z| \leq n^{1/2 +\eps_1}$ and $0 \leq m \leq 5$. 
 Then by Taylor expansion, one has
$$ F(\zeta) = P(\zeta,\overline{\zeta}) + O( n^{-(r+1)/2 + O(\eps_1)} )$$
where $P$ is a polynomial of degree at most $r$ whose coefficients are of size at most $n^{O(\eps_1)}$.  Taking expectations for both $F(\zeta)$ and $F(\zeta')$, we obtain the claim \eqref{barg} when $\zeta, \zeta'$.  A similar argument gives the improved version of \eqref{barg} at the end of Proposition \ref{swap} if one can improve the right-hand side of \eqref{naba} to $O( n^{-m-C' m \eps_1} )$ for some sufficiently large absolute constant $C'_1$.

It remains to show \eqref{naba}.  By up to five applications of the chain rule, the above claims follow from

\begin{lemma}  Suppose that $F(z_0) \neq 0$ for at least one $z_0$ with $|z_0| \leq n^{1/2 +\eps_1}$.  Then for \emph{all} $z$ with $|z_0| \leq n^{1/2 +\eps_1}$, and all $1 \leq j \leq k$, we have
$$
 |\nabla^{k} \lambda_{i_j}(z)| \ll_k n^{5\eps_1k} n^{-k}
$$
and
$$
 |\nabla^{k} Q_{i_j}(z)| \ll_k n^{5 \eps_1 (k+2)} n^{-k}
$$
for all $z$ with $|z| \leq n^{1/2 +\eps_1}$ and all $0 \leq k \leq 10$.
\end{lemma}

\begin{proof} Fix $j$.  Since $F(z_0) \neq 0$, we have
\begin{equation}\label{xii}
Q_{i_j}(A(z_0)) \leq n^{\eps_1}.
\end{equation}

For a technical reason having to do with a subsequent iteration argument, we will replace \eqref{xii} with the slightly weaker bound
\begin{equation}\label{xii-2}
Q_{i_j}(A(z_0)) \leq 2 n^{\eps_1}.
\end{equation}

By the definition of $Q_i$, we have, as a consequence, that
\begin{equation}\label{liz}
 |\lambda_{i'}(A(z_0)) - \lambda_{i_j}(A(z_0))| \gg n^{-\eps_1/2}
\end{equation}
for all $i' \neq i_j$.

By the Weyl inequalities \eqref{weyl}, we thus have
$$ |\lambda_{i'}(A(z)) - \lambda_{i_j}(A(z))| \gg n^{-\eps_1/2}$$
whenever $|z - z_0| \ll n^{-1 -2 \eps_1}$.  From Lemma \ref{crude}, we conclude that
\begin{equation}\label{lamnab}
 |\nabla^m \lambda_{i_j}(A(z))| \ll_m n^{\eps_1 (m+1)/2}
\end{equation}
and
\begin{equation}\label{nabmax}
 |\nabla^m Q_{i_j}(A(z))| \ll_m n^{\eps_1(m+2)/2} n
\end{equation}
for all $m \geq 1$, whenever $|z - z_0| \ll n^{-1  - 2 \eps_1}$.
In particular, from \eqref{xii-2} and the fundamental theorem of calculus we have
\begin{equation}\label{xii-3}
Q_{i_j}(A(z)) \ll n^{\eps_1}
\end{equation}
for all such $z$.

Note that by setting $C_1$ sufficiently large,
we can find $z$ such that
$|z-z_0| \ll n^{-1- 2\eps_1}$, $|z| \leq n^{1/2 +\eps_1}$, and that the real and imaginary parts of $z$ are integer multiples of $n^{-C_1}$.  Then \eqref{pz1}, \eqref{pz2} holds for this value of $z$.  Applying Corollary \ref{better-cor} we conclude that
$$
 |\nabla^{k} \lambda_{i_j}(z)| \ll_k n^{2\eps_1 k} n^{-k} (\sum_{0 \leq \alpha \leq \log n} \frac{2^\alpha}{r_\alpha})^{k-1}
$$
and
$$
|\nabla^k Q_{i_j}(z)| \ll_k n^{2 \eps_1 k} n^{-k} (\sum_{0 \leq \alpha \leq \log n} \frac{2^\alpha}{r_\alpha})^{k+2}
$$
for all $k \geq 1$, where $r_\alpha$ is the minimal value of $|\lambda_i - \lambda_{i_j}|$ for $|i-i_j| \geq 2^\alpha$.  Note that
$$
\sum_{0 \leq \alpha \leq \log n} \frac{2^\alpha}{r_\alpha} \ll \sum_{i \neq i_j} \frac{1}{|\lambda_i - \lambda_{i_j}|}.$$
Meanwhile, from \eqref{xii-3} we have
$$ \sum_{i \neq i_j} \frac{1}{|\lambda_i - \lambda_{i_j}|^2} \ll n^{\eps_1}.$$
From Cauchy-Schwarz this implies that
$$ \sum_{i \neq i_j: |i-i_j| \leq n^{\eps_1}} \frac{1}{|\lambda_i - \lambda_{i_j}|} \ll n^{\eps_1}$$
while from \eqref{noon} we have (with room to spare)
$$ \sum_{i \neq i_j: |i-i_j| \geq n^{\eps_1}} \frac{1}{|\lambda_i - \lambda_{i_j}|} \ll n^{3\eps_1}$$
and thus
$$
\sum_{0 \leq \alpha \leq \log n} \frac{2^\alpha}{r_\alpha} \ll n^{3\eps_1},$$
and so
$$
 |\nabla^{k} \lambda_{i_j}(z)| \ll_k n^{5 \eps_1 k} n^{-k}
$$
and
$$
 |\nabla^{k} Q_{i_j}(z)| \ll_k n^{5 \eps_1(k+2)} n^{-k}
$$
for all $k \geq 1$.  Combining this with \eqref{lamnab}, \eqref{nabmax} we conclude that
$$
 |\nabla^{k} \lambda_{i_j}(z)| \ll_k n^{5 \eps_1 k} n^{-k}
$$
and
\begin{equation}\label{simpol}
 |\nabla^{k} Q_{i_j}(z)| \ll_k n^{5 \eps_1(k+2)} n^{-k}
\end{equation}
for all $z$ with $|z-z_0| \ll n^{-1 +\eps_1}$, and $0 \leq k \leq 10$.

This establishes the lemma in a ball $B(z_0,n^{-1 -2\eps_1})$ of radius $n^{-1 -2 \eps_1}$ centered at $z_0$.  To extend the result to the remainder of the region $\{ z: |z| \leq n^{1/2 +\eps_1}\}$, we observe from \eqref{simpol} that $Q_{i_j}$ varies by at most $O(n^{-1.9})$ (say) on this ball
(instead of $1.9$ we can write any constant less than $2$, given that $\eps_1$ is sufficiently small).  Because of the gap between \eqref{xii} and \eqref{xii-2}, we now see that \eqref{xii-2} continues to hold for all other points $z_1$ in $B(z_0,n^{-1 -2 \eps_1})$ with $|z_1| \leq n^{1/2 +\eps_1}$.  Repeating the above arguments with $z_0$ replaced by $z_1$, and continuing this process, we can eventually cover the entire ball $\{ z: |z| \leq n^{1/2 +\eps_1} \}$ by these estimates.
The key point here is that
at every point of the process  \eqref{xii-2} holds, since
the length of the process is only $n^{3/2 + 3 \eps_1}$ while in each step the value of
$Q_{i_j}$ changes by at most  $O(n^{-1.99})$.
\end{proof}

The proof of Proposition \ref{swap} is now complete.

\section{Good configurations occur frequently}\label{gcc-sec}

The purpose of this section is to prove Proposition \ref{lemma:GCC}.  The arguments here are largely based on those in \cite{ESY1, ESY2,ESY3}.

\subsection{Reduction to a concentration bound for the empirical spectral distribution}

We will first reduce matters to the following concentration estimate for the empirical spectral distribution:

\begin{theorem}[Concentration for ESD]\label{sdb}  For any $\eps, \delta > 0$ and any random Hermitian matrix $M_n = (\zeta_{ij})_{1 \leq i,j \leq n}$ whose upper-triangular entries are independent with mean zero and variance $1$, and such that $|\zeta_{ij}| \leq K$ almost surely for all $i,j$ and some $1 \leq K \leq n^{1/2-\eps}$, and any interval $I$ in $[-2+\eps, 2-\eps]$ of width $|I| \geq \frac{K^2 \log^{20} n}{n}$, the number of eigenvalues $N_I$ of $W_n := \frac{1}{\sqrt{n}} M_n$ in $I$ obeys the concentration estimate
$$ |N_I - n \int_I \rho_{sc}(x)\ dx| \leq \delta n |I|$$
with overwhelming probability.  In particular, $N_I = \Theta_\eps(n |I|)$ with overwhelming probability.
\end{theorem}

\begin{remark} Similar results were established in \cite{ESY1,ESY2,ESY3} assuming stronger regularity hypotheses on the $\zeta_{ij}$.  The proof of this result follows their approach, but also uses  Lemma \ref{lemma:projection} 
and few other ideas which make the current more general setting possible.  In our applications we will take $K = \log^{O(1)} n$, though Theorem \ref{sdb} also has non-trivial content for larger values of $K$.  The loss of $K^2 \log^{20} n$ can certainly be improved, though for our applications any bound which is polylogarithmic for $K = \log^{O(1)} n$ will suffice.
\end{remark}

Let us assume Theorem \ref{sdb} for the moment.  We can then conclude a useful bound on eigenvectors (which will also be applied to prove Theorem \ref{ltail}):

\begin{proposition}[Delocalization of eigenvectors]\label{deloc} Let $\eps, M_n, W_n,\zeta_{ij}, K$  be as in Theorem \ref{sdb}.  Then for any $1 \leq i \leq n$ with $\lambda_i(W_n) \in [-2+\eps,2-\eps]$, if $u_i(W_n)$ denotes a unit eigenvector corresponding to $\lambda_i(W_n)$, then with overwhelming probability each coordinate of $u_i(M_n)$ is $O_{\eps}( \frac{K^2 \log^{20} n}{n^{1/2}} )$.
\end{proposition}

\begin{proof}  By symmetry and the union bound, it suffices to establish this for the first coordinate of $u_i(W_n)$.  By Lemma \ref{lemma:firstcoordinate}, it suffices to establish a lower bound
$$
\sum_{j=1}^{n-1}  (\lambda_j(W_{n-1})-\lambda_i(W_n))^{-2} |u_j(W_{n-1})^* \frac{1}{\sqrt{n}} X|^2 \gg_\eps \frac{n}{K^2 \log^{20} n}$$
with overwhelming probability, where $W_{n-1}$ is the bottom right $n-1 \times n-1$ minor of $W_n$ and $X \in \C^{n-1}$ has entries $\zeta_{i1}$ for $i=2,\ldots,n$.  But by Theorem \ref{sdb}, we can (with overwhelming probability) find a set $J \subset \{1,\ldots,n-1\}$ with $|J| \gg_\eps K^2 \log^{20} n$ such that
$|\lambda_j(W_{n-1})-\lambda_i(W_n)| \ll_\eps \frac{K^2 \log^{20} n}{n}$ for all $n \in J$.  Thus it will suffice to show that
$$ \sum_{j \in J} |u_j(W_{n-1})^* X|^2 \gg_\eps |J|$$
with overwhelming probability.
The left-hand side can be written as $\| \pi_H X \|^2$, where $H$ is the span of all the eigenvectors associated to $J$.  The claim now follows from
Lemma \ref{lemma:projection}.
\end{proof}

We also have the following minor variant:

\begin{corollary}\label{sdb2}  The conclusions of Theorem \ref{sdb} and Proposition \ref{deloc} continues to hold if one replaces a single diagonal entry $\zeta_{pp}$ of $M_n$ by a deterministic real number $x=O(K)$, or if one replaces a single off-diagonal entry $\zeta_{pq}$ of $M_n$ by a deterministic complex number $z=O(K)$ (and also replaces $\zeta_{qp}$ with $\overline{z}$).
\end{corollary}

\begin{proof} After the indicated replacement, the new matrix $M'_n$ differs from the original matrix by a Hermitian matrix of rank at most $2$.  The modification of Theorem \ref{sdb} then follows from Theorem \ref{sdb} and Lemma \ref{lemma:addingrankone}.  The modification of Proposition \ref{deloc} then follows by repeating the proof.  (One of the coefficients of $X$ might now be deterministic rather than random, but it is easy to see that this does not significantly impact Lemma \ref{lemma:projection}.)
\end{proof}

Now we can prove Proposition \ref{lemma:GCC}.  Let $\eps,\eps_1,C,C_1,k,i_1,\ldots,i_k,p,q,A(0)$ be as in that proposition.  By the union bound we may fix $1 \leq j \leq k$, and also fix the $|z| \leq n^{1/2+\eps_1}$ whose real and imaginary parts are multiples of $n^{-C_1}$.  By the union bound again and Corollary \ref{sdb2} (with $K = \log^C n$), the eigenvalue separation condition \eqref{noon} holds with overwhelming probability for every $1 \leq i \leq n$ with $|i-j| \geq n^{\eps_1}$, as does \eqref{pz1} (note that $\|P_{i_j}(A(z)) e_p\|$ is the magnitude of the $p^{th}$ coordinate of a unit eigenvector $u_{i_j}(A(z))$ of $A(z)$).  A similar argument using Pythagoras' theorem gives \eqref{pz2} with overwhelming probability, unless the eigenvalues $\lambda_i(A(z))$ contributing to \eqref{pz2} are not contained in the bulk region $[(-2+\eps')n, (2-\eps')n]$ for some $\eps' > 0$ independent of $n$.  However, it is known (see \cite{GZ}; one can also deduce this fact from Theorem \ref{sdb}) that $\lambda_i(A(z))$ will fall in this bulk region with overwhelming probability whenever $\frac{\eps}{2} n \leq i \leq (1-\frac{\eps}{2}) n$ (say), if $\eps'$ is small enough depending on $\eps$.  Thus, with overwhelming probability, a contribution outside the bulk region can only occur if $2^\alpha \gg_\eps n$, in which case the claim follows by estimating $\| P_{i_j,\alpha}(A(z)) e_p \|$ crudely by $\|e_p\|=1$, and similarly for $\| P_{i_j,\alpha}(A(z)) e_q \|$.  This concludes the proof of Proposition \ref{lemma:GCC} assuming Theorem \ref{sdb}.

\subsection{Spectral concentration}

It remains to prove Theorem \ref{sdb}.  

Following \cite{ESY1, ESY2, ESY3}, we consider  the \emph{Stieltjes transform}
$$ s_n(z) := \frac{1}{n} \sum_{i=1}^n \frac{1}{\lambda_i(W_n)-z}$$
of $W_n$, together with its semicircular counterpart
$$ s(z) := \int_{-2}^2 \frac{1}{x-z} \rho_{sc}(x)\ dx$$
(which was computed explicitly in \eqref{sqrt}).  We will primarily be interested in the imaginary part
\begin{equation}\label{imsn}
\Im s_n(x+\sqrt{-1}\eta) = \frac{1}{n} \sum_{i=1}^n \frac{\eta}{\eta^2 + (\lambda_i(W_n)-x)^2} > 0
\end{equation}
of the Stieltjes transform in the upper half-plane $\eta > 0$.

It is well known that the convergence of the empirical spectral distribution of $W_n$ to $\rho_{sc}(x)$ is closely tied to the convergence of $s_n$ to $s$
(see \cite{BS}, for example).  In particular, we have the following precise connection (cf. \cite[Corollary 4.2]{ESY1}), whose proof is deferred to Appendix 
\ref{section:ESY}. 

\begin{lemma}[Control of Stieltjes transform implies control on ESD]\label{lemma:S-transform}
Let $1/10 \geq \eta \geq 1/n$,  and $L, \eps, \delta > 0$.  Suppose that one has the bound
\begin{equation}\label{soda}
 |s_{n} (z) -s(z) | \le \delta
\end{equation}
with (uniformly) overwhelming probability for all $z$ with $|\Re(z)| \leq L$ and $\Im(z) \geq \eta$.  Then for any interval $I$ in $[-L+\eps,L-\eps]$ with $|I| \geq \max( 2\eta, \frac{\eta}{\delta} \log\frac{1}{\delta} )$, one has
$$ |N_I - n \int_I \rho_{sc}(x)\ dx| \ll_\eps \delta n |I|$$
with overwhelming probability.
\end{lemma}


In view of this lemma, it suffices to show that for each complex number $z$ with $\Re(z) \leq 2-\eps/2$ and $\Im(z) \geq \eta := \frac{K^2 \log^{19} n}{n}$, one has
$$ |s_n(z) - s(z)| \leq o(1)$$
with (uniformly) overwhelming probability.

Fix $z$ as above.  From \eqref{sqrt}, $s(z)$ is the unique solution to the equation
\begin{equation}\label{sz}
s(z) + \frac{1}{s(z)+z} = 0
\end{equation}
with $\Im s(z) > 0$. The strategy is then to obtain a similar equation for $s_n(z)$ (note that one automatically has $\Im(s_n(z)) > 0$).

By Lemma \ref{stielt} we may write
\begin{equation}\label{ssz}
 s_n(z) = \frac{1}{n} \sum_{k=1}^n \frac{1}{\frac{1}{\sqrt{n}} \zeta_{kk} - z - Y_k}
\end{equation}
where 
$$Y_k := a_k^* (W_{n,k} - zI)^{-1} a_k,$$
$W_{n,k}$ is the matrix $W_n$ with the $k^{th}$ row and column removed, and $a_k$ is the $k^{th}$ row of $W_n$ with the $k^{th}$ element removed.

The entries of $a_k$ are independent of each other and of $W_{n,k}$, and have mean zero and variance $\frac{1}{n}$.  By linearity of expectation we thus have, on conditioning on $W_{n,k}$
$$ \E(Y_k|W_{n,k}) = \frac{1}{n} \tr (W_{n,k}-zI)^{-1} = (1 - \frac{1}{n}) s_{n,k}(z)$$
where
$$ s_{n,k}(z) := \frac{1}{n-1} \sum_{i=1}^{n-1} \frac{1}{\lambda_i(W_{n,k})-z}$$
is the Stieltjes transform of $W_{n,k}$.  From the Cauchy interlacing law \eqref{cauchy-interlace} we have
$$ s_n(z) - (1-\frac{1}{n}) s_{n,k}(z) = O\left( \frac{1}{n} \int_\R \frac{1}{|x-z|^2}\ dx \right) = O\left( \frac{1}{n \eta} \right)$$
and thus
\begin{equation}\label{eyk}
 \E(Y_k|W_{n,k}) = s_n(z) + O\left( \frac{1}{K^2 \log^{19} n} \right).
\end{equation}

We now claim that a similar estimate holds for $Y_k$ itself:

\begin{proposition}[Concentration of $Y_k$]\label{yk-conc}  For each $1 \leq k \leq n$, one has $Y_k = s_n(z) + O(\frac{1}{\log n})$ with overwhelming probability.
\end{proposition}

Assume this proposition for the moment.   By hypothesis, $\frac{1}{\sqrt{n}} \zeta_{kk} \leq K/\sqrt{n} \leq n^{-\eps}$ almost surely.  Inserting these bounds into \eqref{ssz}, we see that
$$
s_n(z) + \frac{1}{n} \sum_{k=1}^n \frac{1}{s_n(z) + z + o(1)} = 0$$
with overwhelming probability (compare with \eqref{sz}).  This implies that with overwhelming probability either $s_n(z) = s(z) + o(1)$ or that $s_n(z) = -z + o(1)$.  On the other hand, as $\Im s_n(z)$ is necessarily positive, the second possibility can only occur when $\Im z = o(1)$.  A continuity argument (as in \cite{ESY2}) then shows that the second possibility cannot occur at all (note that $s(z)$ stays a fixed distance away from $-z$ for $z$ in a compact set) and the claim follows.

\subsection{A preliminary concentration bound}

It remains to prove Proposition \ref{yk-conc}.  We begin with a preliminary bound (cf. \cite[Theorem 5.1]{ESY3}):

\begin{proposition}\label{smallsc}  For all $I \subset \R$ with $|I| \geq \frac{K^2 \log^2 n}{n}$, one has
$$ N_I \ll n |I|$$
with overwhelming probability.
\end{proposition}

The proof, which follows the arguments from \cite{ESY3}, but using Lemma \ref{lemma:projection} to simplify things somewhat, is presented in Appendix 
\ref{section:ESY}.

Now we prove Proposition \ref{yk-conc}.  Fix $k$, and write $z=x+\sqrt{-1}\eta$. From \eqref{eyk} it suffices to show that
$$ Y_k - \E(Y_k|W_{n,k}) = O\left( \frac{1}{\log n} \right)$$
with overwhelming probability.  Decomposing $Y_k$ as in \eqref{yk}, it thus suffices to show that
\begin{equation}\label{rje}
\sum_{j=1}^{n-1} \frac{R_j}{\lambda_j(W_{n,k})-(x+\sqrt{-1}\eta)} = O\left( \frac{1}{\log n} \right)
\end{equation}
with overwhelming probability, where $R_j :=| u_j(W_{n,k})^* a_k|^2-1/n$.

Let $1 \leq i_- < i_+ \leq n$, then
$$ \sum_{i_- \leq j \leq i_+} R_j = \| P_H a_k \|^2 - \frac{\dim(H)}{n}$$
where $H$ is the space spanned by the $u_j(W_{n,k})^*$ for $i_- \leq j \leq i_+$.  From Lemma \ref{lemma:projection} and the union bound,
 we conclude that with overwhelming probability
 
\begin{equation}\label{rjb-0}
 |\sum_{i_- \leq j \leq i_+} R_j| \ll \frac{\sqrt{i_+ - i_-} K \log n + K^2 \log^2 n}{n}.
\end{equation}
By the triangle inequality, this implies that
$$
 \sum_{i_- \leq j \leq i_+} \| P_H a_k \|^2  \ll \frac{i_+-i_-}{n} + \frac{\sqrt{i_+ - i_-} K \log n + K^2 \log^2 n}{n}$$
 and hence by a further application of the triangle inequality
\begin{equation}\label{rjb}
\sum_{i_- \leq j \leq i_+} |R_j| \ll \frac{(i_+ - i_-) + K^2 \log^2 n}{n}
\end{equation}
with overwhelming probability.  

Since $\eta \geq K^2 \log^{19} n/n$, the bound \eqref{rjb-0} (together with Proposition \ref{smallsc}) already lets one dispose of the contribution to \eqref{rje} where $|\lambda_j(W_{n,k}) - x| \leq K^2 \log^{10} n/n$.  For the remaining contributions, we subdivide into $O(\log^3 n)$ intervals $\{ j: i_- \leq j \leq i_+\}$ such that in each interval  $a \leq |\lambda_j(W_{n,k}) - x| \leq (1 + \frac{1}{\log^2 n}) a$ for some $a \geq K^2 \log^{10} n/n$
(the value of $a$ varies from interval to interval).  For each such interval, the function $\frac{1}{\lambda_j(W_{n,k})-(x+\sqrt{-1}\eta)}$ has magnitude $O(\frac{1}{a})$ and fluctuates by at most $O(\frac{1}{a \log^2 n})$ as $j$ ranges over the interval.  From \eqref{rjb-0}, \eqref{rjb} we conclude
$$
\left|\sum_{i_- \leq j \leq i_+} \frac{R_j}{\lambda_j(W_{n,k})-(x+\sqrt{-1}\eta)}\right|
\ll
\frac{1}{an} ( \sqrt{i_+ - i_-} K \log n + K^2 \log^2 n ) + \frac{i_{+}- i_{-}}{an\log^2 n} $$
with overwhelming probability. By Proposition \ref{smallsc}, $i_{+} -i_{-} \ll an $ with overwhelming probability.  Thus we have 
$$
\left|\sum_{i_- \leq j \leq i_+} \frac{R_j}{\lambda_j(W_{n,k})-(x+\sqrt{-1}\eta)}\right|
\ll \frac{K \log n}{ \sqrt{an} } + \frac{1}{\log^4 n}$$

with overwhelming probability. Summing over the values of $a$ (taking into account the lower bound $a$) we obtain \eqref{rje} as desired. 

\section{Propagation of narrow spectral gaps}\label{backprop-sec}

We now prove Lemma \ref{backprop}. Fix $i_0, l, n$.  Assume for contradiction that all of the conclusions fail.  We will always assume that $n_0$ (and hence $n$) is sufficiently large.

By \eqref{giln}, we can find $1 \leq i_- \leq i_0-l < i_0 \leq i_+ \leq n+1$ such that
$$ \lambda_{i_+}(A_{n+1})-\lambda_{i_-}(A_{n+1}) = g_{i_0,l,n+1} \min( i_+-i_-, \log^{C_1} n_0 )^{\log^{0.9} n_0}.$$
If $i_+ - i_- \geq \log^{C_1/2} n$, then conclusion (i) holds (for $n$ large enough), so we may assume that
\begin{equation}\label{ipm}
i_+ - i_- < \log^{C_1/2} n.
\end{equation}
We set
\begin{equation}\label{ldef}
 L := \lambda_{i_+}(A_{n+1})-\lambda_{i_-}(A_{n+1}) = g_{i_0,l,n+1} (i_+ - i_-)^{\log^{0.9} N}.
\end{equation}
In particular (by \eqref{gdel}, \eqref{ipm}) we have
\begin{equation}\label{ldel}
 L \leq \delta \exp( \log^{0.91} n ).
\end{equation}
We now study the eigenvalue equation \eqref{jn} for $i=i_-$, which we rearrange as
$$
 \sum_{i_- \leq j \leq n} \frac{|u_j(A_n)^* X_n|^2}{\lambda_j(A_{n}) - \lambda_{i_-}(A_{n+1})} =
 \sum_{1 \leq j < i_-} \frac{|u_j(A_n)^* X_n|^2}{\lambda_{i_-}(A_{n+1}) - \lambda_j(A_{n})}
+  a_{n+1,n+1} - \lambda_{i_-}(A_{n+1}).$$
Observe that
$$
\sum_{i_- \leq j \leq n} \frac{|u_j(A_n)^* X_n|^2}{\lambda_j(A_{n}) - \lambda_{i_-}(A_{n+1})}
\geq \frac{1}{L} \sum_{i_- \leq j < i_+} |u_j(A_n)^* X_n|^2.$$
Since conclusion (ii) fails, we have
$$
\sum_{i_- \leq j \leq n} \frac{|u_j(A_n)^* X_n|^2}{\lambda_j(A_{n}) - \lambda_{i_-}(A_{n+1})} \geq \frac{n(i_+-i_-)}{2^{m/2} L \log^{0.01} n}.
$$
On the other hand, since conclusions (iii), (iv) fail, we have
$$ |a_{n+1,n+1} - \lambda_{i_-}(A_{n+1})| \leq \frac{n \exp( -\log^{0.95} n )}{\delta^{1/2}} \leq \frac{n(i_+-i_-)}{2^{m/2+1} L\log^{0.01} n}$$
thanks to the bounds $i_+-i_- \geq 1$, \eqref{mcivil} and \eqref{ldel}.
By the triangle inequality, we thus have
$$
\sum_{1 \leq j < i_-} \frac{|u_j(A_n)^* X_n|^2}{\lambda_{i_-}(A_{n+1}) - \lambda_j(A_{n})} \geq \frac{n(i_+-i_-)}{2^{m/2+1} L \log^{0.01} n}.$$
Note that all the summands on the left-hand side are non-negative.  By 
a dyadic partition and the pigeonhole principle (using the convergence of the series $1/l^2$), we can thus find $k \geq 1$ such that
\begin{equation}\label{sami}
\sum_{1 \leq j < i_-: 2^{k-1} \leq i_- - j < 2^k} \frac{|u_j(A_n)^* X_n|^2}{\lambda_{i_-}(A_{n+1}) - \lambda_j(A_{n})} \gg \frac{n(i_+-i_-)}{2^{m/2} L k^2 \log^{0.01} n}.
\end{equation}
In particular $2^{k-1} < i_-$.

Let's first suppose that $2^{k-1} \geq \log^{C_1/2} n$.  Then by the failure of conclusion (i), we have
$$ \lambda_{i_-}(A_{n+1}) - \lambda_j(A_{n}) > \delta^{1/4} \exp( \log^{0.95} n ) 2^{k-1}$$
for all $j$ in the summation in \eqref{sami}, and thus (by \eqref{mcivil}, \eqref{ldel} and the trivial bounds $i_+-i_- \geq 1$ and $k = O(\log n)$)
\begin{equation}\label{jo}
\begin{split}
\sum_{1 \leq j < i_-: 2^{k-1} \leq i_- - j < 2^k}  |u_j(A_n)^* X_n|^2 &\gg 
\frac{n(i_+-i_-)}{2^{m/2} L k^2 \log^{0.01}} \delta^{1/4} \exp( \log^{0.95} n ) 2^{k-1} \\
&\gg \frac{n 2^k}{\delta^{1/2}}.
\end{split}
\end{equation}
On the other hand, from the failure of conclusion (v),  we have
\begin{equation}\label{ujx}
 |u_j(A_n)^* X_n|^2 < \frac{n \exp( - \log^{0.96} n )}{\delta^{1/2}}
\end{equation}
for $\eps n/10 \leq j \leq (1-\eps/10) n$.   This already contradicts \eqref{jo} when the range of summation in \eqref{jo} is contained in the bulk region $\eps n/10 \leq j \leq (1-\eps/10) n$.  The only remaining case is when \eqref{jo} approaches the edge, which only occurs when $2^k \gg \eps n$.  But in this case we note from Pythagoras' theorem and the failure of conclusion (vi) that
$$ \sum_{j=1}^n |u_j(A_n)^* X_n|^2 < \frac{n^2 \exp( - \log^{0.96} n )}{\delta^{1/2}},$$
leading again to a contradiction with \eqref{jo}.  We may therefore assume that $2^{k-1} < \log^{C_1/2} n$, thus $k = O( C_1 \log \log n )$.

By the failure of conclusion (vii), we now have $|u_j(A_n)^* X_n|^2 \ll 2^{m/2} n \log^{0.8} n$ for all $j$ in the summation in \eqref{sami}; we conclude that
$$
\sum_{1 \leq j < i_-: 2^{k-1} \leq i_- - j < 2^k} \frac{1}{\lambda_{i_-}(A_{n+1}) - \lambda_j(A_{n})} \gg \frac{i_+-i_-}{2^m L \log^{0.82} n}.
$$
If we set $i_{--} := i_- - 2^{k-1}$, we conclude that $0 < i_- - i_{--} < \log^{C_1/2} n$ and
$$ \lambda_{i_{-}}(A_{n+1}) - \lambda_{i_{--}}(A_n) \leq 2^m \frac{i_- - i_{--}}{i_+-i_-} L \log^{0.83} n .$$
An analogous argument, starting with $i=i_+$ in \eqref{jn} instead of $i=i_-$ and reflecting all the indices, allows us to find $i_{++}$ with $0 \leq i_{++}-i_+ < \log^{C_1/2} n$ such that
$$ \lambda_{i_{++}}(A_n) - \lambda_{i_+}(A_{n+1}) \leq 2^m \frac{i_- - i_{++}}{i_+-i_-} L \log^{0.83} n.$$
Summing, we have
\begin{equation}\label{yom}
 \lambda_{i_{++}}(A_n) - \lambda_{i_{--}}(A_n) \leq L( 1 + 2^m \alpha \log^{0.84} n )
\end{equation}
where $\alpha := \frac{i_{++}-i_{--}}{i_+-i_-}-1$.  Note that $(1+\alpha)(i_+-i_-) = i_{++}-i_{--} \leq \log^{C_1} N$, and so by \eqref{giln}, \eqref{gilp}
$$
\frac{\lambda_{i_{++}}(A_n) - \lambda_{i_{--}}(A_n)}{ (1+\alpha)^{\log^{0.9} N} (i_+-i_-)^{\log^{0.9} N}} \geq 2^m g_{i_0,l,n+1}.$$
Combining this with \eqref{yom}, \eqref{ldef} we conclude that
$$ 1 + 2^m \alpha \log^{0.84} n \geq 2^m (1+\alpha)^{\log^{0.9} N}$$
and hence
$$ 1 + \alpha \log^{0.84} n \geq (1+\alpha)^{\log^{0.9} N}$$
But this contradicts the elementary estimate $(1+\alpha)^x \geq 1+x\alpha$ for $\alpha >0$ and $x \geq 1$, and Lemma \ref{backprop} follows.

\section{Bad events are rare}\label{bad-event-sec}

We now prove Proposition \ref{bad-event}.  Let the notation and assumptions be as in that proposition.

We first prove (a).  The truncation assumption \eqref{supi} ensures that the events (iii), (v), (vi) from Proposition \ref{backprop} are empty for $n$ large enough.  The event (i) fails with overwhelming probability thanks to Theorem \ref{sdb}.  The event (iv) fails with overwhelming probability because of the well-known fact that the operator norm of $A_n$ is $O(n)$ with overwhelming probability (see e.g. \cite{AGZ}; there are many proofs, for instance one can start by observing that $\|A_n\|_{op} \leq 2 \sup_x \|A_n x\|$, where $x$ ranges over a $1/2$-net of the unit ball, and use the union bound followed by a standard concentration of measure result, such as the Chernoff inequality).  This concludes the proof of (a).

Now we prove (b) and (c) jointly.  By \eqref{supi} and Proposition \ref{deloc} we can find $C'$ such that all the coefficients of the eigenvectors $u_j(A_n)$ for $\eps n/2 \leq j \leq (1-\eps/2) n$ are of magnitude at most $n^{-1/2} \log^{C'} n$ with overwhelming probability.  

Let us first consider (vii), in which we will be able to obtain the better upper bound of $2^{-\kappa m} 2^{-2C_1 n}$ for the conditional probability of occurrence (thus establishing (b) and (c) simultaneously for (vii)).  If $2^m \geq \log^{C_3} n$ for some sufficiently large $C_3$, then the desired bound comes from \eqref{supi} and Lemma \ref{lemma:projection}. (In fact, the Chernoff bound would suffice as well, and the event fails with overwhelming probability.)  Now suppose instead that $2^m \leq \log^{O(1)} n$.  We wish to show that
\begin{equation}\label{sim}
 \P( |S_i| \geq 2^{m/2} \log^{0.8} n ) \leq 2^{-\kappa m} \log^{-2C_1} n,
\end{equation}
where $S_i \in \C$ is the random walk
\begin{equation}\label{si}
 S_i := \zeta_{1,n+1} w_{i,1} + \ldots + \zeta_{n,n+1} w_{i,n}
\end{equation}
and $w_{i,1},\ldots,w_{i,n}$ are the coefficients of $u_i(A_n)$, which by hypothesis have magnitude $O( n^{-1/2} \log^{C'} n)$ and square-sum to $1$.

Observe that $S_i$ has mean zero and variance $1$.  Applying Theorem  \ref{bes} and \eqref{supi}, we conclude that
$$ \P( |S_i| \geq t ) \ll \exp( - ct^2 ) + n^{-1/2} \log^{O(1)} n$$
for any $t \geq 1$ and some absolute constant $c>0$, which easily yields \eqref{sim} in the range $2^m \leq \log^{O(1)} n$.

The consideration of (ii) is similar. Write the left-hand side of \eqref{smallin} as $\|\pi_H(X_n)\|$, where $H$ is the span of the $u_j(A_n)$ for $i_- \leq j < i_+$.  Applying \eqref{supi} and Lemma \ref{lemma:projection} we obtain the claim when $i_+-i_- \geq \log^{C_3} n$ for sufficiently large $c_3$ (in fact (ii) now fails with overwhelming probability), so we may assume instead that $i_+-i_- \leq \log^{O(1)} n$.  In this case, the event (ii)
can now be expressed as
\begin{equation}\label{lemon}
 |\vec S| \leq \frac{(i_+-i_-)^{1/2}}{2^{m/4} \log^{0.005} n}
\end{equation}
where $\vec S \in \C^{i_+-i_-}$ is the random vector with components $S_j$ defined in \eqref{si}.

From the orthonormality of the $u_i(A_n)$, we see that $\vec S$ has mean zero and has covariance matrix equal to the identity.  Applying Theorem  \ref{bes} again, we see that
$$ \P( |\vec S| \leq t ) \ll O(t/(i_+-i_-)^{1/2})^{(i_+-i_-)/4} + n^{-1/2} t^{-3} \log^{O(1)} n$$
Applying this with $t := \frac{(i_+-i_-)^{1/2}}{2^{m/4} \log^{0.005} n}$ and using the fact that $i_+ - i_- \geq l \geq C_2$ by hypothesis, one concludes that \eqref{lemon} occurs with probability
$$ \ll O(2^{m/4} \log^{0.005} n)^{-C_2/4} + n^{-1/2} 2^{3m/4} \log^{O(1)} n$$
which proves the claim as long as $C_2$ is large and $2^m \leq n^{1/100}$.  But the case $2^m \geq n^{1/100}$ then follows by noting that the probability of the event \eqref{lemon} is non-increasing in $m$.  The proof of Proposition \ref{bad-event} is now complete.

\appendix

\section{Concentration of determinant} \label{section:determinant}

In view of the standard identity
$$ \int_{-2}^2 \log |y|  \rho_{sc}(y)\ dy = - \frac{1}{2}$$
(which can be verified for instance by applying contour integration to a branch cut of $(4-z^2)^{1/2} \log z$ around the slit $[-2,2]$; see also Remark \ref{remp} below) and Stirling's formula, it suffices to prove the latter claim.

Fix $z$; we allow implied constants to depend on $z$.  We of course have
\begin{align*}
\log |\det (M_n - \sqrt{n} zI)|  &= \sum_{j=1}^n \log |\lambda_j(M_n) - \sqrt{n} z| \\
&= \frac{1}{2} n \log n + \sum_{j=1}^n \log |\lambda_j(W_n-z)|
\end{align*}
so it will suffice to show that
$$ \left|\frac{1}{n} \sum_{j=1}^n \log |\lambda_j(W_n)-z| - \int_{-2}^2 \log |y-z| \rho_{sc}(y)\ dy\right| \leq n^{-c}$$
asymptotically almost surely for some $c>0$.  Making the change of variables $y = t(x)$, where $t$ is defined in \eqref{lt}, it suffices to show that
$$ \left|\frac{1}{n} \sum_{j=1}^n \log |\lambda_j(W_n)-z| - \int_0^1 \log |t(x)-z|\ dx\right| \leq n^{-c}$$
asymptotically almost surely for some $c>0$.  

From Theorem \ref{theorem:main2} (with $k=1$) we have $\inf_j |\lambda_j(W_n)-z| \geq n^{-2}$ (say) asymptotically almost surely (because the expected number of eigenvalues in the interval $[z-n^{-2},z+n^{-2}]$ is $o(1)$).  From this (and \eqref{bai}) we conclude that $\sup_j \log |\lambda_j(W_n)-z| = O( \log n )$ asymptotically almost surely.   Thus, the contribution of all $j$ with $\left|t\left(\frac{j}{n}\right) - \Re z\right| \leq n^{-\eps}$ will be negligible for any fixed $\eps > 0$, and it suffices to show that
$$ \left|\frac{1}{n} \sum_{1 \leq j \leq n: \left|\left(\frac{j}{n}\right) - \Re z\right| > n^{-\eps}} \log |\lambda_j(W_n)-z| 
- \int_0^1 \log |t(x)-z|\ dx\right| \leq n^{-c}$$
asymptotically almost surely.

By \eqref{lambdaj-t} we see that with probability $1-o(1)$, one has $\lambda_j(W_n) = t\left(\frac{j}{n}\right) + O(n^{-\delta})$ for all $1 \leq j \leq n$ and some absolute constant $\delta > 0$, where $-2 \leq t(a) \leq 2$ is defined by \eqref{lt}. By Taylor expansion, we thus have asymptotically almost surely that
$$ \log |\lambda_j(W_n)-z| = \log \left|t\left(\frac{j}{n}\right)-z\right| + O(n^{-\delta/2})$$
(say) for all $1 \leq j \leq n$ with $\left|t\left(\frac{j}{n}\right) - \Re z\right| > n^{1-\eps}$, if $\eps$ is chosen sufficiently small depending on $\delta$.  The claim then follows (for $\eps$ small enough) by approximating $\int_0^1 \log |t(x)-z|\ dx$ by its Riemann integral away from the possible singularity at $\Re z$.

\begin{remark}\label{remp} The logarithmic potential $\int_{-2}^2 \log |y-z| \rho_{sc}(y)\ dy$ for the semicircular distribution can be computed explicitly as
$$ \int_{-2}^2 \log |y-z| \rho_{sc}(y)\ dy = \frac{1}{2} \Re \frac{z-\sqrt{z^2-4}}{z+\sqrt{z^2-4}} + \log \left| \frac{\sqrt{z^2-4}+z}{2}\right|$$
where $\sqrt{z^2-4}$ is the branch of the square root of $z^2-4$ with cut at $[-2,2]$ which is asymptotic to $z$ at infinity; this can be seen by integrating the formula
\begin{equation}\label{sqrt}
 \int_{-2}^2 \frac{1}{y-z} \rho_{sc}(y)\ dy = \frac{1}{2} (-z + \sqrt{z^2-4})
\end{equation}
for the Stieltjes formula for the semicircular potential, which can be easily verified by the Cauchy integral formula.
\end{remark}

\section{The distance between a random vector and a subspace} \label{section:projection}

The prupose of this appendix is to prove Lemma \ref{lemma:projection}. We restate this lemma for the reader's convenience. 

\begin{lemma}\label{lemma:projection-again}   Let $X=(\xi_1, \dots, \xi_n) \in \C^n$ be a
random vector whose entries are independent with mean zero, variance $1$, 
and are bounded in magnitude by $K$ almost surely for some $K $, where $K \ge 10( \E |\xi|^{4} +1)$. Let $H$ be a subspace of dimension $d$ and
$\pi_H$ the orthogonal projection onto $H$. Then
$$\P (|\|\pi_H (X)\| - \sqrt d| \ge t) \le 10 \exp(-
\frac{t^{2}}{10K^2}). $$
In particular, one has
$$ \| \pi_H(X)\| = \sqrt{d} + O( K \log n )$$
with overwhelming probability.
\end{lemma}

 It is easy to show that
$\E \|\pi_{H} (X)\|^{2} = d$, so it is indeed natural to expect that with high probability
$\pi_{H} (X)$ is around $\sqrt d$.

In a previous paper \cite{TVdet}, the authors proved Lemma \ref{lemma:projection-again}
for the special case when $\xi_{i} $ are Bernoulli
random variables (taking value $\pm 1$ with probability half).
This proof is a simple generalization of one in \cite{TVdet} (see also \cite[Appendix E]{TVhard}). We use the following theorem, which is a consequence of Talagrand's inequality
(see \cite[Theorem E.2]{TVhard} or \cite{Le, Tal}).

\begin{theorem}[Talagrand's inequality]\label{theorem:T}
Let $\D$ be the unit disk  $\{z\in \C, |z| \le 1 \}$. For every
product probability $\mu$ on $\D^n$, every convex $1$-Lipschitz
function $F: \C^n \to \R$, and every $r \ge 0$,
$$\mu (|F- M(F)| \ge r) \le 4 \exp(-r^2/16), $$
where $M(F)$ denotes the median of $F$.
\end{theorem}

\begin{remark} In fact, the results still holds for the space
$\D_1 \times \dots \times \D_n$ where $\D_i$ are complex region
with diameter $2$.
\end{remark}
 
An easy change of variables reveals the following generalization
of this inequality: if $\mu$ is supported on a dilate $K \cdot \D^n$ of
the unit disk for some $K > 0$, rather than $\D^n$ itself, then for every $r > 0$ we have
\begin{equation}\label{mood}
\mu (|F- M(F)| \ge r) \le 4 \exp(-r^2/16K^2).
\end{equation}
In what follows, we assume that $K \ge g(n)$, where $g(n)$ is tending (arbitrarily slowly) to infinity with $n$.
 The map $X \rightarrow |\pi_H (X)|$ is clearly convex and $1$-Lipschitz.  Applying \eqref{mood} we conclude that
\begin{equation}\label{po}
 \P( \Big| |\pi_H (X)| - M( |\pi_H (X) )| \Big| \geq t ) \le 4\exp( -t^2 / 16 K^2 ) ).
\end{equation}
for any $t > 0$.  To conclude the proof, it suffices to show that
\begin{equation}  \Big| M(|\pi_H(X)|  -\sqrt d \Big| \le 2K \end{equation}

Let $P = (p_{jk})_{1 \leq
j,k \leq n}$ be the $n \times n$ orthogonal projection matrix onto $H$.
We have that $\tr P^2= \tr P= \sum_i p_{ii} = d$ and $|p_{ii}| \le 1$.  Furthermore,
$$
|\pi_H (X)| ^2
=  \sum_{1 \le i,i \le n}  p_{ij} \xi_i \overline \xi_j
=\sum_{i=1}^n p_{ii} |\xi_i|^2 + \sum_{1\le i \neq j \le n} p_{ij} \xi_i \overline \xi_j. $$
 Consider the event $\CE_{+}$ that $|\pi_H(X)| \ge \sqrt d + 2K$. Since this implies
 $|\pi_H(X)|^2 \ge d + 4 \sqrt d K + 4K^2$, we have
 
 $$\P( \CE_{+} ) \le \P( \sum_{i=1}^n p_{ii} |\xi_{i}|^2 \ge d + 2 \sqrt d K ) +
 \P (|\sum_{1\le i \neq j \le n} p_{ij} \xi_i \overline \xi_j | \ge 2 \sqrt d K). $$

 Let $S_1:= \sum_{i=1}^n p_{ii} (|\xi_i|^2-1 )$, then we have by Chebyshev's  inequality
 
$$  \P( \sum_{i=1}^n p_{ii} |\xi_{i}|^2 \ge d + 2 \sqrt d K ) \le  \P(|S_1| \ge 2 \sqrt d K)
\le \frac{\E (|S_1|^2 ) } {4 d K^2} . $$

On the other hand, by the assumption on $K$

$$\E |S_1|^2 = \sum_{i=1}^n  p_{ii}^2 \E (|\xi_i|^2-1 )^2  = \sum_i p_{ii}^2 (\E |\xi_i|^4-1 )
\le \sum_{i=1}^n p_{ii}^2 K =  dK. $$

Thus,
$$ \P(|S_1| \ge 2 \sqrt d K)
\le \frac{\E (|S_1|^2 ) } {4 d K^2}  \le \frac{1}{K} ) \le 1/10. $$

Similarly, set $S_2:= |\sum_{i \neq j} p_{ij} \xi_i \overline {\xi} |.$ We have
$\E S_2^2 = \sum_{i \neq j} |p_{ij}|^2  \le d . $ So again  by Chebyshev's inequality
$$\P( S_2 \ge 2\sqrt d K) \le \frac{d}{ 4 d K^2}  \le 1/10. $$
 
It follows that $\P(\CE_{+} ) \le 1/5$ and so $M (\|\pi_H (X)\| ) \le \sqrt d + 2K$.
To prove the lower bound, let   $\CE_{-}$ be the event that $\|\pi_H(X)\| \le \sqrt d - 2K$
and notice that
$$\P(\CE_{-} ) \le \P(|\pi_H (X)|^2 \le d -2 \sqrt d K) \le \P(S_1 \le d- \sqrt d K)+ \P(S_2 \ge \sqrt d K). $$
Both terms on the RHS can be bounded by $1/5$ by the same argument as above. The proof is complete.

\section{Controlling the spectral density by the Stieltjes transform} \label{section:ESY}

In this appendix we establish Lemma \ref{lemma:S-transform} and Proposition \ref{smallsc}.

\begin{proof}  (Proof of Lemma \ref{lemma:S-transform}.) 
From \eqref{soda} we see that with overwhelming probability, one has
$$ |s_n(x + \sqrt{-1} \eta)| \ll 1$$
for all $-L \leq x \leq L$ which are multiples of $n^{-100}$ (say).  From \eqref{imsn} one concludes that
$$ \frac{1}{n} \sum_{i=1}^n \frac{\eta}{\eta^2 + |\lambda_i(W_n)-x|^2} \ll 1$$
and we thus have the crude bound
$$ N_I \ll \eta n$$
whenever $I \subset [-L, L]$ is an interval of length $|I| = \eta$.  Summing in $I$, we thus obtain the bound
\begin{equation}\label{ni}
 N_I \ll |I| n
\end{equation}
with overwhelming probability whenever $I \subset [-L, L]$ has length $|I| \geq \eta$.  (One could also invoke Proposition \ref{smallsc}  for this step.)

Next, let $I \subset [-L+\eps,L-\eps]$ be such that $|I| \geq 2\eta$, and consider the function
$$ F(y) := n^{-100} \sum_{x \in I; n^{-100}|x} \frac{\eta}{\pi(\eta^2 + |y-x|^2)}$$
where the sum ranges over all $x \in I$ that are multiples of $n^{-100}$.  Observe that
$$ \frac{1}{n} \sum_{i=1}^n F(\lambda_i(W_n)) = n^{-100} \frac{1}{\pi} \Im \sum_{x \in I; n^{-100}|x} s_n(x+\sqrt{-1}\eta)$$
and
$$ \int_\R F(y) \rho_{sc}(y)\ dy = n^{-100} \frac{1}{\pi} \Im \sum_{x \in I; n^{-100}|x} s(x+\sqrt{-1}\eta).$$
With overwhelming probability, we have $s_n(x+\sqrt{-1}\eta) = s(x+\sqrt{-1}\eta) + O(\delta)$ for all $x$ in the sum by hypothesis, and hence
$$ \frac{1}{n} \sum_{i=1}^n F(\lambda_i(W_n)) = \int_\R F(y) \rho_{sc}(y)\ dy + O( |I| \delta ).$$
On the other hand, from Riemann integration one sees that
$$ F(y) = \int_I \frac{\eta}{\pi(\eta^2 + |y-x|^2)}\ dx + O(n^{-10})$$
(say).    One can then establish the pointwise bounds
$$ F(y) \ll \frac{1}{1 + (\dist(y,I)/\eta)} + n^{-10}$$
when $y \not \in I$ and $\dist(y,I) \leq |I|$,
$$ F(y) \ll \frac{\eta |I|}{\dist(y,I)^2} + n^{-10}$$
when $y \not \in I$ and $\dist(y,I) > |I|$, and (since $\frac{\eta}{\pi(\eta^2 + |y-x|^2)}$ has integral $1$) in the remaining case 
$$ F(y) = 1 + O\left(\frac{1}{1 + (\dist(y,I^c)/\eta)}\right) + O(n^{-10}).$$
Using these bounds one sees that
$$ \int_\R F(y) \rho_{sc}(y)\ dy = \int_I \rho_{sc}(y)\ dy + O\left( \eta \log \frac{|I|}{\eta} \right)$$
and a similar argument using Riemann integration and \eqref{ni} (as well as the trivial bound $N_J \leq n$ when $J$ lies outside $[-L,L]$) gives
$$ \frac{1}{n} \sum_{i=1}^n F(\lambda_i(W_n)) = \frac{1}{n} N_I + O_\eps\left( \eta \log \frac{|I|}{\eta} \right).$$
Putting all this together, we conclude that
$$ N_I = n \int_I \rho_{sc}(y)\ dy  + O_\eps( \delta n |I| ) + O_\eps\left( \eta n \log \frac{|I|}{\eta} \right).$$
The latter error term can be absorbed into the former since $|I| \geq \frac{\eta}{\delta} \log \frac{1}{\delta}$, and the claim follows.
\end{proof}

\begin{proof}  (Proof of Lemma \ref{smallsc})
By the union bound it suffices to show this for $|I|=\eta:= \frac{K^2 \log^2 n}{n}$.  Let $x$ be the center of $I$, then by \eqref{imsn} it suffices to show that the event that
\begin{equation}\label{nin}
 N_I \geq C n \eta
\end{equation}
and
\begin{equation}\label{nan}
\Im s_n( x + \sqrt{-1} \eta ) \geq C
\end{equation}
for some large absolute constant $C$, fails with overwhelming probability.

Suppose that we have both \eqref{nin} and \eqref{nan}.  By \eqref{ssz} we have
$$ \frac{1}{n} \sum_{k=1}^n |\Im \frac{1}{\frac{1}{\sqrt{n}} \zeta_{kk} - (x+\sqrt{-1} \eta) - Y_k}| \geq C;$$
using the crude bound $|\Im \frac{1}{z}| \leq \frac{1}{|\Im z|}$, we conclude
$$ \frac{1}{n} \sum_{k=1}^n \frac{1}{|\eta + \Im Y_k|} \geq C;$$
On the other hand, by writing $W_{n,k}$ in terms of an orthonormal basis $u_j(W_{n,k})$ of eigenfunctions, one sees that
\begin{equation}\label{yk}
Y_k = \sum_{j=1}^{n-1} \frac{|u_j(W_{n,k})^* a_k|^2}{\lambda_j(W_{n,k})-(x+\sqrt{-1}\eta)}
\end{equation}
and hence
$$ \Im Y_k \geq \eta \sum_{j=1}^{n-1} \frac{|u_j(W_{n,k})^* a_k|^2}{\eta^2 + (\lambda_j(W_{n,k})-x)^2}.$$
On the other hand, from \eqref{nin} we can find $1 \leq i_- < i_+ \leq n$ with $i_+ - i_- \geq \eta n$ such that $\lambda_i(W_n) \in I$ for all $i_- \leq i \leq i_+$; by the Cauchy interlacing property \eqref{cauchy-interlace} we thus have $\lambda_i(W_{n,k}) \in I$ for $i_- \leq i < i_+$.  We conclude that
$$  \Im Y_k \gg \frac{1}{\eta} \sum_{i_- \leq j < i_+} |u_j(W_{n,k})^* a_k|^2 = \frac{1}{\eta} \| P_{H_k} a_k \|^2$$
where $P_{H_k}$ is the orthogonal projection to the $i_+-i_-$-dimensional space $H_k$ spanned by the eigenvectors $u_j(W_{n,k})$ for $i_- \leq j < i_+$.  Putting all this together, we conclude that
$$ \frac{1}{n} \sum_{k=1}^n \frac{\eta}{\eta^2 + \| P_{H_k} a_k \|^2} \gg C.$$
On the other hand, from Lemma \ref{lemma:projection-again} we see that $\| P_{H_k} a_k \|^2 = O( \eta )$ with overwhelming probability.  (One has to take the union bound over all possible choices of $i_-, i_+$, but there are only $O(n^2)$ such choices at most, so this is not a problem.)  The claim then follows by taking $C$ sufficiently large.
\end{proof}

\section{A multidimensional Berry-Ess\'een theorem} \label{section:BE}

In this section, we prove Theorem \ref{bes}.
We will need the following multidimensional Berry-Ess\'een theorem, which is a generalisation of \cite[Proposition D.2]{TVhard}.

\begin{theorem}\label{berry-esseen}  Let $N, n \geq 1$ be integers, let $v_1,\ldots,v_n \in \C^N$ be vectors, let $\zeta_1,\ldots,\zeta_n$ be independent complex-valued variables with mean zero, variance $\E |\zeta_j|^2 = 1$, and the third moment bound
\begin{equation}\label{xi3}
\sup_{1 \leq i \leq n} \E |\zeta_i|^3 \leq C
\end{equation}
for some constant $C \geq 1$.  Let $S$ be the $\C^N$-valued random variable
$$ S := \sum_{i=1}^n v_i \zeta_i.$$
We identify $\C^N$ with $\R^{2N}$ in the usual manner, and define the \emph{covariance matrix} $M$ of $S$ to be the unique symmetric $2N \times 2N$ real matrix such that
\begin{equation}\label{umu}
 u^* M u := \E |\Re u^* S|^2
\end{equation}
for all $u \in \C^N \equiv \R^{2N}$.  

Let $G$ be a gaussian random variable on $\R^{2N} \equiv \C^N$ with mean zero and with the same covariance matrix $M$ as $S$, thus
$$ u^* M u = \E |G \cdot u|^2 = \E |\Re u^* S|$$
for all $u \in \R^{2n} \equiv \C^n$ (where $u \cdot v = \Re v^* u$ denotes the dot product on $\R^{2N}$).
More explicitly, $G$ has the distribution function
$$ \frac{1}{(2\pi)^n (\det M)^{1/2}} \exp( - x^* M^{-1} x / 2 )\ dx_1 \ldots dx_{2N}$$
if $M$ is invertible, with an analogous limiting formula when $M$ is singular.  Then for any $\eps > 0$ and any measurable set $\Omega \subset \R^{2N} \equiv \C^N$, one has
\begin{equation}\label{som}
 \P( S \in \Omega ) \leq \P( G \in \Omega \cup \partial_\eps \Omega ) + O( C N^{3/2} \eps^{-3} \sum_{1 \leq j \leq n} |v_j|^3 )
\end{equation}
and similarly
\begin{equation}\label{som2}
 \P( S \in \Omega ) \geq \P( G \in \Omega \backslash \partial_\eps \Omega ) - O( C N^{3/2} \eps^{-3} \sum_{1 \leq j \leq n} |v_j|^3 )
\end{equation}
where 
$$\partial_\eps \Omega := \{ x \in \R^{2N}: \dist_{\infty}(x,\partial \Omega) \leq \eps \},$$
$\partial \Omega$ is the topological boundary of $\Omega$, and $\dist_\infty$ is the distance with respect to the $\ell^\infty$ metric on $\R^{2N}$.
\end{theorem}

\begin{remark} The main novelty here, compared with that in \cite[Proposition D.2]{TVhard}, is that the random variable $\zeta_j$ is not assumed to be $\C$-normalized (which means that the real and imaginary parts of $\zeta_j$ have covariance matrix equal to half the identity.)  For instance, some of the $\zeta_j$ could be purely real, or supported on some other line through the origin, such as the imaginary axis.
\end{remark}

\begin{proof} We obtain the  result  by repeating the proof of \cite[Proposition D.2]{TVhard} with some proper modification.  
For the readers convenience, we present all details. 

It suffices to prove \eqref{som}, as \eqref{som2} follows by replacing $\Omega$ with its complement.

Let $\psi: \R \to \R^+$ be a bump function supported on the unit ball $\{x \in \R: |x| \leq 1 \}$ of total mass $\int_\R \psi = 1$, let $\Psi_{\eps,N}: F^\R \to \R^+$ be the approximation to the identity
$$ \Psi_{\eps,\R}(x_1,\ldots,x_N) := \prod_{i=1}^\R \frac{1}{\eps} \psi( \frac{x_i}{\eps} ),$$
and let $f: \R^n \to \R^+$ be the convolution
\begin{equation}\label{third}
 f(x) = \int_{\R^N} \Psi_{\eps,N}(y) 1_\Omega( x-y)\ dy
 \end{equation}
where $1_\Omega$ is the indicator function of $\Omega$.  Observe that $f$ equals $1$ on $\Omega \backslash \partial_\eps \Omega$, vanishes outside of $\Omega \cup \partial_\eps \Omega$, and is smoothly varying between $0$ and $1$ on $\partial_\eps \Omega$.  Thus it will suffice to show that
$$ |\E( f( S ) ) - \E( f( G ) )| \ll C N^{3/2} \eps^{-3} (\sum_{1 \leq j \leq n} |v_j|^3).$$

We now use a Lindeberg replacement trick (cf. \cite{lindeberg, PR}). For each $1 \leq i \leq n$, let $g_i$ be a complex gaussian with mean zero and with the same covariance matrix as $\zeta_i$, thus
$$ \E \Re(g_i)^2 = \E \Re(\zeta_i)^2; \quad \E \Im(g_i)^2 = \E \Im(\zeta_i)^2; \quad \E \Re(g_i) \Im(g_i) = \E \Re(\zeta_i) \Im(\zeta_i).$$
In particular $g_i$ has mean zero and variance $1$.  We construct the $g_1,\ldots,g_n$ to be jointly independent.    Observe from \eqref{umu} that the random variable
$$ g_{1} v_1 + \ldots + g_{n} v_n \in \C^N$$
has mean zero and covariance matrix $M$, and thus has the same distribution as $G$.  Thus if we define the random variables
$$ S_j := \zeta_1 v_1 + \ldots + \zeta_j v_j + g_{j+1} v_j + \ldots + g_{n} v_n \in \C^N$$
we have the telescoping triangle inequality
\begin{equation}\label{fsag}
 |\E(f(S)) - \E( f(G) )| \leq \sum_{j=1}^{n} |\E f( S_{j} ) - \E f( S_{j-1} )|.
\end{equation}
For each $1 \leq j \leq n$, we may write
$$ S_j = S'_j + \zeta_j v_j; S_{j-1} = S'_j + g_{j} v_j$$
where
$$ S'_j := \zeta_1 v_1 + \ldots + \zeta_{j-1} v_{j-1} + g_{j+1} v_j + \ldots + g_{n} v_n.$$
By Taylor's theorem with remainder we thus have
\begin{equation}\label{fa}
 f(S_j) = P_{S'_j}( \Re \zeta_j, \Im \zeta_j ) + O( |\zeta_j|^3 \sup_{x \in \R^n} \sum_{k=0}^3 |(v_j \cdot \nabla)^k (\sqrt{-1} v_j \cdot \nabla)^{3-k} f(x)| )
 \end{equation}
and
\begin{equation}\label{fb}
 f(S_{j-1}) = P_{S'_j}( \Re g_j, \Im g_j ) + O( |g_j|^3 \sup_{x \in \R^n} \sum_{k=0}^3 |(v_j \cdot \nabla)^k (\sqrt{-1} v_j \cdot \nabla)^{3-k} f(x)| )
 \end{equation}
where $P_{S'_j}$ is some quadratic polynomial depending on $S'_j$, and $v_j$ and $\sqrt{-1} v_j$ are viewed as vectors in $\R^{2N}$.  A computation using \eqref{third} and the Leibniz rule reveals that all third partial derivatives of $f$ have magnitude $O( \eps^{-3} )$, and so by Cauchy-Schwarz we have
$$ \sum_{k=0}^3 |(v_j \cdot \nabla)^k (\sqrt{-1} v_j \cdot \nabla)^{3-k} f(x)| ) \ll |v_j|^3 N^{3/2} \eps^{-3}.$$
Observe that $\zeta_j$, $g_{j}$ are independent of $S'_j$, and have the same mean and covariance matrix.  Subtracting \eqref{fa} from \eqref{fb} and taking expectations using \eqref{xi3} we conclude that
$$ | \E( f(S_j) ) - \E( f( S_{j-1} ) )| \ll C |v_j|^3 N^{3/2} \eps^{-3} $$
and the claim follows from \eqref{fsag}.
\end{proof}


\begin{remark} The bounds here are not best possible, but are sufficient for our applications.
\end{remark}

Now we are ready to prove Theorem \ref{bes}. 

\begin{proof} (Proof of Theorem \ref{bes}) We first prove the upper tail bound on $S_i$.  Here, the main tool is the $N=1$ case of Theorem \ref{berry-esseen}.  The variance of $S_i$ is
\begin{equation}\label{psy}
 \E |S_i|^2 = \sum_{j=1}^n |a_{i,j}|^2 = 1
\end{equation}
since the rows of $A$ have unit size.  Thus, the $2 \times 2$ covariance matrix of $S_i$ is $O(1)$.  Let $G_i$ be a complex gaussian with mean zero and the same covariance matrix as $S_i$.  By Theorem \ref{berry-esseen}, we have
$$
 \P( |S_i| \geq t ) \leq \P( |G| \geq t - \sqrt{2} \eps ) + O( C \eps^{-3} \sum_{i=1}^N \sum_{j=1}^n|a_{i,j}|^3 )$$
for any $\eps > 0$.  Selecting $\eps := \frac{1}{10}$ (say), and using the fact that $G$ has variance $1$, we conclude that
$$
 \P( |S_i| \geq t ) \leq \exp( - c t^2 ) + O( C \sum_{i=1}^N \sum_{j=1}^n |a_{i,j}|^3 ),$$
and the claim follows from \eqref{summer} and \eqref{psy}.
 
Now we prove the lower tail bound on $\vec S$, using  Theorem \ref{berry-esseen} in full generality.  Observe that for any unit vector $u \in \C^N \equiv \R^{2N}$, one has
$$ \E |u^* S|^2 = \E \|u^* A\|^2 = 1$$
by the orthonormality of the rows of $A$.  Thus, by \eqref{umu}, the operator norm of the covariance matrix $M$ of $S$ has operator norm at most $1$.  On the other hand, we have
$$ \tr M = \E |S|^2 = \| A \|_F^2 = N.$$
Thus, the $2N$ eigenvalues of $M$ range between $0$ and $1$ and add up to $N$.  This implies that  at least $N/2$ of them are at least $1/4$, and so one can find an $\lfloor N/2\rfloor$-dimensional real subspace $V$ of $\R^{2N}$ such that $M$ is invariant on $V$ and has all eigenvalues at least $1/4$ on $V$.

Now let $G$ be a gaussian in $\R^{2N} \equiv \C^N$ with mean zero and covariance matrix $M$.  By Theorem \ref{berry-esseen}, we have
$$ \P( | \vec S| \leq t ) \leq \P( |G| \leq t + \sqrt{2N} \eps ) + O( C N^{3/2} \eps^{-3} \sum_{i=1}^N \sum_{j=1}^n |a_{i,j}|^3 )$$
for any $\eps > 0$. By \eqref{summer}, \eqref{psy} we have
$$ \sum_{i=1}^N \sum_{j=1}^n  |a_{i,j}|^{3} \le N \sigma.$$

Setting $\eps := t / \sqrt{2N}$, we conclude that
$$ \P( |\vec S| \leq t ) \leq \P( |G| \leq 2t ) + O( C N^4 t^{-3} \sigma ).$$
Let $G_V$ be the orthogonal projection of $G$ to $V$. Clearly
$$ \P( |G| \leq 2t ) \leq \P( |G_V| \leq 2t ).$$

The gaussian $G_V$ has mean zero and covariance matrix at least $\frac{1}{4} I_V$ (i.e. all eigenvalues are at least $1/4$).  By applying a linear transformation to reduce the covariance, we see that the quantity $\P( |G_V| \leq 2t)$ is maximized when the covariance matrix \emph{exactly} $\frac{1}{4} I_V$. Thus, 
 in any orthonormal basis of $G_V$, the $\lfloor N/2\rfloor$ $g_1,\ldots,g_{\lfloor N/2\rfloor}$ components of $G_V$ are independent real gaussians of variance $1/4$.  If $|G_V| \leq 2t$, then $g_1^2 + \ldots + g_{\lfloor N/2\rfloor}^2 \leq 4t^2$, and thus (by Markov's inequality) $g_i^2 \leq 8t^2/N$ for at least $\lfloor N/4\rfloor$ of the indices $i$.  The number of choices of these indices is at most $2^{\lfloor N/2\rfloor}$, and the events $g_i^2 \leq 2 t^2/N$ are independent and occur with probability $O( t/\sqrt{N} )$, so we conclude from the union bound that
$$\P( |G_V| \leq 2t ) \leq O( t/\sqrt{N} )^{\lfloor N/4\rfloor}$$
and the claim follows.
\end{proof}

\vskip2mm

{\it Acknowledgement.}  We would like to thank P. Sarnak for bringing this beautiful problem to our attention. 
We would like to thank M. Krishnapur and S. Chatterjee who introduced us to the Lindeberg method. 
We would like to thank  G. Ben Arous, M. Krishnapur, K. Johansson,  J. Lebowitz, S. P\'ech\'e,   B. Rider, A. Soshnikov 
and  O. Zeitouni for useful conversations regarding the state-of-the-art of GUE/ GOE  and Johansson matrices.  
Part of the paper was written while the second author was visiting Microsoft Research in Boston and he would like to thank J. Chayes and C. Borgs for their hospitality.   We thank Alain-Sol Sznitman and the anonymous referees for corrections.

\end{document}